\documentclass[letterpaper, oneside]{amsart}
\usepackage{layout}	

\usepackage[
]{geometry}

\usepackage[parfill]{parskip}   
\usepackage{amssymb}
\usepackage{amsthm}
\usepackage{amsmath}
\usepackage{enumitem}
\usepackage{mathtools}
\usepackage{hyperref}
\usepackage{colonequals}
\usepackage{mathrsfs}
\usepackage{color}
\usepackage{subcaption}
\usepackage{adjustbox}
\usepackage{ytableau}
\usepackage{bm}
\usepackage{bbm}
\usepackage{tensor}
\usepackage{esvect}
\usepackage{tabmac}

\usepackage{tikz}
\usetikzlibrary{cd}

\theoremstyle{plain}
\newtheorem{thm}{Theorem}[section]
\newtheorem*{thm*}{Theorem}
\newtheorem{lem}[thm]{Lemma}
\newtheorem{prop}[thm]{Proposition}

\theoremstyle{definition}
\newtheorem{defn}[thm]{Definition}
\newtheorem*{ack}{Acknowledgement}

\newtheorem{example}[thm]{Example}

\theoremstyle{remark}
\newtheorem*{rmk}{Remark}

\numberwithin{equation}{section}

\newcommand{\bC}{\mathbb{C}}

\newcommand{\bN}{\mathbb{N}}

\newcommand{\bZ}{\mathbb{Z}}

\newcommand{\ol}[1]{\overline{#1}}

\newcommand{\cB}{\mathcal{B}}

\newcommand{\cD}{\mathcal{D}}

\newcommand{\cJ}{\mathcal{J}}

\newcommand{\cO}{\mathcal{O}}
\newcommand{\cP}{\mathcal{P}}

\newcommand{\fM}{\mathfrak{M}}
\newcommand{\ft}{\mathfrak{t}}

\newcommand{\Sym}{\mathcal{S}}
\newcommand{\affSn}{\overline{\Sym_n}}
\newcommand{\extSn}{\widetilde{\Sym_n}}
\newcommand{\Odom}{\Omega_{\textup{dom}}}
\newcommand{\offset}{\vv{\mathbf{s}}}
\newcommand{\shift}{\bm{\omega}}
\newcommand{\tsc}{{\underline{c}}}
\newcommand{\lc}{\Gamma}
\newcommand{\lcc}{{\lc}^{\textup{can}}}
\newcommand{\lca}{{\lc}^{\textup{an}}}

\newcommand{\Tc}{{T}^{\textup{can}}}
\newcommand{\Ta}{{T}^{\textup{an}}}

\newcommand{\fw}{\textup{fw}}
\newcommand{\bk}{\textup{bk}}
\newcommand{\st}{\mathfrak{st}}

\newcommand{\RSYT}{\textup{RSYT}}

\DeclareMathOperator{\Rep}{\mathcal{R}}

\DeclareMathOperator{\Mat}{\textup{Mat}}

\DeclareMathOperator{\Dom}{\textup{Dom}}

\DeclareMathOperator{\Lie}{\textup{Lie}}

\DeclareMathOperator{\lch}{\textup{lch}}

\DeclareMathOperator{\rev}{\textup{rev}}

\newcommand{\br}[1]{\left\langle{#1}\right\rangle}

\newcommand{\floor}[1]{\left\lfloor{#1}\right\rfloor}
\newcommand{\ceil}[1]{\left\lceil{#1}\right\rceil}
\newcommand{\nco}[1]{\left[{#1}\right]_n}
\newcommand{\nbr}[1]{\left({#1}\right)}

\title[Asymptotic Hecke algebras and Lusztig-Vogan bijection via AMBC]{Asymptotic Hecke algebras and Lusztig-Vogan bijection via affine matrix-ball construction}
\author{Dongkwan Kim} 
\author{Pavlo Pylyavskyy}
\thanks{P. P. was partially supported by NSF grants DMS-1148634 and DMS-1351590.}
\date{\today}							

\begin{document}
\begin{abstract} The affine matrix-ball construction (abbreviated AMBC) was developed by Chmutov, Lewis, Pylyavskyy, and Yudovina as an affine generalization of the Robinson-Schensted correspondence. 
We show that AMBC gives a simple way to compute a distinguished involution in each Kazhdan-Lusztig cell of an affine symmetric group. We then use AMBC to give the first known canonical presentation for the asymptotic Hecke algebras of extended affine symmetric groups. As an application, we show that AMBC gives a conceptual way to compute the Lusztig-Vogan bijection. For the latter we build upon prior works of Achar and Rush.
\end{abstract}

\setcounter{tocdepth}{1}
\maketitle

\renewcommand\contentsname{}
\tableofcontents

\section{Introduction}
The Robinson-Schensted algorithm provides a bijection between the symmetric group and the set of pairs of standard Young tableaux of the same shape. It has become ubiquitous since its invention and now it appears in the theory of crystals, symmetric functions, representations of symmetric groups, Kazhdan-Lusztig cells, etc. Furthermore, it has been generalized in many ways, e.g. as the Robinson-Schensted-Knuth correspondence \cite{knu70}, domino insertion algorithm for hyperoctahedral groups \cite{gar90, gar92, gar93},  exotic Robinson-Schensted algorithm \cite{ht12}, to name a few.

Recently, Chmutov, Lewis, Pylyavskyy, and Yudovina \cite{cpy18}, \cite{clp17} introduced its affine generalization, called the affine matrix-ball construction (abbreviated AMBC), which extends the work of Shi \cite{shi91}. The affine matrix-ball construction provides a bijection $\Phi: w \mapsto (P(w), Q(w), \vv{\rho}(w))$ from the extended affine symmetric group to the set of triples $(P, Q, \vv{\rho})$ where $P$ and $Q$ are row-standard Young tableaux of the same shape and $\vv{\rho}$ is an integer vector satisfying certain properties. This algorithm also enjoys a lot of properties which the usual Robinson-Schensted algorithm possesses.

On the other hand, an (extended) affine symmetric group (or more generally an affine Weyl group) plays a central role in representation theory. It is strongly connected to the Springer theory (and its affine analogue), representations of Lie groups and Lie algebras, geometry of flag varieties, etc. (See e.g. \cite{hum91, cg97, lp07, yun16}.) In order to understand this group and its representation theory, it is desirable to find not only its characters but the parametrization and realizations of its (irreducible) representations.

To this end, Lusztig \cite{lus87:cell} introduced the \emph{asymptotic Hecke algebra}, conventionally denoted $\cJ$ (see \cite[Chapter 18]{lus14:hecke} for more details). It is a free abelian group with basis parametrized by elements of the corresponding affine Weyl group, say $\ft_w$, and the multiplication is given by $\ft_u\cdot\ft_v=\sum_{w}\gamma_{u,v,w^{-1}}\ft_w$ where the structure constants $\gamma_{u,v,w^{-1}}$ come from those of the affine Hecke algebra with respect to the canonical basis. Usually, it has a much simpler structure than the original affine Hecke algebra. For example, if we start with the extended affine symmetric group then the structure constants of such $\cJ$ can be obtained from Littlewood-Richardson coefficients of some (smaller) symmetric groups. Also, we will later observe that in this case $\cJ$ is a direct sum of certain matrix algebras. Nonetheless, it contains a lot of information of the representation theory of corresponding affine Hecke algebra. Thus, in order to study affine Hecke algebras it is often very useful to first understand the structure of the corresponding asymptotic version $\cJ$.

This paper develops such understanding of the structure of the asymptotic Hecke algebra $\cJ$ for $GL_n$ in terms of AMBC. Here we briefly explain our results. For each integer partition $\lambda \vdash n$, we define $F_\lambda$ to be $GL_{m_1} \times GL_{m_2} \times \cdots$ where $m_i$ is the number of $i$'s appearing in the parts of $\lambda$. Also, let $\tsc_\lambda$ be the two-sided Kazhdan-Lusztig cell in the extended affine symmetric group parametrized by $\lambda$ in the sense of \cite{lus89:cell}. Xi \cite{xi02} proved that for a certain left cell $\lca \subset \tsc_\lambda$, there exists a ring isomorphism between  $\cJ_{(\lca)^{-1} \cap \lca}$ and the representation ring $\Rep(F_\lambda)$ where the former is defined to be the restriction of $\cJ$ to the intersection of $\lca$ and its inverse. Our first main result is the following theorem.

\begin{thm*}[Theorem \ref{thm:xibij}] The ring isomorphism $\cJ_{(\lca)^{-1} \cap \lca} \xrightarrow{\simeq} \Rep(F_\lambda)$ defined by Xi coincides with the map $\ft_w\mapsto [V_{\vv{\rho}(w)}]$ where $\vv{\rho}(w)$ is the integer vector of the image $w \mapsto (P(w), Q(w), \vv{\rho}(w))$ under AMBC and $[V_{\vv{\rho}(w)}]$ is the class of an irreducible representation of $F_\lambda$ parametrized by $\vv{\rho}(w)$.
\end{thm*}

Originally, in his paper Xi used his bijection above to prove the conjecture of Lusztig \cite[Conjecture 10.5]{lus89:cell}. It states that $\cJ_{\tsc_\lambda}$, or the asymptotic Hecke algebra restricted to $\tsc_\lambda$, is isomorphic to a matrix algebra with base ring $\Rep(F_\lambda)$. He gave a description of such an isomorphism, but his construction was not canonical; one needs to pick an identification of each left cell with a fixed one, which relies on the choice of star operations and multiplication by a shift element that connect two such cells. In this paper, we describe such an isomorphism that naturally comes from AMBC, which is therefore more canonical; AMBC does not depend on the cell structure of (extended) affine symmetric group but relies only on their combinatorial properties. More precisely, we have the following theorem. Here, $\cD$ is the set of elements in the extended affine symmetric group which correspond to primitive idempotents in $\cJ$, called the distinguished involutions.

\begin{thm*}[Theorem \ref{thm:dist} and \ref{thm:matisom}] Consider a matrix algebra over $\Rep(F_\lambda)$, denoted $\fM$, whose rows (resp. columns) are parametrized by left cells (resp. right cells) in $\tsc_\lambda$. Then there exists an isomorphism of rings $\cJ_{\tsc_\lambda} \xrightarrow{\simeq} \fM$ which sendss $\ft_w$ to the matrix whose $(\Gamma_{P(w)}, (\Gamma_{Q(w)})^{-1})$-entry is $[V_{\vv{\rho}(w)}]$ and $0$ elsewhere. Here, $(P(w), Q(w), \vv{\rho}(w))$ is the image of $w$ under the AMBC map, and $\Gamma_{T}$ (resp.  $(\Gamma_{T})^{-1}$) is the left cell (resp. right cell) parametrized by $T$. Moreover, we have $w \in \cD$ if and only if $P(w) = Q(w)$ and $\vv{\rho}(w)=0$.
\end{thm*}

Meanwhile, the conjecture of Lusztig on the structure of the asymptotic Hecke algebra $\cJ$ provides a certain bijection between the set of irreducible representations of $GL_n$, denoted $\Dom(GL_n)$, and the union of $\Dom(F_\lambda)$ where $\lambda$ runs over all the partitions of $n$. This is now called the Lusztig-Vogan bijection, and is studied by Achar \cite{ach01, ach04}, Bezrukavnikov \cite{bez03, bez04, bez09}, Ostrik \cite{ost00}, and Rush \cite{rus17:thesis, rus17}. As AMBC gives a proof of Lusztig's conjecture, there also exists a new interpretation of the Lusztig-Vogan bijection in terms of AMBC. Our next result is that these bijections are all equal; we have

\begin{thm*}[Theorem \ref{thm:lvbij}] The Lusztig-Vogan bijection induced from AMBC (see the beginning of Section \ref{sec:equality} for the precise definition) is equal to that of Achar \cite{ach01}, Bezrukavnikov \cite{bez03, bez04}, Ostrik \cite{ost00}, and Rush \cite{rus17:thesis, rus17}.
\end{thm*}

This paper is organized as follows; in Section 2 we recall basic notations and definitions used in this paper; in Section 3 we introduce the notion of canonical and anti-canonical left cells; in Section 4 we prove the theorem stating that the bijection of Xi is compatible with the result of AMBC; in Section 5 we study how multiplication by a shift element and a star operation behave in terms of AMBC; in Section 6 we describe the image of distinguished involutions under AMBC; in Section 7 we provided an isomorphism between the asymptotic Hecke algebra attached to a two-sided cell and a certain matrix algebra using the AMBC map; in Section 8 we describe the Lusztig-Vogan bijection, explain some of its properties, and state that the bijection which is derived from AMBC is equal to the one previously studied by Achar, Bezrukavnikov, Ostrik, and Rush, whose proof is given in Section 9. In the appendix, we briefly discuss how our results can be applied to other affine Weyl groups of type $A$, especially those of $SL_n$ and $PGL_n$.

\begin{ack} We are grateful to Michael Chmutov and an anonymous referee for their helpful comments.
\end{ack}

\section{Backgrounds}
\subsection{Setup} \label{sec:setup}
For $a,b \in \bZ$, we set $[a,b]$ to be $\{x\in \bZ\mid a\leq x\leq b\}$. For $n \in \bZ_{>0}$, we define $[n]=\{\ol{1}, \ol{2}, \ldots, \ol{n}\}$ to be the set of residues modulo $n$.

For $n \in \bZ_{>0}$, we define the \emph{extended affine symmetric group} to be
\[\extSn\colonequals \{ w : \bZ \rightarrow \bZ \mid w \textup{ is bijective, } w(i+n)=w(i)+n \textnormal{ for all } i \in \bZ\}.
\]
We shall use the \emph{window notation} for elements of $\extSn$: $w = [w(1), w(2), \ldots, w(n)]$. An important distinguished element $\shift \in \extSn$ is the {\it {shift permutation}} $\shift = [2,3,\ldots, n, n+1]$. We shall often identify $w \in \extSn$ with the set $\cB_w \subset \bZ \times \bZ$ of all pairs $(i, w(i)), i \in \bZ$. We refer to elements of $\bZ \times \bZ$ that are of the form $(i, w(i))$ as \emph{balls} of $w$. If we set $\Sym_n$ to be the (finite) symmetric group permuting $[1,n]$, then there exists a canonical embedding $\Sym_n \rightarrow \extSn$ which sends $w$ to $[w(1), \ldots, w(n)]$.

A {\it {partition}} $\lambda =(\lambda_1, \lambda_2, \ldots)$ is a finite sequence of integers satisfying $\lambda_1 \geq \lambda_2 \geq \cdots>0.$ We denote by $\lambda'$ the {\it {transpose}} of $\lambda$, often also referred to as the {\it {conjugate partition}}. We define its size to be $n = \sum_i \lambda_i$ and also write $\lambda \vdash n$. A partition $\lambda=(\lambda_1, \lambda_2, \ldots) \vdash n$ can also be denoted $\lambda = (1^{m_1}2^{m_2}\cdots)$, where $m_i$ is the number of parts of size $i$ among the $\lambda_j$-s.

For $\lambda \vdash n$ we define a {\it {tabloid}} $T =(T_1, T_2, \ldots)$ of shape $\lambda$ to be a Young diagram of shape $\lambda$ filled with the elements $[n]$ where each residue appears exactly once. The order of elements in each row does not matter, but it is convenient to regard elements of $T_i$ as being ordered with respect to the linear ordering $\ol{1} < \cdots < \ol{n}$. In this case we say that $T$ is {\it {row-standard}}. For a partition $\lambda\vdash n$, define $\RSYT(\lambda)$ to be the set of row-standard Young tabloid of shape $\lambda$.

For $n \geq 1$, we define $\Dom(GL_n)=\Dom(GL_n(\bC))$ to be the set of integer vectors
$$\Dom(GL_n)\colonequals \{ (a_1, a_2, \ldots, a_n) \in \bZ^n \mid a_1\geq a_2\geq \cdots \geq a_n\}.$$
We naturally identify this with the set of dominant weights of $GL_n$. Also let $\textup{Rep}(GL_n)=\textup{Rep}(GL_n(\bC))$ be the category of finite dimensional rational representations of $GL_n$. Then for any $\vv{\mu} \in \Dom(GL_n)$, there is an irreducible representation $V(\vv{\mu}) \in \textup{Rep}(GL_n)$ of highest weight $\vv{\mu}$, and $\{V(\vv{\mu}) \in\textup{Rep}(GL_n) \mid \vv{\mu} \in \Dom(GL_n)\}$ is a complete collection of irreducible objects in $\textup{Rep}(GL_n)$ up to isomorphism. If we further define $\Rep(GL_n)$ to be the Grothendieck ring of $\textup{Rep}(GL_n)$ (see \cite[II.6]{wei13} for the definition of Grothendieck rings), then the classes of $V(\vv{\mu})$ form a $\bZ$-basis of $\Rep(GL_n)$.

Let $\cJ$ be the asymptotic Hecke algebra of $\extSn$ defined in \cite[Section 2]{lus87:cell}. It is a free $\bZ$-module with a basis $\{\ft_w \mid w\in \extSn\}$, and $(\cJ, \{\ft_w \mid w\in \extSn\})$ is a based ring in the sense of \cite[Section 1]{lus87:leading}. Also if we set $\cD \subset \extSn$ to be the set of distinguished involutions, then $\cD$ is finite and the unit of $\cJ$ is given by $\sum_{w\in \cD}\ft_w$.

For a subset $X \subset \extSn$, we define $\cJ_X \colonequals \bigoplus_{w \in X} \bZ \ft_w$ to be the free sub-$\bZ$-module of $\cJ$ generated by $\{\ft_w \mid w \in X\}$. Then for a two-sided cell $\tsc \subset \extSn$, $\cJ_{\tsc}$ is a two-sided ideal of $\cJ$. Also, $(\cJ_{\tsc},  \{\ft_w \mid w\in \tsc\})$ is a based ring and $\sum_{w\in \cD\cap \tsc}\ft_w$ is its unit. In other words, we have a decomposition of based rings $\cJ = \bigoplus_{\tsc \subset \extSn} \cJ_\tsc$
where the sum is over all the two-sided cells $\tsc$ of $\extSn$. Also if $\lc$ is a left cell in $\tsc$, then $\cJ_{ \lc}$ (resp. $\cJ_{\lc^{-1}}$,  $\cJ_{\lc^{-1} \cap \lc}$) is a left (resp. right, two-sided) ideal of $\cJ_{\tsc}$. It is known \cite{lus87:cell} that $\cJ_{\lc^{-1} \cap \lc}$ is a commutative ring with unit $\ft_w$ where $w$ is the unique element in $ \lc \cap \cD$.

Recall some facts about nilpotent orbits, we refer the reader to \cite{col93nilp} for an exposition. Nilpotent adjoint orbits of $\Lie GL_n$ are parametrized by partitions of $n$, where to each $\lambda \vdash n$ one associates a nilpotent orbit $\cO_{\lambda}$ such that the Jordan type of any element in $\cO_\lambda$ is $\lambda$. Also, there exists a canonical bijection between nilpotent adjoint orbits of $\Lie GL_n$ and two-sided cell in $\extSn$ defined in \cite{lus89:cell}. We denote the two-sided cell corresponding to $\cO_\lambda$ by $\tsc_\lambda$.

Let the conjugate partition of $\lambda$ be $\lambda'=(\lambda'_1, \lambda'_2, \ldots)$. We define 
\begin{gather*}
F_\lambda \colonequals GL_{m_1} \times GL_{m_2} \times \cdots \qquad \textup{ and } L_\lambda \colonequals GL_{\lambda_1'}\times GL_{\lambda_2'} \times \cdots.
\end{gather*}
Then $\Dom(F_\lambda) = \Dom(GL_{m_1}) \times \Dom(GL_{m_2}) \times \cdots$ is naturally a subset of $\bZ^{l(\lambda)}$. For $N \in \cO_\lambda$, we may identify $F_\lambda$ with the reductive part of $Z_{GL_n}(N)$, the stabilizer subgroup of $N$ in $GL_n=GL_n(\bC)$. Then it follows from \cite{xi02} that there exists an isomorphism of rings from $\cJ_{\lc^{-1} \cap \lc}$ to $\Rep(F_\lambda)$ under which each $\ft_w$ maps to an irreducible representation, say $V(\epsilon(w))$. Here, $\Gamma$ is a certain left cell contained in $\tsc_\lambda$. (Such $\Gamma$ will be called the anti-canonical left cell of $\tsc_\lambda$; see Section \ref{sec:cancells} and \ref{sec:xibij}.) Furthermore, he proved that there exists a (non-canonical) isomorphism $\cJ_{\tsc_\lambda} \simeq \Mat_{\chi\times \chi} (\Rep(F_\lambda))$ compatible with the isomorphism above. Here, $\chi = \frac{n!}{\lambda_1!\lambda_2!\cdots}$ is the Euler characteristic of the Springer fiber of some/any $N \in \cO_\lambda$.

\subsection{AMBC and Kazhdan-Lusztig cells}
Here we briefly recall some notations from \cite{cpy18} and \cite{clp17} and discuss the relations between the affine matrix-ball construction (abbreviated AMBC) and Kazhdan-Lusztig cells.

For $T=(T_1, \ldots, T_l)\in \RSYT(\lambda)$ and $i\in [1, l-1]$ we define $\lch_i(T)$, called the local charge in row $i$ of $T$, as follows. Suppose that $T_i=(a_1, \ldots, a_s)$ and $T_{i+1} = (b_1, \ldots, b_t)$. Then $\lch_i(T)$ is the smallest $d\in \bN$ such that $a_{l-d}<b_{l}$ for $l\in [d+1, t]$. Pictorially, this measures necessarily shift of $T_i$ to the right so that $(T_i, T_{i+1})$ becomes a standard Young tableau (of skew shape). For example, if $T_i = (3,5,7,8)$ and $T_{i+1} = (1,2,4,6)$ then we have $\lch_i(T) = 2$ which can be obtained from the following picture.
\[\ytableaushort{3578,1246} \quad \Rightarrow\quad \ytableaushort{{\none}{\none}3578,1246}
\]

For $P, Q \in \RSYT(\lambda)$ where $\lambda=(\lambda_1, \ldots, \lambda_l)$, we define the symmetrized offset constants with respect to $(P, Q)$, denoted by $\offset_{P,Q}=(s_1, \ldots, s_l) \in \bZ^l$, as follows.
\[s_i = 
\left\{\begin{aligned}
&0& \textnormal{ if } i=1 \textnormal{ or } \lambda_i>\lambda_{i+1},
\\&s_{i-1}+\lch_{i-1}(P)-\lch_{i-1}(Q) & \textnormal{ otherwise}.
\end{aligned}\right.
\]
In other words, we have $s_i-s_{i-1} = \lch_{i-1}(P)-\lch_{i-1}(Q)$ whenever $\lambda_{i-1}=\lambda_i$. (See \cite[Definition 5.8]{cpy18} and \cite[Theorem 5.10]{clp17} for the equivalent definitions.) It is easy to show that for tabloids $P, Q, R$ of the same shape, we have
$$\offset_{P,Q}+\offset_{Q,R} = \offset_{P,R}, \textup{ thus in particular } \offset_{P,Q}+\offset_{Q,P}=\offset_{P,P}=0.$$

\begin{example} \label{ex:1}
Let $P= \tableau[sY]{\ol{2}, \ol{4}, \ol{6} \\ \ol{3}, \ol{7}, \ol{8} \\ \ol{1}, \ol{5}, \ol{9}}$ and
$Q= \tableau[sY]{\ol{3}, \ol{5}, \ol{7} \\ \ol{1}, \ol{2}, \ol{8} \\ \ol{4}, \ol{6}, \ol{9}}$. 
Then $\lch_1(P) = 0, \lch_2(P) = 1, \lch_1(Q)=2$, and $\lch_2(Q)=0$. Thus it follows that $\offset_{P,Q} = (0,-2,-1)$.  
\end{example}

For $\lambda=(\lambda_1, \ldots, \lambda_r)$ and $\vv{\rho} = (\rho_1, \ldots, \rho_r)$, we define $\rev_\lambda(\vv{\rho})$ 
to be the integer vector obtained from $\vv{\rho}$ by reversing the order of elements corresponding to the parts of the same length in $\lambda$. For example, if $\lambda=(2,2,1,1,1)$ and $\vv{\rho}=(3,1,5,2,4)$ then we have $\rev_\lambda(\vv{\rho}) = (1,3,4,2,5)$. We say that $\vv{\rho} \in \bZ^{l(\lambda)}$ is dominant with respect to $(P,Q)$ if $-(\vv{\rho}-\offset_{P,Q}) \in \Dom(F_\lambda)$, or equivalently $\rev_\lambda(\vv{\rho}-\offset_{P,Q}) \in \Dom(F_\lambda)$. In other words, $\vv{\rho} \in \bZ^{l(\lambda)}$ is dominant with respect to $(P,Q)$ if and only if $\vv{\rho}-\offset_{P,Q}$ is ``anti-dominant''.

\begin{example} \label{ex:2}
 In the Example \ref{ex:1} the vector $\vv{\rho} = (2,0,2) \in \bZ^3$ is dominant with respect to $(P,Q)$ because $\vv{\rho} - \offset_{P,Q} = (2,2,3)$ is a nondecreasing sequence.
\end{example}

We set
\begin{align*}
\Omega &\colonequals \bigsqcup_{\lambda\vdash n} \RSYT(\lambda) \times \RSYT(\lambda) \times \bZ^{l(\lambda)},
\\\Odom &\colonequals \{(P, Q, \vv{\rho}) \in \Omega \mid \vv{\rho} \textup{ is dominant with respect to } (P,Q)\}.
\end{align*}
We define $\Phi \colon \extSn \rightarrow \Odom \colon w \mapsto (P(w), Q(w), \vv{\rho}(w))$ to be the bijection defined in \cite{cpy18} using the affine matrix-ball construction (abbreviated AMBC). Also, let $\Psi \colon \Omega \rightarrow \extSn$ be the surjection defined by the backward AMBC. Then by \cite[Theorem 5.12]{cpy18} we have $\Psi|_{\Odom} = \Phi^{-1}$. Both AMBC and backward AMBC were explained in great detail in both \cite{cpy18} and \cite{clp17}.

Before we investigate their definitions we recall some of their properties in terms of Kazhdan-Lusztig cells. It follows from \cite{shi91} and \cite{cpy18} that for each tabloid $P$, the union of fibers $\Psi(P,Q,\vv{\rho})$ for $(P,Q,\vv{\rho}) \in \Omega$ where we vary $Q$ and $\vv{\rho}$ are exactly a right cell of $\extSn$, denoted by $\lc^{-1}_P$. Similarly, if we fix a tabloid $Q$, the union of fibers is a left cell of $\extSn$, denoted by $\lc_Q$. Moreover, it essentially follows from \cite{shi86} that $w\in \tsc_\lambda$ if and only if the shape of both $P(w)$ and $Q(w)$ is $\lambda$. On the other hand, \cite[Proposition 3.1]{clp17} implies that if $\Phi(w) = (P, Q, \offset_{P,Q}+\vv{\rho})$ then $\Phi(w^{-1}) = (Q, P, \offset_{Q, P}-\rev_\lambda(\vv{\rho})).$ If we want to restrict to a non-extended affine symmetric group $\affSn$ defined to be
\[\affSn=\{w \in \extSn \mid w(1)+\cdots+w(n) = n(n+1)/2\},\]
then by \cite[Theorem 10.3]{cpy18} the same claims hold verbatim if we add the condition $\sum_i \rho_i = 0$ where $\vv{\rho}=(\rho_1, \rho_2, \ldots)$.

For $w\in \extSn$, we define 
$$R(w) = \{\ol{i} \in [n] \mid w(i)>w(i+1)\} \qquad \textup{(resp. } L(w) = \{\ol{i} \in [n] \mid w^{-1}(i)>w^{-1}(i+1)\}\textup{)},$$
called the right (resp. left) descent set of $w$. Also for a tabloid $T$, we define its $\tau$-invariant by
$$\tau(T) \colonequals \{\ol{i} \in [n] \mid \ol{i} \textup{ lies in a strictly higher row of $T$ than } \ol{i+1}\}.$$
Then by \cite[Proposition 3.6]{clp17}, we have $L(w) = \tau(P(w))$ and $R(w) = \tau(Q(w))$.

\subsection{Definition of $\Phi$ and $\Psi$} \label{sec:phipsi} Here we discuss the definitions of $\Phi$ and $\Psi$ and related notations briefly.

For $(x_1, y_1), (x_2, y_2)\in \bZ^2$, we say that $(x_1, y_1)$ is (weakly) southwest of $(x_2, y_2)$, denoted by $(x_1,y_1)\leq_{SW} (x_2, y_2)$, if $x_1\geq x_2$ and $y_1\leq y_2$. Similarly, we say that $(x_1, y_1)$ is (weakly) northwest of $(x_2, y_2)$, denoted by $(x_1,y_1)\leq_{NW} (x_2, y_2)$, if $x_1\leq x_2$ and $y_1\leq y_2$. Pictorially, we regard elements in $\bZ^2$ as located in the $xy$-plane where the $x$-axis directs southbound and $y$-axis directs eastbound, i.e. the whole plane is rotated clockwise by $90^\circ$ from the conventional direction.

We define a partial (affine) permutation to be an injection $w: X+n\bZ \rightarrow \bZ$ for some $X \subset [1,n]$ such that $w(i+kn)=w(i)+kn$ for $i \in X$ and $k \in \bZ$. Then its window notations is defined to be $w=[w(1), \ldots, w(n)]$ where we set $w(i)=\emptyset$ for $i \not\in X$. We say that a partial permutation $S: X +n\bZ \rightarrow \bZ$ is a stream if $S(i)<S(j)$ whenever $i<j$ (if $S(i), S(j) \neq \emptyset$), i.e. it is a chain with respect to the northwest ordering. The density of $S$ is defined to be the size of $X$. The altitude $a(S)$ of $S$ is defined to be $\sum_{x\in X} (\ceil{w(x)/n}-1)$ where $\ceil{t}$ is the smallest integer not smaller than $t$. (cf. \cite[Definition 2.11]{clp17}) If we set $D(x, y) = \ceil{y/n}-\ceil{x/n}$, called the block diagonal of $(x,y)$, then we have $a(S) = \sum_{x\in X}D(x,w(x))$.

For a partial permutation $w$ and a substream $S \subset w$, we call $S$ a channel of $w$ if the density of $S$ is maximum among all the substreams of $w$. Among all the channels of $w$, there exists a unique one $C$ such that for any channel $C' \subset w$ and any ball $(x, w(x)) \in C$, there exists $(y, w(y)) \in C'$ such that $x\geq y$ and $w(x)\leq w(y)$. We call such $C$ the southwest channel of $w$.



For a stream $S$, we say that $\tilde{d}: S \rightarrow \bZ$ is a proper numbering if it is a bijection and $\tilde{d}(x, S(x))<\tilde{d}(y, S(y))$ whenever $x<y$. It is clear that such a numbering is unique up to shift. For a channel $C$l of $w$ equipped with the proper numbering $\tilde{d}:C \rightarrow \bZ$, we define the channel numbering $d^C_w=d^C: w \rightarrow \bZ$ to be 
\[d^C(b) = \max \{\tilde{d}(b')+k \mid \textnormal{there exists a reverse path } (b'=b_0, b_1, \ldots, b_k=b) \textnormal{ of balls in } w\}.\]
Here, a reverse path is a sequence $(x_1, y_1), \ldots, (x_m, y_m)$ such that $x_1<\cdots <x_m$ and $y_1<\cdots<y_m$. Then $d^C$ is also uniquely determined once $\tilde{d}$ is fixed. When $C$ is the southwest channel of $w$ we call it the southwest channel numbering of $w$ and write $d=d^C$ (or $d_w=d_w^C$).


For a partial permutation $w$ and its channel $C \subset w$, we define the forward step $w \mapsto (\fw_C(w),\st_C(w))$ as follows. For each $m\in \bZ$, we label the balls of $w$ mapped to $m$ under $d=d^C_w$ by $(x_1, w(x_1)), \ldots, (x_k, w(x_k))$ so that $x_1>\cdots>x_k$. Then it is easy to show that $w(x_1)<\cdots <w(x_k)$. Set $Z_m=I(Z_m)\sqcup O(Z_m) \sqcup S(Z_m)$, called the zigzag labeled $m$, to be
\begin{gather*}
I(Z_m)= \{ (x_i, w(x_i)) \mid i \in [1,k]\},
\\O(Z_m)= \{ (x_i, w(x_{i+1})) \mid i \in [1,k-1]\}, \qquad S(Z_m) = \{ (x_k, w(x_{1}))\}.
\end{gather*}
Then we have $w \cap Z_m = I(Z_m)$. Now set $\fw_C(w) = \bigsqcup_{m\in \bZ}O(Z_m)$ and $\st(w) = \bigsqcup_{m\in \bZ}S(Z_m)$. Then $\fw_C(w)$ is again a partial permutation and $\st_C(w)$ is a stream whose density is equal to $C$. When $C$ is the southwest channel of $w$, we simply write $\fw(w)$ and $\st(w)$ instead of $\fw_C(w)$ and $\st_C(w)$.

We are ready to provide the precise definition of $\Phi: \extSn \rightarrow \Odom$. Starting with $w_0=w \in \extSn$, we successively set $w_{i+1}=\fw(w_{i})$ and $S_{i+1}=\st(w_{i})$ until we obtain an empty partial permutation. For each $S_{i}$ there exist $P_i, Q_i \subset [1,n]$ such that $S_i : P_i +n\bZ \rightarrow Q_i+n\bZ$ is a bijection. Now we set $\Phi(w) = (P, Q, \vv{\rho})$ where $P=(P_1, P_2,\ldots)$, $Q=(Q_1, Q_2,\ldots)$, and $\vv{\rho}=(a(S_1), a(S_2), \ldots)$ (the sequence of altitudes of $S_i$).

For a stream $S$ and a partial permutation $w$, we say that $S$ is compatible with $w$ if $S \cup w$ is still a partial permutation and the density of $S$ is not smaller than that of any substream of $w$. Then we define the backward numbering $d^{\bk,S}_w=d^{\bk,S}: w \rightarrow \bZ$ as follows. (cf. \cite[Definition 4.1]{cpy18}) Let $\tilde{d}:S \rightarrow \bZ$ be a proper numbering and for $(x, w(x)) \in w$ we let $d(x, w(x))=\max\{\tilde{d}(y, S(y)) \in S \mid y<x, S(y)<w(x)\}$. Now we repeat the following process:
\begin{enumerate}[label=\textbullet]
\item If $d(x, w(x))<d(y, w(y))$ for any $x, y$ such that $x<y$ and $w(x)<w(y)$ (if $w(x), w(y)\neq\emptyset$), then we terminate the process.
\item Otherwise, choose a ball $(x, w(x))$ such that:
\begin{enumerate}[label=-]
\item there exists a ball $(y,w(y))$ such that $d(x,w(x))\geq d(y,w(y))$, $x<y$, and $w(x)<w(y)$,
\item for any ball $(z, w(z))$ such that $z<x$ and $w(z)<w(x)$ we have $d(z,w(z))<d(x,w(x))$.
\end{enumerate}
\item For each $i \in \bZ$ we lower the value of $d(x+in, w(x)+in)$ by 1 and return to the first step.
\end{enumerate}
After this process is done, we set $d^{\bk, S}_w=d$. This numbering is always well-defined.

For a partial permutation $w$ and a compatible stream $S$, we define the backward step $(w, S) \mapsto \bk_S(w)$ as follows. Let $\tilde{d}$ be the proper numbering on $S$ and $d=d^{\bk,S}_w$ be the induced backward numbering on $w$. For each $m\in \bZ$, we label the balls of $w$ mapped to $m$ under $d$ by $(x_1, w(x_1)), \ldots, (x_k, w(x_k))$ so that $x_1>\cdots>x_k$. Also there exists a unique ball $(y, S(y))$ such that $\tilde{d}(y, S(y))=m$. Then it is easy to show that $w(x_1)<\cdots <w(x_k)$. Set $Z_m = I(Z_m)\sqcup O(Z_m)\sqcup S(Z_m)$, called the zigzag labeled $m$, to be
\begin{gather*}
I(Z_m)= \{ (x_{i+1}, w(x_i)) \mid i \in [1,k-1]\}\sqcup\{(x_1, S(y))\}\sqcup\{(y, w(x_k))\},
\\ O(Z_m) = \{ (x_i, w(x_{i})) \mid i \in [1,k]\}, \qquad S(Z_m) = \{ (y, S(y))\}.
\end{gather*}
 Then $w \cap Z_m = O(Z_m)$ and $S \cap Z_m = S(Z_m)$. We set $\bk_S(w) = \bigsqcup_{m \in \bZ} I(Z_m)$. Then $\bk_S(w)$ is a partial permutation each of whose channel has its density equal to that of $S$.

Now we define $\Psi: \Omega \rightarrow \extSn$ as follows. For $P=(P_1, \ldots, P_l)$, $Q=(Q_1, \ldots, Q_l)$, and $\vv{\rho} = (\rho_1, \ldots, \rho_l)$, we define $S_i$ to be the unique stream yielding a bijection from $P_i+n\bZ$ to $Q_i+n\bZ$ and whose altitude is given by $a(S_i)=\rho_i$. Now starting with the empty partial permutation $w_l$ we successively define $w_{i-1}=\bk_{S_i}(w_i)$ for $i \in [1,l]$. This process is well-defined and we set $\Psi(P, Q, \vv{\rho}) = w_0$.
\begin{example}
 For $P$ and $Q$ as in Example \ref{ex:1} and for $\vv{\rho} = (2,0,2)$ as in Example \ref{ex:2} we have $\Psi(P,Q,\vv{\rho}) = [3,7,14,2,18,4,19,8,6]$. 
\end{example}

\section{Canonical and anti-canonical left cells} \label{sec:cancells}
Let $w_0 \in \Sym_n$ be the longest element in $\Sym_n$ and define
$$\extSn_f \colonequals \{w \in \extSn \mid l(w)+l(w_0) = l(ww_0)\} = \{w \in \extSn \mid R(w) \subset \{\ol{n}\}\},$$
i.e. the set of elements each of which has minimal length in its left $\Sym_n$-coset. Then for each two-sided cell $\tsc \subset \extSn$, the intersection $\lcc \colonequals \tsc \cap \extSn_f$ is a single left cell and is called the canonical left cell of $\tsc$ \cite{lx88}. If $\tsc=\tsc_\lambda$, then we denote its canonical left cell by $\lcc_\lambda$. 

\begin{lem} Let $\Tc_\lambda$ be the reverse row superstandard tabloid of shape $\lambda$ with start at 1 in the sense of \cite[2.1]{clp17} (see the example below). If $w \in \lcc_\lambda$, then we have $Q(w)=\Tc_\lambda$.
\end{lem}
\begin{proof} By assumption we have $\tau(Q(w))=R(w) \subset \{\ol{n}\}$. Since $\Tc_\lambda$ is the unique tabloid of shape $\lambda$ which satisfies $\tau(\Tc_\lambda) \subset \{\ol{n}\}$, the result follows.
\end{proof}

\begin{example}
 We have $\Tc_{(4,3,1)} = \tableau[sY]{\ol{5}, \ol{6}, \ol{7}, \ol{8} \\ \ol{2}, \ol{3}, \ol{4} \\ \ol{1}}.$
\end{example}

This time, let $w_0^\lambda$ be the longest element in $\Sym_{\lambda'} \colonequals \Sym_{\lambda_1'}\times \Sym_{\lambda_2'} \times \cdots$. Thus for example $w_0^{(1,1,\dotsc,1)} = w_0$ and $w_0^{(n)} = id$. By direct calculation, we see that the shape of $P(w_0^\lambda)$ and $Q(w_0^\lambda)$ is $\lambda$, thus $w_0^\lambda \in \tsc_\lambda$.

\begin{lem}\label{lem:antican} Let $\Ta_\lambda$ be the standard Young tableau of shape $\lambda$ with $\{1, 2, \ldots, \lambda_1'\}$ on the first column, $\{\lambda_1'+1, \ldots, \lambda_1'+\lambda_2'\}$ on the second column, etc. (See the example below.) Then we have $\Phi(w_0^\lambda)=(\Ta_\lambda, \Ta_\lambda, 0)$.
\end{lem}
\begin{proof} It is clear that $R(w) = [n]-\{\lambda_1', \lambda_1'+\lambda_2', \ldots\}$. Since $\Ta_\lambda$ is the unique tabloid of shape $\lambda$ which satisfies $\tau(\Ta_\lambda) = [n]-\{\lambda_1', \lambda_1'+\lambda_2', \ldots\}$, we have $Q(w_0^\lambda)=\Ta_\lambda$. Since $w_0^\lambda$ is an involution, we also have $P(w_0^\lambda) =\Ta_\lambda$ by \cite[Proposition 3.1]{clp17}. Moreover, $w_0^\lambda \in \Sym_n$ thus $\vv{\rho}(w) = 0$ by \cite[Theorem 10.2]{cpy18}.
\end{proof}

We define $\lca_\lambda \subset \tsc_\lambda$ to be the left cell containing $w_0^\lambda$, called the anti-canonical left cell of $\tsc_\lambda$. Then the lemma above implies that $w \in \lca_\lambda$ if and only if $Q(w)=\Ta_\lambda$.

\begin{example}
 We have $\Ta_{(4,3,1)} = \tableau[sY]{\ol{1}, \ol{4}, \ol{6}, \ol{8} \\ \ol{2}, \ol{5}, \ol{7} \\ \ol{3}}.$
\end{example}

\section{Bijection of Xi and AMBC} \label{sec:xibij}
Our goal in this section is to show that the bijection $\vv{\varepsilon} \colon (\lca_\lambda)^{-1} \cap \lca_\lambda \rightarrow \Dom(F_\lambda)$ defined by Xi \cite{xi02} can be interpreted in terms of AMBC. First we describe his bijection $\vv{\varepsilon}: (\lca_\lambda)^{-1} \cap \lca_\lambda  \rightarrow \Dom(F_\lambda)$ following \cite[5.2.1]{xi02}. 
We start with the following definition.
\begin{defn} For $\lambda = (\lambda_1, \lambda_2, \ldots) \vdash n$, we say that $\{S_1, S_2, \ldots, S_{l(\lambda)}\}$ is a complete stream family of $w\in \tsc_\lambda$ (or {\it {complete antichain family}} in the sense of \cite{xi02}) if $S_i$ are pairwise disjoint streams, the density of each $S_i$ is equal to $\lambda_i$, and $\bigsqcup_{i=1}^{l(\lambda)} S_i = \cB_w$.
\end{defn}

Let $\{S_1, S_2, \ldots, S_{l(\lambda)}\}$ be a complete stream family of $w \in (\lca_\lambda)^{-1} \cap \lca_\lambda$, which always exists by \cite[Theorem 5.1.12]{xi02}. 
After rearranging $S_i$'s if necessary, we may assume that $a(S_i)\geq a(S_{i+1})$ whenever $\lambda_i = \lambda_{i+1}.$ Now we define $\vv{\varepsilon}(w)=(a(S_1), a(S_2), \ldots, a(S_{l(\lambda)})) \in \Dom(F_\lambda)$, which gives the bijection $\vv{\varepsilon}: (\lca_\lambda)^{-1} \cap \lca_\lambda  \rightarrow \Dom(F_\lambda)$ of Xi. By \cite[Lemma 5.2.4]{xi02}, this value does not depend on the choice of the complete stream family $\{S_1, S_2, \ldots, S_l\}$ of $w$.

We state the main theorem of this section.
\begin{thm}  \label{thm:xibij} For $w\in (\lca_\lambda)^{-1} \cap \lca_\lambda$, we have $\vv{\varepsilon}(w) = \rev_\lambda(\vv{\rho}(w))$.
\end{thm}
The rest of this section is devoted to its proof. 
To this end, first we define 
$$\cP_\lambda = \left\{\left[an+ \sum_{j=1}^{i-1} \lambda'_j+1, an+ \sum_{j=1}^i\lambda'_j\right] \subset \bZ \mid a \in \bZ, 0\leq i \leq \lambda_1\right\}.$$
For example, $\cP_{(3,2,1,1)}$ consists of $\{1,2,3,4\}, \{5,6\}, \{7\}$ and their shifts by multiples of $7$. It is clear that $\cP_{\lambda}$ is a partition of $\bZ$.
\begin{lem} \label{lem:ordered}Let $w \in \lca_\lambda$. 
\begin{enumerate}[label=\textup{(\alph*)}]
\item $\cB_w \cap (A \times \bZ)$ for any $A\in P_\lambda$ is totally ordered with respect to the southwest ordering.
\item For any stream $S \subset \cB_w$, $\#(S \cap (A\times \bZ)) \leq 1$.
\end{enumerate}
\end{lem}
\begin{proof} Since $w\in \lca_\lambda$, we have $Q(w) = \Ta_\lambda$ which means that $R(w) = \tau(Q(w)) = [n] - \{\ol{\lambda'_1}, \ol{\lambda_1'+\lambda_2'}, \ldots\}$. It implies that
$w( \sum_{j=1}^{i-1} \lambda'_j+1 )>w(\sum_{j=1}^{i-1} \lambda'_j+2)>\cdots > w( \sum_{j=1}^i\lambda'_j)$ for any $0\leq i \leq \lambda_1$. Now the statement easily follows from this.
\end{proof}

For a stream $S \subset \cB_w$, we define
$$S^\square\colonequals \{ (A,B) \in P_\lambda^2 \mid S \cap (A\times B) \neq \emptyset\}.$$
Suppose that $w \in (\lca_\lambda)^{-1} \cap \lca_\lambda$. Then the lemma above applied to both $w$ and $w^{-1}$ implies the following: for any stream $S \subset \cB_w$ we have $(A,B) \in S^\square \Leftrightarrow \#(S\cap(A\times B))=1$. Also for any $(A,B), (C,D) \in S^\square$, we have either $A<C, B<D$ or $A>C, B>D$. (Here $X<Y$ means $x<y$ for any $x\in X$ and $y\in Y$.) In other words, $S^\square$ is totally ordered with respect to northwest ordering.

%


The next lemma is the key step of the proof of the main theorem in this section. 
\begin{lem} \label{lem:strfam} Let $w\in (\lca_\lambda)^{-1} \cap \lca_\lambda$, and assume that $\{S_2', S_3', \ldots, S_{l(\lambda)}'\}$ is a complete stream family of the partial permutation $\fw(w)$. Then there exists a complete stream family $\{S_1, S_2, \ldots, S_{l(\lambda)}\}$ of $w$ such that
\begin{enumerate}[label=\textup{(\arabic*)}]
\item $S_1$ is a channel of $\cB_w$, 
\item $a(S_1)=\rho_1$ where $\vv{\rho}(w) = (\rho_1, \rho_2, \ldots)$, and
\item $S_i^\square = S_i'^\square$.
\end{enumerate}
\end{lem}
\begin{proof} 
Recall the construction of the backward AMBC algorithm. In particular, $\cB_w$ is obtained by following the steps below.
\begin{enumerate}[label=\textup{(\alph*)}]
\item Let $S\colonequals\st_{\rho_1}(\Ta_1, \Ta_1)$ and give $S$ some proper numbering; here $\Ta_\lambda=(\Ta_1, \Ta_2, \ldots)$. (Here, $\st_{\rho_1}(\Ta_1, \Ta_1)$ is the unique stream of altitude $\rho_1$ which maps $\Ta_1+n\bZ$ to itself. See \cite[Section 3.4]{cpy18} for the definition of $\st$.)
\item Calculate the backward numbering $d_w^{\textup{bk}, S}$ on $\cB_{\fw(w)}$. 
\item For each $i \in \bZ$, consider the zigzag $Z_i$ labeled $i$ whose outer corner-posts are balls in $\cB_{\fw(w)}$ labeled by $i$.
\item The union of all inner corner-posts of $Z_i$ for all $i \in \bZ$ is $\cB_w$.
\end{enumerate}

Now suppose that some proper numbering on $S=\st_{\rho_1}(\Ta_1, \Ta_1)$ is given, say $S=\{S^{(i)}\mid i\in \bZ\}$ where $S^{(i)}= (a_i, b_i) \in S$ is the ball labeled $i$. We claim that if $(x,y) \in \cB_{\fw(w)}$ satisfies $(x,y) \geq_{NW}S^{(i)}$ but $(x,y) \not\geq_{NW}S^{(i+1)}$, then $d_w^{\textup{bk}, S}(x,y) = i$. In other words, on the construction of the backward numbering in \ref{sec:phipsi} the iterating process terminates without changing the initial numbering, say $d_0$. Indeed, suppose $(a, b), (c,d) \in \cB_{\fw(w)}$ satisfies $d_0(a,b) = d_0(c,d)$. In other words,
$$(a,b), (c,d) \in \{(x, y) \in \bZ^2 \mid x\geq a_i, y\geq b_i\}-\{(x, y) \in \bZ^2 \mid x\geq a_{i+1}, y\geq b_{i+1}\}.$$
However, the balls in
$$\cB_{\fw(w)} \cap\left(\{(x, y) \in \bZ^2 \mid x\geq a_i, y\geq b_i\}-\{(x, y) \in \bZ^2 \mid x\geq a_{i+1}, y\geq b_{i+1}\}\right)$$
are totally ordered with respect to southwest ordering by Lemma \ref{lem:ordered}. Note that this lemma applies to $\fw(w)$ since $\Phi(\fw(w)) = ((\Ta_\lambda)_{\geq 2}, (\Ta_\lambda)_{\geq 2}, \vv{\rho}_{\geq 2})$ where $(\Ta_\lambda)_{\geq 2} =(\Ta_2,\Ta_3, \ldots)$ is also an anti-canonical tableau when restricted to indices not in $\Ta_1$. Therefore, if $(a,b) \leq_{NW}(c,d)$ then $(a,b)=(c,d)$ and the claim follows by looking at the construction of the backward numbering.


Now let $I(Z_i)$ be the set of inner corner-posts of $Z_i$, $O(Z_i)$ be the set of outer corner-posts of $Z_i$, and $S(Z_i)=\{S^{(i)}\}$ as in \ref{sec:phipsi}. Then the backward AMBC algorithm replaces balls in $O(Z_i) \cup S(Z_i)$ with those in $I(Z_i)$. By Lemma \ref{lem:ordered}, it is easy to check that for $(x,y) \in I(Z_i)$ and $(z,w) \in O(Z_i)$ we have
\begin{enumerate}[label=$\bullet$]
\item if $y=w$ and $a_i \leq x,z <a_{i+1}$, then $z=x+1$ and
\item if $x=z$ and $b_i \leq y,w <b_{i+1}$, then $w=y+1$.
\end{enumerate}
From this, it follows that for any $A, B \in \cP_\lambda$ we have
$$\#((A\times B) \cap I(Z_i)) = \#((A\times B) \cap (O(Z_i) \cup S(Z_i))).$$
In other words, we have
\begin{enumerate}[label=$\bullet$]
\item $\#((A\times B) \cap \cB_w) = \#((A\times B) \cap \cB_{\fw(w)})+1$ if $(A,B) \in S^\square$, and
\item $\#((A\times B) \cap \cB_w) = \#((A\times B) \cap \cB_{\fw(w)})$ otherwise.
\end{enumerate}

Now we can choose $S_1, S_2, \ldots$ so that $S_1^\square = S^\square$ and $S_i^\square = S_i'^\square$ for $2\leq i \leq l(\lambda)$. Indeed, for any pair $A, B \in \cP_\lambda$ and for any stream $S_i'$ we know that at most one element of this stream lies in $A \times B$. We also know the same for the stream $S$, as can be seen by examining the rows of $\Ta_{\lambda}$. Let us arbitrarily biject elements of $O(Z_i) \cup S(Z_i)$ lying inside a particular $A \times B$ with those of $I(Z_i)$ lying inside the same $A \times B$. This maps each of the streams $S$, $S_i'$ into a new stream $S_1$, $S_i$. It is clear that $\{S_1, S_2, \ldots\}$ satisfy the desired properties.
\end{proof}

To illustrate this proof, we give the following example. 

\begin{example}
 Let $n = 16$ and $w = [11,5,4,3,2,-9,13,10,9,8,1,15,12,22,14,16]$. Then $\fw(w) = [\emptyset, 11,5,4,3,-8,\emptyset, 13,10,9,2,\emptyset,15,\emptyset,22,\emptyset]$. It is easy to see that $w \in (\lca_\lambda)^{-1} \cap \lca_\lambda$ for $\lambda = (5,4,2,2,2,1)$. Let $A =B = \{7,8,9,10,11\}$. Then 
 $$(A\times B) \cap \cB_w = \{(8,10), (9,9), (10,8)\} \text{  and } (A\times B) \cap (\cB_{\fw(w)} \cup S) = \{(7,7), (9,10), (10,9)\}.$$ We 
 choose an arbitrary bijection between the two sets, and repeat it for other pairs $(A,B)$. Then streams of any complete stream family of $\cB_{\fw(w)}$ together with $S$ biject to streams of a complete stream family of $\cB_{w}$.
\end{example}

We are ready to prove the main theorem in this section.
\begin{proof}[Proof of Theorem \ref{thm:xibij}]
We proceed by induction on $l=l(\lambda) = \lambda'_1$. When $l=1$, we have $w = \shift^d$ for some $d\in\bZ$ and in this case it is easy to see that $\varepsilon(w) = \rev_\lambda(\vv{\rho})(w)=(d)$. (Recall that $\shift$ is the shift permutation $[2,3,\ldots, n, n+1]$.) In general, we choose a complete stream family $\{S_2', S_3', \ldots, S_l'\}$ of $\fw(w)$ and $\{S_1, S_2, \ldots, S_l\}$ of $w$ as Lemma \ref{lem:strfam}. Then $\rho_1=a(S_1)$, and for any $2\leq i \leq l$ we have
$$a(S_i)= \sum_{(x,y) \in S_i, 1\leq x \leq n} \ceil{\frac{y}{n}}-1=\sum_{(x,y) \in S_i', 1\leq x \leq n} \ceil{\frac{y}{n}}-1=a(S_i')$$
since $S_i^\square = S_i'^\square$. In other words, as a multiset we have
\begin{align*}
\{\rho_1, \rho_2, \ldots, \rho_l\} &= \{a(S_1)\} \cup \{\rho_2, \ldots, \rho_l\}
\\&= \{a(S_1)\} \cup \{a(S_2'), \ldots, a(S_l')\}
\\&= \{a(S_1)\} \cup \{a(S_2), \ldots, a(S_l)\}
\end{align*}
by induction assumption. After reordering $a(S_i)$ if necessary, we have $\vv{\varepsilon}(w)=\rev_\lambda(\vv{\rho}(w))$ as desired.
\end{proof}

\section{Shift, Knuth moves, and AMBC}
In this section we discuss how multiplication by $\shift=[2,3,\ldots, n,n+1]$ and Knuth moves behave under the AMBC. We start with some definitions.
\begin{defn} For $\lambda= (\lambda_1, \ldots, \lambda_r)$ and $\vv{\rho} = (\rho_1, \ldots, \rho_r)$, we say that $\vv{\rho}$ is determinantal (with respect to $\lambda$) if $\rho_i=\rho_j$ whenever $\lambda_i = \lambda_j$.
\end{defn}
Note that if $\vv{\rho}$ is determinantal with respect to $\lambda$ then both  $\vv{\rho} \in \Dom(F_\lambda)$ and $\rev_\lambda(\vv{\rho}) \in \Dom(F_\lambda)$. Also it is a highest weight of a tensor product of powers of determinant representations of factors of $F_\lambda$.

\begin{defn} Let $T=(T_1, T_2, \ldots)$ be a tabloid of shape $\lambda$ and suppose that $\ol{i} \in T_t$ for some $1\leq i \leq n$ and $1\leq t \leq l(\lambda)$. We define $\vv{\delta}(T,i)=(\delta_1, \delta_2, \ldots, \delta_{l(\lambda)})$ and $\vv{\iota}(T, i)=(\iota_1, \iota_2, \ldots, \iota_{l(\lambda)})$ as follows. Here $[-]$ is the Iverson bracket.
\begin{enumerate}[label=$\bullet$]
\item If $\lambda_t=\lambda_{t-1}$, then we put $\vv{\delta}(T,i)= \vv{0}$. Otherwise, $ \vv{\delta}(T,i) =  ([\lambda_j=\lambda_t])_{1\leq j \leq l(\lambda)}.$
\item $\vv{\iota}(T,i) = ([j=t])_{1\leq j \leq l(\lambda)}=([\ol{i} \in T_j])_{1\leq j \leq l(\lambda)}$.
\end{enumerate}
\end{defn}


\subsection{Multiplication by $\shift$} Here we discuss how to relate $\Phi(\shift w), \Phi(w \shift^{-1})$ and $\Phi(w)$. Let us start with the following lemma.

\begin{lem}\label{lem:shiftlch} Let $T=(T_1, T_2, \ldots)$ be a tabloid of shape $\lambda$ and suppose that $\ol{n}$ is contained in $T_t$.
\begin{enumerate}[label=\textup{(\arabic*)}]
\item $\lch_j(\shift(T))=\lch_j(T)$ if $j\neq t, t-1$ and $\lambda_j = \lambda_{j+1}$.
\item If $\lambda_t=\lambda_{t+1}$, then $\lch_t(\shift(T))= \lch_t(T)-1$. 
\item If $\lambda_{t-1}=\lambda_t$, then $\lch_{t-1}(\shift(T))= \lch_{t-1}(T)+1$.
\end{enumerate}
\end{lem}
\begin{proof} The first part is straightforward. For (2), we choose the ordering on $\shift(T_t)$ to be the standard activation ordering as in \cite[Definition 5.1]{clp17} and that on $T_t$ to be the same one except that $\ol{n}$ is set to be the smallest entry. Then $\shift : T_t \rightarrow \shift(T_t)$ is order-preserving and $\shift$ also preserves the charge matching of ($T_t$, $T_{t+1}$), i.e. if $\ol{a} \in T_t$ is matched to $\ol{b} \in T_{t+1}$, then $\shift(\ol{a}) = \ol{a+1} \in \shift(T_t)$ is matched to $\shift(\ol{b}) = \ol{b+1} \in \shift(T_{t+1})$. (See \cite[Definition 5.2]{clp17} for the definition of the charge matching.) However, $\ol{n} \in T_t$ contributes to the charge whereas $\shift(\ol{n}) = \ol{1} \in \shift(T_t)$ does not, from which (2) follows. (3) is also similarly proved.
\end{proof}

\begin{lem} \label{lem:shiftoffset}
For tabloids $P, Q$ of the same shape $\lambda$, we have 
$$\offset_{P, \shift(Q)}=\offset_{P, Q}-\vv{\iota}(Q, n)+\vv{\delta}(Q, n) \quad \textup{ and } \quad \offset_{\shift(P),Q} = \offset_{P, Q} +\vv{\iota}(P, n) - \vv{\delta}(P,n).$$
\end{lem}
\begin{proof} It follows from Lemma \ref{lem:shiftlch} and \cite[Theorem 5.10]{clp17}.
\end{proof}

\begin{prop} \label{prop:shift}
Suppose that $\Phi(w) = (P, Q, \offset_{P,Q}+\vv{\rho})$. Then we have
\begin{align*}
\Phi(w\shift^{-1}) &= (P, \shift(Q), \offset_{P,\shift(Q)}+\vv{\rho}-\vv{\delta}(Q,n))
\\\Phi(\shift w) &= (\shift(P), Q, \offset_{\shift(P),Q}+\vv{\rho}+\vv{\delta}(P,n))
\end{align*}
In particular, if  $\Phi(w) = (P, P, \vv{\rho})$ then $\Phi(\shift w \shift^{-1}) = (\shift(P), \shift(P), \vv{\rho})$.
\end{prop}
\begin{proof} 
By Lemma \ref{lem:shiftoffset} it is enough to argue the following:
\begin{align*}
\Phi(w\shift^{-1}) &= (P, \shift(Q), \offset_{P,Q}+\vv{\rho}-\vv{\iota}(Q,n)),
\\\Phi(\shift w) &= (\shift(P), Q, \offset_{P,Q}+\vv{\rho}+\vv{\iota}(P,n)).
\end{align*}
This follows from the definition of the altitude of a stream. Indeed, if we compare the process of applying the AMBC to 
$w$ and $\shift w$, they are identical up to the $\shift$ of residues modulo $n$, except there is a unique forward step of the AMBC when the block diagonal $D(b)$ of exactly one ball $b$ increases by $1$ in the shifted version. It is exactly the step $t$ such that $\bar n \in P_t$. Similar consideration works for $w\shift^{-1}$.
\end{proof}

We finish with an example. 

\begin{example} \label{ex:shift}
Let $n=9$ and $w = [-1,3,10,-5,14,-3,18,7,2]$. Then $\Phi(w) = (P, Q, \offset_{P,Q}+\vv{\rho})$, where $P, Q, \offset_{P,Q}$ are as in the Example \ref{ex:1},
and $\vv{\rho} = (0,1,2)$. Then $\shift w = [0,4,11,-4,15,-2,19,8,3]$ and $\Phi(\shift w) = (\shift(P), Q, \offset_{P,Q}+\vv{\rho'})$, where $\vv{\rho'} = (0,1,3) = \vv{\rho} + (0,0,1) = \vv{\rho} + \iota(P,9)$, and  $\shift(P) = \tableau[sY]{\ol{3}, \ol{5}, \ol{7} \\ \ol{4}, \ol{8}, \ol{9} \\ \ol{1}, \ol{2}, \ol{6}}$. One checks that $\offset_{\shift(P),Q} = (0,-2,0)$, in agreement with Proposition \ref{prop:shift}.
\end{example}

\subsection{Knuth moves and star operations}
Until the end of this section we assume $n\geq 3$.

For $w\in \extSn$ and $i \in \bZ$ such that either $w(i-1)$ or $w(i+2)$ is between $w(i)$ and $w(i+1)$, we define the right star operation $w \mapsto w^*$ for $* \sim \ol{i}$ where $w^*$ is obtained from $w$ by exchanging $w(i+kn)$ and $w(i+1+kn)$ for each $k \in \bZ$. We similarly define the left star operation $w \mapsto {}^*w$ for $*\sim \ol{i}$ to be $w \mapsto ((w^{-1})^{*})^{-1}$ if $w^{-1}$ satisfies the above condition. We also call them Knuth moves as it is an affine analogue of the usual Knuth moves for $\Sym_n$.


\begin{defn} Let $n\geq 3$. For a tabloid $T$ and $1\leq i\leq n$, suppose that the rows of $T$ containing $\ol{i}$ and $\ol{i+1}$ are different. Let $S$ be the tabloid obtained from $T$ by exchanging $\ol{i}$ and $\ol{i+1}$. If they satisfy either
\begin{gather*}
\{ \tau(T) \cap \{\ol{i}, \ol{i+1}\}, \tau(S) \cap \{\ol{i}, \ol{i+1}\} \} =\{\{\ol{i}\}, \{\ol{i+1}\}\} \textup{ or}
\\\{ \tau(T) \cap \{\ol{i-1}, \ol{i}\}, \tau(S) \cap \{\ol{i-1}, \ol{i}\} \} =\{\{\ol{i-1}\}, \{\ol{i}\}\},
\end{gather*}
then we say that $T^*$ is well-defined for $*\sim \ol{i}$, and define $T^*\colonequals S$.
\end{defn}
Note that $T^*$ is well-defined then $(T^*)^*$ is also well-defined and $(T^*)^*=T$. Moreover, by \cite[Proposition 3.6]{clp17}, we have $L(w) = \tau(P(w))$ and $R(w) = \tau(Q(w))$. This implies that for $i \in \bZ$, $Q(w)^*$ (resp. $P(w)^*$) is well-defined for $*\sim \ol{i}$ if and only if $w^*$ (resp. ${}^*w$) is well-defined for $*\sim \ol{i}$.
Now \cite[Theorem 3.11]{clp17} provides the following description of how  star operations changes the result of AMBC.
Suppose that $\Phi(w) = (P, Q, \vv{\rho})$ for some $w\in \extSn$. If $*\not\sim \ol{n}$, then we have
$$\Phi(w^*) = (P, Q^*, \vv{\rho}) \textup{ and } \Phi({}^*w) = (P^*, Q, \vv{\rho})$$
whenever $w^*, {}^*w$ are well-defined. If $*\sim \ol{n}$, then we have
$$\Phi(w^*) = (P, Q^*, \vv{\rho}+\vv{\iota}(Q,1)-\vv{\iota}(Q,n)) \textup{ and } \Phi({}^*w) = (P^*, Q,\vv{\rho}-\vv{\iota}(P,1)+\vv{\iota}(P,n))$$
whenever $w^*, {}^*w$ are well-defined.

\ytableausetup{smalltableaux, aligntableaux=center}
Now we discuss how the star operation affects symmetrized offset constants.
\begin{lem} \label{lem:knuthlch} For $n \geq 3$, let $T$ be a tabloid of shape $\lambda \vdash n$ such that $T^*$ is well-defined for $*\sim \ol{i}$. If $\ol{i} \neq \ol{n}$, then  $\lch_j(T) = \lch_j(T^*)$ for any $j$. Now suppose that $\ol{i} = \ol{n}$, $\ol{1} \in T_s$, and $\ol{n} \in T_t$.
\begin{enumerate}[label*=\textup{(\alph*)}]
\item If $|t-s|\geq 2$, then $\lch_j(T^*)=\lch_j(T)$ if $j \not\in \{s-1,s, t-1,t\}$ and
\begin{enumerate}[label=$\bullet$]
\item if $\lambda_s=\lambda_{s+1}$, then $\lch_s(T^*)=\lch_s(T)+1$. 
\item if $\lambda_{s-1}=\lambda_s$, then $\lch_{s-1}(T^*)=\lch_{s-1}(T)-1$.
\item if $\lambda_t=\lambda_{t+1}$, then $\lch_t(T^*)=\lch_t(T)-1$. 
\item if $\lambda_{t-1}=\lambda_t$, then $\lch_{t-1}(T^*)=\lch_{t-1}(T)+1$. 
\end{enumerate}
\item If $t=s+1$, then either $\ol{2} \in T_{t}$ or $\ol{n-1} \in T_{s}$. In this case $\lch_j(T^*)=\lch_j(T)$ if $j \not\in \{s-1, s=t-1, t\}$ and
\begin{enumerate}[label=$\bullet$]
\item if $\lambda_{s-1}=\lambda_{s}$, then $\lch_{s-1}(T^*)=\lch_{s-1}(T)-1$.
\item if $\lambda_s=\lambda_{s+1}(=\lambda_t)$, then $\lch_s(T^*)=\lch_s(T)+2$.
\item if $\lambda_t=\lambda_{t+1}$, then $\lch_t(T^*)=\lch_t(T)-1$.
\end{enumerate}
\item If $s=t+1$, then either $\ol{2} \in T_{s}$ or $\ol{n-1} \in T_{t}$. In this case $\lch_j(T^*)=\lch_j(T)$ if $j \not\in \{t-1, t=s-1, s\}$ and
\begin{enumerate}[label=$\bullet$]
\item If $\lambda_{t-1}=\lambda_t$, then $\lch_{t-1}(T^*)=\lch_{t-1}(T)+1$.
\item If $\lambda_t=\lambda_{t+1}(=\lambda_s)$, then $\lch_{s-1}(T^*)=\lch_{s-1}(T)-2$.
\item If $\lambda_{s}=\lambda_{s+1}$, then $\lch_{s}(T^*)=\lch_{s}(T)+1$.
\end{enumerate}
\end{enumerate}
\end{lem}
\begin{proof} Let us first assume that $*\sim \ol{i} \neq \ol{n}$ and suppose that $\ol{i} \in T_t$ and $\ol{i+1}\in T_s$. Then $s\neq t$, and if $|s-t|\geq 2$ then it is clear that $\lch_j(T^*)=\lch_j(T)$ for any $j$. Thus it remains to check the case when $|s-t|=1$, i.e. $t=s+1$ or $s=t+1$. But since $T^{**}=T$, by symmetry it suffices to show the case when $t=s+1$ which we assume from now on.

It is still clear that $\lch_j(T^*)=\lch_j(T)$ if $j \neq s$. Thus we only need to show that $\lch_s(T^*)=\lch_s(T)$ when $\lambda_s=\lambda_{s+1}$. Since in this case we have $\ol{i} \in \tau(T^*)-\tau(T)$, for the star operation to be well-defined we should have either $\ol{i+2} \in T_{t}=T_{s+1}$ so that $\ol{i+1} \in \tau(T)-\tau(T^*)$ or $\ol{i-1} \in T_{s}=T_{t-1}$ so that $\ol{i-1} \in \tau(T)-\tau(T^*)$.

Here we only deal with the case when $\ol{i+2} \in T_{t}=T_{s+1}$, but the other case can also be similarly proved. We take activation orderings on $T_s$ and $T^*_{s}$ which are the same as the standard activation ordering except that $\ol{i+1}\in T_s$ and $\ol{i} \in T^*_{s}$ are the smallest element in each row. Then $\ol{i+1} \in T_s$ is matched to $\ol{i+2} \in T_{s+1}$ and $\ol{i} \in T_s^*$ is matched to $\ol{i+1} \in T_{s+1}^*$. Furthermore, the matchings between $T_s$ and $T_{s+1}$ are the same as those between $T_s^*$ and $T_{s+1}^*$ except one; there exists $\ol{j}$ such that $\ol{j} \in T_{s}$ is matched to $\ol{i} \in T_{s+1}$ and $\ol{j} \in T_{s}^*$ is matched to $\ol{i+2} \in T_{s+1}^*$. However, this does not affect the local charge since $\ol{j} \neq \ol{i+1}$, i.e. either $j < i$ or $i+2<j$. Thus we have $\lch_{s}(T^*)=\lch_{s}(T)$.

It remains to consider the case when $*\sim \ol{n}$. We assume that $\ol{1} \in T_s$ and $\ol{n} \in T_t$, and first suppose that $|t-s|\geq 2$. It is still clear that $\lch_j(T^*)=\lch_j(T)$ if $j \not\in \{s-1, s, t-1, t\}$. Now if $\lambda_s=\lambda_{s+1}$, then we choose activation orderings on $T_{s}$ and $T^*_{s}$ which are the same as the standard activation ordering except that $\ol{n} \in T^*_{s}$ becomes the smallest element. Then it is clear that the star operation preserves matchings between $(T_s, T_{s+1})$ and $(T_s^*, T_{s+1}^*)$, and $\ol{1}\in T_s$ does not contribute to the charge but $\ol{n} \in T_{s}^*$ always does. Thus it follows that $\lch_s(T^*)=\lch_s(T)+1$.

On the other hand, if $\lambda_{s-1}=\lambda_s$ then we take the standard activation ordering on both $T_{s-1}$ and $T^*_{s-1}$. Then it is also clear that the star operation is matching-preserving, but the element $\ol{k} \in T_{s-1}$ which is matched to $\ol{1}$ contributes to the charge whereas $\ol{k} \in T_{s-1}^*$ which is necessarily matched to $\ol{n}$ does not. In other words, we have $\lch_{s-1}(T^*)=\lch_{s-1}(T)-1$. Now the statements about the local charges on $t$-th and $(t-1)$-th row follow from applying this argument to $T^*$.

If $|t-s|=1$, then it suffices to consider the case when $t=s+1$ because the other case is proved by switching $T$ and $T^*$. Now if $t=s+1$, i.e. if $\ol{1} \in T_s$ and $\ol{n} \in T_{s+1}$, then clearly $\lch_j(T^*)=\lch_j(T)$ when $j \not \in \{s-1, s, s+1\}$. First suppose that $\lambda_{s-1} =\lambda_s$. If we pose the standard activation ordering on both $T_{s-1}$ and $T_{s-1}^*$, then the star operation is matching-preserving. Also, there exists $2\leq k \leq n-1$ such that $\ol{k}\in T_{s-1}$ is matched to $\ol{1} \in T_s$ but $\ol{k}\in T_{s-1}^*$ is matched to $\ol{n} \in T_{s}^*$. Thus it is clear that $\lch_{s-1}(T^*) = \lch_{s-1}(T)-1$. Now when $\lambda_{s+1}=\lambda_{s+2}$, then similar argument applies and one can show that $\lch_{s+1}(T^*) = \lch_{s+1}(T)-1$.

This time we assume that $\lambda_s=\lambda_{s+1}$ and prove that $\lch_{s}(T^*) = \lch_{s}(T)+2$. Note that either $\ol{2} \in T_{s+1}$ or $\ol{n-1} \in T_s$ for the star operation to be well-defined for $* \sim \ol{n}$. Here we only discuss the case when $\ol{2} \in T_{s+1}$, but the other case is proved similarly. We choose the activation ordering of $T_s$ and $T_s^*$ to be the standard one except that $\ol{n} \in T_s^*$ becomes the smallest element. Then $\ol{1}\in T_s$ is matched to $\ol{2} \in T_{s+1}$, whereas $\ol{n} \in T_{s}^*$ is matched to $\ol{1} \in T_{s+1}^*$. Besides these matchings, there is no difference between $T_s$ and $T_s^*$, whereas $T_{s+1}^*$ is obtained from $T_{s+1}$ by replacing $\ol{n}$ with $\ol{2}$. This causes the local charge of $T^*$ at row $s$ to be bigger by 1 than that of $T$. Together with considering the matching $(\ol{1}, \ol{2})$ in $T$ and $(\ol{n}, \ol{1})$ in $T^*$, we see that $\lch_{s}(T^*) = \lch_{s}(T)+2$ as desired.
\end{proof}

\begin{lem} \label{lem:knuthoffset} For $n \geq 3$, let $T, P, Q$ be tabloids of shape $\lambda \vdash n$ such that $T^*$ is well-defined for $*\sim \ol{i}$. If $\ol{i} \neq \ol{n}$, then we have $\offset_{P,T}=\offset_{P,T^*}$ and $\offset_{T,Q}=\offset_{T^*,Q}$. If $\ol{i}=\ol{n}$, then we have
\begin{align*}
\offset_{P,T^*}&=\offset_{P,T}+\vv{\iota}(T,1)-\vv{\iota}(T,n)-\vv{\delta}(T,1)+\vv{\delta}(T,n) \textup{ and}
\\\offset_{T^*,Q}&=\offset_{T,Q}-\vv{\iota}(T,1)+\vv{\iota}(T,n)+\vv{\delta}(T,1)-\vv{\delta}(T,n).
\end{align*}
\end{lem}
\begin{proof}  It follows from Lemma \ref{lem:knuthlch} and \cite[Theorem 5.10]{clp17}.
\end{proof}
%

\begin{prop} \label{prop:star}
For $n\geq 3$, suppose that $w \in \extSn$ satisfies $\Phi(w) = (P, Q, \offset_{P,Q}+\vv{\rho})$ for some tabloids $P, Q$ of shape $\lambda \vdash n$ and $\rev_\lambda(\vv{\rho}) \in \Dom(F_\lambda)$.
\begin{enumerate}[label=\arabic*)] 
\item If $Q^*$ is well-defined for $*\sim \ol{i} \neq \ol{n}$ then $\Phi(w^*) = (P, Q^*, \offset_{P,Q^*}+\vv{\rho})$. Similarly if $P^*$ is well-defined for $*\sim \ol{i} \neq \ol{n}$ then  $\Phi({}^*w) = (P^*, Q, \offset_{P^*,Q}+\vv{\rho})$.
\item If $Q^*$ is well-defined for $*\sim \ol{n}$ then $\Phi(w^*) = (P, Q^*, \offset_{P, Q^*}+\vv{\rho}+\vv{\delta}(Q, 1)-\vv{\delta}(Q, n))$. Similarly if $P^*$ is well-defined for $*\sim \ol{n}$ then $\Phi({}^*w) = (P^*, Q, \offset_{P^*, Q}+\vv{\rho}-\vv{\delta}(P, 1)+\vv{\delta}(P, n))$.
\end{enumerate}
In particular, if $\Phi(w) = (P, P, \vv{\rho})$ for some $P$ such that $P^*$ is well-defined then $\Phi({}^*w^*) = (P^*, P^*, \vv{\rho})$. (Here ${}^*w^*=({}^*w)^*={}^*(w^*)$.)
\end{prop}
\begin{proof} By symmetry, it suffices only to prove the statements for $w^*$ when $Q^*$ is well-defined. If $*\sim \ol{i} \neq \ol{n}$, then by \cite[Theorem 3.11]{clp17} we have $\Phi(w^*) = (P, Q^*, \offset_{P,Q}+\vv{\rho})=(P, Q^*, \offset_{P,Q^*}+\vv{\rho})$. Now if $* \sim \ol{n}$, then again by \cite[Theorem 3.11]{clp17} we have
\begin{align*}
\Phi(w^*) &=(P, Q^*, \offset_{P,Q}+\vv{\rho}+\vv{\iota}(Q, 1)-\vv{\iota}(Q, n))
\\&=(P, Q^*, \offset_{P,Q^*}+\vv{\rho}+\vv{\delta}(Q, 1)-\vv{\delta}(Q, n))
\end{align*}
by Lemma \ref{lem:knuthoffset}.
\end{proof}

\begin{example} \label{ex:star}
Let $n=9$ and $w = [-1,3,10,-5,14,-3,18,7,2]$ as in the Example \ref{ex:shift}. Recall that $\Phi(w) = (P, Q, \offset_{P,Q}+\vv{\rho})$, where $P, Q, \offset_{P,Q}$ are as in the Example \ref{ex:1},
and $\vv{\rho} = (0,1,2)$. For $* \sim \bar 9$ we have $w^* = [-7,3,10,-5,14,-3,18,7,8]$. Then one checks that $\Phi(w^*) = (P, Q^*, (0,0,0))$. Since for $P = \tableau[sY]{\ol{2}, \ol{4}, \ol{6} \\ \ol{3}, \ol{7}, \ol{8} \\ \ol{1}, \ol{5}, \ol{9}}$ and $Q^* = \tableau[sY]{\ol{3}, \ol{5}, \ol{7} \\ \ol{2}, \ol{8}, \ol{9} \\ \ol{1}, \ol{4}, \ol{6}}$ we have $\offset_{P,Q^*} = (0,-1,-2)$, the claim of Proposition \ref{prop:star} reduces to $$(0,0,0) = (0,-1,-2) + (0,1,2) + \vv{\delta}(Q, 1)-\vv{\delta}(Q, 9).$$ The latter is true since $\vv{\delta}(Q, 1) = \vv{\delta}(Q, 9) = (0,0,0)$.
\end{example}

\section{Distinguished involutions}
Here we study distinguished involutions defined in \cite{lus87:cell}. Originally, we say that $w \in \extSn$ is a distinguished involution if its Coxeter length equals $a(w)+2 \deg P_{id, w}$, where $a(w)$ is the value of Lusztig's $a$-function and $P_{id,w}$ is the Kazhdan-Lusztig polynomial attached to $(id, w)$. However, this definition is equivalent that $\ft_w$ is the unit element in $\cJ_{\lc^{-1} \cap \lc}$ where $\lc$ is the left cell containing $w$, see \ref{sec:setup}. Recall that $\cD$ is the set of such elements in $\extSn$. The main result in this section is as follows.
\begin{thm} \label{thm:dist} Suppose that $w\in \extSn$. Then $w \in \cD$ if and only if $\Phi(w) = (T, T, 0)$ for some $T$.
\end{thm}

Its proof consists of the following three lemmas.
\begin{lem} \label{lem:conn}
Let $\lc_1, \lc_2, \lc_3, \lc_4$ be left cells of $ \extSn$ contained in the same two-sided cell. Then,
\begin{enumerate}[label=\textup{(\alph*)}]
\item $\lc_1^{-1} \cap \lc_3$ can be obtained from $\lc_2^{-1}\cap \lc_4$ by applying (left and right) star operations and (left and right) multiplication by $\shift$. 
\item $\lc_1$ can be obtained from $\lc_2$ by applying right star operations and right multiplication by $\shift$.
\item $\lc_1^{-1} \cap \lc_1$ can be obtained from $\lc_2^{-1} \cap \lc_2$ by applying the map $w\mapsto {}^*w^*$ (where the left and right star operation corresponds to the same $\ol{i}$ for some $1\leq i \leq n$) and conjugation by $\shift$.
\end{enumerate}
\end{lem}
\begin{proof} It is a reformulation of \cite[Lemma 2.2.1, Corollary 2.2.2, Proposition 2.2.3]{xi02} based on the result of \cite{shi86}. Also note that (b) follows from (a) and (c) follows from (b).
\end{proof}
\begin{lem} \label{lem:distconn} If $w\in \cD$, then ${}^*w^* \in \cD$ (where the left and right star operation corresponds to the same $\ol{i}$ for some $1\leq i \leq n$) and $\shift w\shift^{-1} \in \cD$.
\end{lem}
\begin{proof} It directly follows from \cite[Proposition 1.4.6]{xi02}.
\end{proof}
\begin{lem} \label{lem:distSn} $\cD \cap \Sym_n$ consists of all the involutions in $\Sym_n$. In particular, $w_0^\lambda \in \cD$.
\end{lem}
\begin{proof} This follows from \cite[Corollary 1.9(d)]{lus87:cell} and the fact that every involution in $\Sym_n$ is distinguished.
\end{proof}

\begin{proof}[Proof of Theorem \ref{thm:dist}] By Lemma \ref{lem:conn}(c) and Lemma \ref{lem:distconn}, any distinguished involution can be obtained from another one in the same two-sided cell by applying $w\mapsto {}^*w^*$ and $w\mapsto \shift w \shift^{-1}$ several times. Therefore, by Proposition \ref{prop:shift} and \ref{prop:star}, it suffices to show that $\Phi(w) = (T, T, 0)$ for at least one distinguished involution in each two-sided cell. Now the result follows from the fact that $w_0^\lambda \in \tsc_\lambda\cap\cD$ by Lemma \ref{lem:distSn} and $\Phi(w_0^\lambda) = (\Ta_\lambda, \Ta_\lambda, 0)$ by Lemma \ref{lem:antican}.
\end{proof}

\begin{example}
Let $n=9$ and consider the tabloid $T = \tableau[sY]{\ol{2}, \ol{4}, \ol{6}, \ol{9} \\ \ol{3}, \ol{7}, \ol{8} \\ \ol{1}, \ol{5}}$. Applying the reverse AMBC construction to $(T, T, 0)$ we get the distinguished involution $[-3,5,3,7,2,10,4,8,9] \in \widetilde{\Sym_9}$. 
\end{example}

\section{Structure of asymptotic Hecke algebras attached to two-sided cells}\label{sec:jring}

Let $\tsc=\tsc_\lambda$ for some $\lambda \vdash n$ and recall the definition of $\vv{\varepsilon}: (\Ta_\lambda)^{-1} \cap \Ta_\lambda \rightarrow \Dom(F_\lambda)$ in Section \ref{sec:xibij}.
 A conjecture of Lusztig \cite{lus89:cell}, proved by Xi \cite{xi02} for $G=GL_n$, states that there exists an isomorphism between $\cJ_{\tsc_\lambda}$ and $\Mat_{\chi\times\chi} (\Rep(F_\lambda))$ where $\chi = \frac{n!}{\lambda_1!\lambda_2!\cdots}$. Furthermore, it restricts to an isomorphism $\cJ_{(\lca_\lambda)^{-1} \cap \lca_\lambda} \simeq \Rep(F_\lambda)$ which maps $\ft_w$ to $V(\vv{\varepsilon}(w))$. 
 
One problem of Xi's construction is that the isomorphism $\cJ_{\tsc_\lambda}\simeq \Mat_{\chi\times\chi} (\Rep(F_\lambda))$ is not canonical; it depends on the choice of the identification of each left cell in $\tsc_\lambda$ with $\lcc_\lambda$ using star operations and multiplication by $\shift$. Here we propose a more canonical construction of such an isomorphism using AMBC.

\begin{defn} For an element $w\in \extSn$ such that $\Phi(w) = (P, Q, \vv{\rho})$, we denote $\ft_w$ by $\ft(P, Q, \vv{\rho})$. (Here $\vv{\rho}$ is always dominant with respect to $(P,Q)$.) Also we define $\ft_{P,Q} \colonequals \ft(P, Q, \offset_{P,Q})$.
\end{defn}
Our claim is that there exists an isomorphism such that $\ft_{P,Q}$ corresponds to an ``elementary matrix''. The main result of this section is the following theorem.
\begin{thm} \label{thm:matisom} Let us label the left cells in $\tsc=\tsc_\lambda$ by $\lc_1, \lc_2, \ldots, \lc_\chi$. Then there exists an algebra isomorphism $\Upsilon=\Upsilon_\lambda\colon \cJ_{\tsc} \rightarrow \Mat_{\chi\times \chi} (\Rep(F_\lambda))$ such that if $w \in \lc_i^{-1} \cap \lc_j$ and $\Phi(w) = (P, Q, \offset_{P,Q}+\vv{\rho})$, then $\Upsilon(\ft_w)$ is the matrix whose $(i,j)$-entry is $V(\rev_\lambda(\vv{\rho}))$ and other entries are zero.
\end{thm}
It is clear that $\Upsilon$ gives a well-defined isomorphism of abelian groups. Therefore, the theorem is true if for any $\vv{\rho}, \vv{\rho}' \in \rev_\lambda(\Dom(F_\lambda))$ we have
\begin{equation*}
\ft(P, Q, \offset_{P, Q}+\vv{\rho})\cdot\ft(Q', R, \offset_{Q',R}+\vv{\rho}') = \delta_{Q, Q'} \sum_{\vv{\rho}''} m_{\rho, \rho', \rho''}\ft(P, R, \offset_{P,R}+\vv{\rho}'')
\end{equation*}
where $V(\rev_\lambda(\vv{\rho}))\otimes V(\rev_\lambda(\vv{\rho}')) \simeq \bigoplus_{\vv{\rho}''}V(\rev_\lambda(\vv{\rho}''))^{\oplus m_{\rho, \rho', \rho''}}.$ This equation holds when $Q\neq Q'$ by \cite[Corollary 1.9]{lus87:cell}. Thus we may assume that $Q=Q'$, i.e. we only need to show that
\begin{equation}\label{eq:mult}
\ft(P, Q, \offset_{P, Q}+\vv{\rho})\cdot\ft(Q, R, \offset_{Q,R}+\vv{\rho}') =\sum_{\vv{\rho}''} m_{\rho, \rho', \rho''}\ft(P, R, \offset_{P,R}+\vv{\rho}'').
\end{equation}

\begin{example}
Let $n=9$ and $w = [-1,3,10,-5,14,-3,18,7,2]$. Then $\Phi(w) = (P, Q, \offset_{P,Q}+\vv{\rho})$, where 
$P= \tableau[sY]{\ol{2}, \ol{4}, \ol{6} \\ \ol{3}, \ol{7}, \ol{8} \\ \ol{1}, \ol{5}, \ol{9}}$,
$Q= \tableau[sY]{\ol{3}, \ol{5}, \ol{7} \\ \ol{1}, \ol{2}, \ol{8} \\ \ol{4}, \ol{6}, \ol{9}}$, 
$\offset_{P,Q} = (0,-2,-1)$, and  $\vv{\rho} = (0,1,2)$. 
Let $w'= [-6,2,-4,15,18,-2,8,22,10]$. Then $\Phi(w') = (Q, R, \offset_{Q,R}+\vv{\rho'})$, where 
$R= \tableau[sY]{\ol{4}, \ol{5}, \ol{8} \\ \ol{2}, \ol{7}, \ol{9} \\ \ol{1}, \ol{3}, \ol{6}}$, 
$\offset_{Q,R} = (0,1,-1)$, and  $\vv{\rho'} = (0,0,2)$. Tensoring $GL_3$ representations with highest weights $(2,1,0)$ and $(2,0,0)$ we get 
$$V(2,1,0) \otimes V(2,0,0) = V(4,1,0) \oplus V(3,2,0) \oplus V(3,1,1) \oplus V(2,2,1).$$
Taking into account that $\offset_{P,R} = (0,-1,-2)$ and applying inverse AMBC we conclude that 
$$\ft_w \cdot \ft_{w'} = \ft_{[-7,3,-5,18,19,-3,7,23,8]} + \ft_{[-7,7,-5,14,18,-3,8,19,12]} + \ft_{[-5,3,-3,14,18,2,7,19,8]} + \ft_{[-5,7,-3,10,14,2,8,18,12]}.$$
For example, $(0,-1,-2)+(0,1,4) = (0,0,2)$ and  $\Phi^{-1}((P,R,(0,0,2))) = [-7,3,-5,18,19,-3,7,23,8]$.
On the other hand $\ft_{w'} \cdot \ft_{w} = 0$ since $R \not = P$.
\end{example}

The rest of this section is devoted to the proof of Equation (\ref{eq:mult}). We start with the following lemma.

\begin{lem} \label{lem:gamma}
For some $u, v \in \extSn$, suppose that $\ft_u \ft_{v} = \sum_{w\in \extSn} \gamma_{u,v,w^{-1}}\ft_w$.
\begin{enumerate}[label=\textup{(\arabic*)}]
\item If $v^*$ is well-defined for some $* \sim \ol{i}$, then $w^*$ is also well-defined when $\gamma_{u,v,w^{-1}}\neq 0$, and we have $\ft_u \ft_{v^*} = \sum_{w\in \extSn} \gamma_{u,v,w^{-1}}\ft_{w^*}$.
\item If ${}^*u$ is well-defined for some $* \sim \ol{i}$, then ${}^*w$ is also well-defined when $\gamma_{u,v,w^{-1}}\neq 0$, and we have $\ft_{{}^*u} \ft_{v} = \sum_{w\in \extSn} \gamma_{u,v,w^{-1}}\ft_{{}^*w}$.
\item Suppose that $\ft_u\ft_v\neq 0$. Then for some $*\sim \ol{i}$, $u^*$ is well-defined if and only if ${}^*v$ is well-defined. In this case we have $\ft_{u{}^*} \ft_{{}^*v} = \sum_{w\in \extSn} \gamma_{u,v,w^{-1}}\ft_{w}$.
\item For any $i,j,k \in \bZ$, we have $\ft_{\shift^iu\shift^{-j}}  \ft_{\shift^jv\shift^{-k}} = \sum_{w\in \extSn} \gamma_{u,v,w^{-1}}\ft_{\shift^i w \shift^{-k}}$.
\end{enumerate}
\end{lem}
\begin{proof} The first three statements follows from \cite[Theorem 1.6.2]{xi02}. The last one follows from the fact that $\ft_{\shift w}=\ft_{\shift}\ft_w$, $\ft_{w\shift}=\ft_w\ft_{\shift}$, and $\ft_{\shift^{-1}} = \ft_{\shift}^{-1}$.
\end{proof}

Now suppose that one of $\vv{\rho}, \vv{\rho}' \in \rev_\lambda(\Dom(F_\lambda))$ is determinantal. Then we know that 
$$V(\rev_\lambda(\vv{\rho}))\otimes V(\rev_\lambda(\vv{\rho}')) \simeq V(\rev_\lambda(\vv{\rho})+\rev_\lambda(\vv{\rho}')).$$
Here we prove (\ref{eq:mult}) in an analogous situation.

\begin{lem}\label{lem:determ} Suppose that one of $\vv{\rho}, \vv{\rho}' \in \rev_\lambda(\Dom(F_\lambda))$ is determinantal. Then 
\begin{gather*}
\ft(P, Q, \offset_{P,Q}+\vv{\rho})\cdot \ft(Q, R, \offset_{Q,R}+\vv{\rho}')=\ft(P, R, \offset_{P,R}+\vv{\rho}+\vv{\rho}').
\end{gather*}
In particular, we have $\ft_{P,Q}\ft_{Q,R}=\ft_{P,R}$.
\end{lem}
\begin{proof} For tabloids $P, Q, R$ of the same shape we denote by $\mathfrak{P}(P, Q, R)$ the statement that the equation above is true for the triple $(P, Q, R)$ and any $\vv{\rho}, \vv{\rho}'$ such that one of $\vv{\rho}, \vv{\rho}'$ is determinantal. The proof of this lemma consists of three steps.
\begin{enumerate}[label=\textup{\bfseries (\arabic*)}, leftmargin=0pt, itemindent=*]
\item $\mathfrak{P}(T, T, T)$ holds for any tabloid $T$. We already know that $\mathfrak{P}(\Ta_\lambda, \Ta_\lambda, \Ta_\lambda)$ holds by Theorem \ref{thm:xibij} and \cite[Theorem 8.2.1]{xi02}. Thus by Lemma \ref{lem:conn}, it suffices to show that 
$$\mathfrak{P}(T, T, T) \Rightarrow \mathfrak{P}(T^*, T^*, T^*), \mathfrak{P}(\shift(T), \shift(T), \shift(T))$$
where $T^*$ is well-defined for some $*\sim \ol{i}$. This follows from Proposition \ref{prop:shift}, Proposition \ref{prop:star}, and Lemma \ref{lem:gamma}.
\item $\mathfrak{P}(T, T, T')$ holds for any tabloid $T$ and $T'$ of the same shape. We know that it holds when $T=T'$ by the first step. Thus again by Lemma \ref{lem:conn}, it suffices to show that 
$$\mathfrak{P}(T, T, T') \Rightarrow \mathfrak{P}(T, T, T'^*), \mathfrak{P}(T,T, \shift(T')).$$
where $T'^*$ is well-defined for some $*\sim \ol{i}$. We again use Proposition \ref{prop:shift}, Proposition \ref{prop:star}, and Lemma \ref{lem:gamma}. First of all if $*\not\sim \ol{n}$, then from
$$\ft(T, T, \vv{\rho})\cdot \ft(T, T', \offset_{T,T'}+\vv{\rho}')=\ft(T, T', \offset_{T,T'}+\vv{\rho}+\vv{\rho}')$$
we have
$$\ft(T, T, \vv{\rho})\cdot \ft(T, T'^*, \offset_{T,T'^*}+\vv{\rho}')=\ft(T, T'^*, \offset_{T,T'^*}+\vv{\rho}+\vv{\rho}')$$
thus $\mathfrak{P}(T, T, T'^*)$ holds. Now if $*\sim \ol{n}$ then we have
\begin{gather*}
\ft(T, T, \vv{\rho})\cdot \ft(T, T'^*, \offset_{T,T'^*}+\vv{\rho}'+\vv{\delta}(T', 1)-\vv{\delta}(T',n))
\\=\ft(T, T'^*, \offset_{T,T'^*}+\vv{\rho}+\vv{\rho}'+\vv{\delta}(T', 1)-\vv{\delta}(T',n)).
\end{gather*}
Thus if we replace $\vv{\rho}'$ by $\vv{\rho}'-\vv{\delta}(T', 1)+\vv{\delta}(T',n)$ then we see that $\mathfrak{P}(T, T, T'^*)$ also holds in this case. Note that $\vv{\delta}(T', 1)$ and $\vv{\delta}(T',n)$ are determinantal, thus one of $\vv{\rho}, \vv{\rho}'$ is determinantal if and only if one of $\vv{\rho}, \vv{\rho}'-\vv{\delta}(T', 1)+\vv{\delta}(T',n)$ is determinantal. Finally, we also have
$$\ft(T, T, \vv{\rho})\cdot \ft(T, \shift(T'), \offset_{T,\shift(T')}+\vv{\rho}'-\vv{\delta}(T', n))=\ft(T, \shift(T'), \offset_{T,\shift(T')}+\vv{\rho}+\vv{\rho}'-\vv{\delta}(T', n)).$$
Thus if we replace $\vv{\rho}'$ by $\vv{\rho}'+\vv{\delta}(T', n)$ then it follows that $ \mathfrak{P}(T,T, \shift(T'))$ holds.
\item $\mathfrak{P}(P, Q, R)$ holds for any tabloids $P, Q, R$ of the same shape. Similarly to above, it suffices to show that
$$\mathfrak{P}(P,Q, R) \Rightarrow \mathfrak{P}(P, Q^*, R), \mathfrak{P}(P, \shift(Q), R)$$
where $Q^*$ is well-defined for some $*\sim \ol{i}$. We again use Proposition \ref{prop:shift}, Proposition \ref{prop:star}, and Lemma \ref{lem:gamma}. If $*\not\sim \ol{n}$, then from
$$\ft(P, Q, \offset_{P,Q}+\vv{\rho})\cdot \ft(Q, R, \offset_{Q,R}+\vv{\rho}')=\ft(P, R, \offset_{P,R}+\vv{\rho}+\vv{\rho}')$$
we have
$$\ft(P, Q^*, \offset_{P,Q^*}+\vv{\rho})\cdot \ft(Q^*, R, \offset_{Q^*,R}+\vv{\rho}')=\ft(P, R, \offset_{P,R}+\vv{\rho}+\vv{\rho}')$$
which implies $\mathfrak{P}(P, Q^*, R)$. On the other hand, if $*\sim \ol{n}$ then we have
\begin{gather*}
\ft(P, Q^*, \offset_{P,Q^*}+\vv{\rho}+\vv{\delta}(Q,1)-\vv{\delta}(Q,n) )\cdot \ft(Q^*, R, \offset_{Q^*,R}+\vv{\rho}'-\vv{\delta}(Q, 1)+\vv{\delta}(Q, n))
\\=\ft(P, R, \offset_{P,R}+\vv{\rho}+\vv{\rho}').
\end{gather*}
Thus by replacing $\vv{\rho}$ and $\vv{\rho}'$ with $\vv{\rho}-\vv{\delta}(Q,1)+\vv{\delta}(Q,n)$ and $\vv{\rho}'+\vv{\delta}(Q, 1)-\vv{\delta}(Q, n)$, respectively, we see that $\mathfrak{P}(P, Q^*, R)$ holds. Finally, we also have
\begin{gather*}
\ft(P, \shift(Q), \offset_{P,\shift(Q)}+\vv{\rho}-\vv{\delta}(Q, n))\cdot \ft(\shift(Q), R, \offset_{\shift(Q),R}+\vv{\rho}'+\vv{\delta}(Q, n))
\\=\ft(P, R, \offset_{P,R}+\vv{\rho}+\vv{\rho}').
\end{gather*}
Thus by replacing $\vv{\rho}$ and $\vv{\rho}'$ with $\vv{\rho}+\vv{\delta}(Q,n)$ and $\vv{\rho}'-\vv{\delta}(Q, n)$, respectively, we see that $\mathfrak{P}(P, \shift(Q), R)$ holds. 
\end{enumerate}
The lemma is proved.
\end{proof}

%

Now we prove (\ref{eq:mult}) in the case when $P=Q=R$ without any restriction on $\vv{\rho}$ or $\vv{\rho}'$
\begin{lem} \label{lem:weight} For any tabloid $T$ we have
$$\ft(T,T, \vv{\rho})\cdot \ft(T, T,\vv{\rho}') = \sum_{\vv{\rho}''} m_{\rho, \rho', \rho''}\ft(T, T,\vv{\rho}'')$$
where $V(\rev_\lambda(\vv{\rho}))\otimes V(\rev_\lambda(\vv{\rho}')) \simeq \bigoplus_{\vv{\rho}''}V(\rev_\lambda(\vv{\rho}''))^{\oplus m_{\rho, \rho', \rho''}}.$
\end{lem}
\begin{proof} By Theorems \ref{thm:xibij} and \cite[Theorem 8.2.1]{xi02}, at least we know that the statement holds when $T=\Ta$. Now it follows from Proposition \ref{prop:shift}, Proposition \ref{prop:star}, and Lemma \ref{lem:gamma}.
\end{proof}


\begin{proof}[Proof of Theorem \ref{thm:matisom}]
Here we prove that equation (\ref{eq:mult}) holds. By Lemma \ref{lem:determ} we have
\begin{align*}
&\ft(P, Q, \offset_{P, Q}+\vv{\rho})\cdot\ft(Q, R, \offset_{Q,R}+\vv{\rho}')
\\&=\ft(P, P, \vv{\rho})\cdot\ft(P, Q, \offset_{P, Q})\cdot\ft(Q, Q, \vv{\rho}')\cdot \ft(Q, R, \offset_{Q,R})
\\&=\ft(P, P, \vv{\rho})\cdot\ft(P, Q, \offset_{P, Q}+ \vv{\rho}')\cdot \ft(Q, R, \offset_{Q,R})
\\&=\ft(P, P, \vv{\rho})\cdot\ft(P, P, \vv{\rho}')\cdot \ft(P, Q, \offset_{P, Q})\cdot \ft(Q, R, \offset_{Q,R})
\\&=\ft(P, P, \vv{\rho})\cdot\ft(P, P, \vv{\rho}')\cdot \ft(P, R,  \offset_{P,R})
\end{align*}
which is the same as
$$\sum_{\vv{\rho}''} m_{\rho, \rho', \rho''}\ft(P, P, \vv{\rho}'')\cdot \ft(P, R, \offset_{P,R})$$
by Lemma \ref{lem:weight}. Again by Lemma \ref{lem:determ} it is equal to
$$\sum_{\vv{\rho}''} m_{\rho, \rho', \rho''}\ft(P,  R, \offset_{P,R}+\vv{\rho}''),$$
which is what we want to prove.
\end{proof}

\section{Equality of two Lusztig-Vogan bijections} \label{sec:equality}
Let $G$ be a reductive group over $\bC$. In \cite{lus89:cell}, Lusztig defined a conjectural bijection between $\Dom(G)$ and
$$\mathbf{O}\colonequals \{ (N, \rho) \mid N \in \Lie G,  \vv{\rho} \in \Dom(F_N)\}/G$$
where $F_N$ is the reductive part of the stabilizer $Z_G(N)$ of $N \in \Lie G$ and the quotient by $G$ is with respect to adjoint action on $\Lie G$. Later, a similar bijection wes also independently conjectured by Vogan \cite{vog91}. We call such a correspondence the Lusztig-Vogan bijection.

In this section we focus on the case when $G=GL_n(\bC)$. Then we may and shall identify $\mathbf{O}$ with $\bigsqcup_{\lambda \vdash n} \Dom(F_\lambda)$. Let us describe some properties of this bijection as described in \cite[10.8]{lus87:cell}. Suppose that $\Theta :  \Dom(GL_n)  \rightarrow \mathbf{O}=\bigsqcup_{\lambda \vdash n} \Dom(F_\lambda)$ is such a bijection. We regard $\Dom(GL_n)$ as a subset of $\extSn$ by identifying $\vv{\mu} = (\mu_1, \mu_2, \ldots, \mu_n)$ with $w_{\vv{\mu}}\colonequals [n\mu_1+1, n\mu_2+2, \ldots, n\mu_n+n] \in \extSn$. Then we have a decomposition of $\extSn$ into double $\Sym_n$ cosets
$$\extSn = \bigsqcup_{\vv{\mu} \in \Dom(GL_n)} \Sym_n w_{\vv{\mu}} \Sym_n.$$
In particular, there exists an one-to-one correspondence between $\Dom(GL_n)$ and double $\Sym_n$-cosets of $\extSn$, i.e. $\Dom(GL_n) \simeq \Sym_n\backslash \extSn/\Sym_n$.

Recall that $\extSn_f$ is defined to be the set of minimal length elements in each left $\Sym_n$-coset. We define ${}_f\extSn_f \colonequals (\extSn_f)^{-1} \cap \extSn_f$, i.e. the set of minimal length elements in each double $\Sym_n$-coset. Then clearly we have a bijection between ${}_f\extSn_f$ and double $\Sym_n$-cosets in $\extSn$, i.e. ${}_f\extSn_f \simeq \Sym_n\backslash \extSn/\Sym_n$. On the other hand, from the result of Section \ref{sec:cancells} we have ${}_f\extSn_f=\bigsqcup_{\lambda \vdash n} (\lcc_{\lambda})^{-1}\cap \lcc_{\lambda}.$
Therefore, we have a chain of bijections
$$\tilde{\Theta}\colon\bigsqcup_{\lambda \vdash n} (\lcc_\lambda)^{-1}\cap \lcc_\lambda = {}_f\extSn_f \rightarrow  \Sym_n \backslash \extSn / \Sym_n\rightarrow \Dom(GL_n) \xrightarrow{\Theta}  \bigsqcup_{\lambda \vdash n} \Dom(F_\lambda).$$
Now \cite[10.8]{lus87:cell} asserts that $\tilde{\Theta}$ should satisfy that $\tilde{\Theta}((\lcc_\lambda)^{-1}\cap \lcc_\lambda) = \Dom(F_{\lambda})$, and furthermore it induces an algebra isomorphism
$$\cJ_{(\lcc_\lambda)^{-1}\cap \lcc_\lambda}\simeq \Rep(F_\lambda): \ft_w \mapsto V(\tilde{\Theta}(w)).$$

We already have one candidate deduced from the result of Section \ref{sec:jring}. Indeed, recall that there is an algebra isomorphism $\Upsilon_\lambda\colon \cJ_{\tsc} \rightarrow \Mat_{\chi\times \chi} (\Rep(F_\lambda))$
which restricts to an isomorphism $\Upsilon_\lambda\colon \cJ_{(\lcc_\lambda)^{-1}\cap \lcc_\lambda} \rightarrow \Rep(F_\lambda).$ It sends $\ft_w$ for $w\in (\lcc_\lambda)^{-1}\cap \lcc_\lambda$ to $\rev_\lambda(\vv{\rho}(w)) \in \Dom(F_\lambda)$. Let us define
$$\tilde{\Theta}_1:\bigsqcup_{\lambda\vdash n} (\lcc_\lambda)^{-1}\cap \lcc_\lambda  \rightarrow  \bigsqcup_{\lambda \vdash n} \Dom(F_\lambda)$$
to be the disjoint union of such bijections. Then it is clear that $\tilde{\Theta}_1$ satisfies the properties of $\tilde{\Theta}$ above. In other words, the composition
$$\Theta_1 \colon \Dom(GL_n) \simeq \Sym_n \backslash \extSn / \Sym_n \simeq  \bigsqcup_{\lambda \vdash n} (\lcc_\lambda)^{-1}\cap \lcc_\lambda \xrightarrow{\tilde{\Theta}_1} \bigsqcup_{\lambda \vdash n} \Dom(F_\lambda) = \mathbf{O}$$
is a Lusztig-Vogan bijection.

There is another realization of such a bijection studied by Achar, Bezrukavnikov, Ostrik, and Rush. First of all, for a reductive group $G$ over $\bC$, Ostrik established such a bijection using a $G\times \bC^\times$-equvariant $K$-theory of the nilpotent cone of $\Lie G$ \cite{ost00}. Also, Bezrukavnikov defined two bijections in terms of a bounded derived category of $G$-equivariant coherent sheaves on the nilpotent cone of $\Lie G$ \cite{bez03} and a tensor category attached to each two-sided cell in the affine Weyl group of $G$ \cite{bez04}. Later, it is proved that they are indeed the same as one another \cite{bez09}.

For $G=GL_n(\bC)$, Achar also described such a bijection in a combinatorial way \cite{ach01} and proved that this bijection is the same as the one defined by Bezrukavnikov and Ostrik \cite{ach04}. Later, his method is simplified by Rush \cite{rus17:thesis, rus17}. From now on we denote this bijection by $\Theta_2 : \Dom(GL_n(\bC)) \rightarrow \mathbf{O}$.

Note that the bijection $\tilde{\Theta}_2 : \bigsqcup_{\lambda\vdash n} (\lcc_\lambda)^{-1}\cap \lcc_\lambda  \rightarrow  \bigsqcup_{\lambda \vdash n} \Dom(F_\lambda)$ induced from $\Theta_2$ satisfies the properties that $\tilde{\Theta}$ above possesses. Indeed, $\tilde{\Theta}_2$ is derived from the equivalence of two tensor categories $\textup{Rep}(Z_\tsc) \rightarrow \mathcal{A}^f_\tsc$ as in \cite[1.7]{bez09}. (Here $Z_\tsc$ is the stabilizer of a nilpotent element $N\in \Lie GL_n(\bC)$ whose orbit corresponds to the two-sided cell $\tsc$.) If we take the $K$-theory of this equivalence, then we have an algebra isomorphism $\Rep(F_\tsc) \simeq \Rep(Z_\tsc) \rightarrow \cJ_{(\lcc_\tsc)^{-1}\cap \lcc_\tsc}$ which sends $\vv{\rho} = \tilde{\Theta}_2(w)$ to $\ft_w$ for $w\in(\lcc_\tsc)^{-1}\cap \lcc_\tsc$, from which the claim follows. (Since $F_\tsc$ is the reductive part of $Z_\tsc$, we may identify $\Rep(F_\tsc)$ and $\Rep(Z_\tsc)$ canonically.)

We claim that these two realizations of the Lusztig-Vogan bijection coincide, i.e.
\begin{thm}\label{thm:lvbij} We have $\Theta_1=\Theta_2$, or equivalently $\tilde{\Theta}_1 = \tilde{\Theta}_2$.
\end{thm}

Let us show how this can be used to compute Lusztig-Vogan bijection in practice. 

\begin{example}
 Let $n=7$ and take a dominant weight $\vv{\mu} = (5,1,1,1,-2,-2,-2) \in \Dom(GL_7)$. Then $w_{\vv{\mu}} = [36,9,10,11,-9,-8,-7]$. Rearranging this window in increasing order we find the element
$[-9,-8,-7,9,10,11,36] \in {}_f \widetilde{\Sym_7}_f$ in the same double $\Sym_7$-coset $\Sym_7 \backslash \widetilde{\Sym_7}/ \Sym_7$ as $w_{\vv{\mu}}$. Applying to it the AMBC construction we get $(T, T, (-2,1,3))$, where 
$T = \tableau[sY]{\ol{5}, \ol{6}, \ol{7} \\ \ol{2}, \ol{3}, \ol{4} \\ \ol{1}} = \Tc_{(3,3,1)}$ is a canonical tableau, as expected.
We conclude that the Lusztig-Vogan bijection maps $\vv{\mu} = (5,1,1,1,-2,-2,-2)$ to the pair $(\lambda, \vv{\rho})$, where $\lambda = (3,3,1)$ is the shape of $T$ and $\rev_\lambda(\vv{\rho}) = (1,-2,3) \in \Dom(F_{\lambda})$. (Note that $\offset_{T, T}=0$.)
\end{example}

\begin{example}
Let $n=7$. Take a nilpotent orbit corresponding to $\lambda = (2,2,1,1,1)$ and let $\vv{\rho} = (0,0,1,0,-1) \in \Dom(F_{\lambda})$. We wish to apply the inverse Lusztig-Vogan bijection to the pair $(\lambda, \vv{\rho})$. We have $\Tc_{(2,2,1,1,1)} = \tableau[sY]{\ol{6}, \ol{7} \\ \ol{4}, \ol{5} \\ \ol{3} \\ \ol{2} \\ \ol{1}}$, and applying the inverse AMBC construction to the triple 
 $(\Tc_{(2,2,1,1,1)}, \Tc_{(2,2,1,1,1)}, (0,0,-1,0,1))$ we get $[-28,-8,-2,4,10,16,36] \in {}_f \widetilde{\Sym_7}_f$. Rearranging this window in the order increasing modulo $7$ we get $w = [36,16,10,4,-2,-8,-28]$. This happens to be $w_{\vv{\mu}}$ for $\vv{\mu} = (5,2,1,0,-1,-2,-5) \in \Dom(GL_7)$, which is the desired answer.  
\end{example}

\section{Proof of Theorem \ref{thm:lvbij}}
This section is devoted to the proof of Theorem \ref{thm:lvbij}. We start with the following lemma.
\begin{lem} \label{lem:Jringgen}For $\lambda =(1^{m_1}2^{m_2}\cdots)\vdash n$, we identify $\Dom(F_\lambda) = \prod_{i\geq 1} \Dom(GL_{m_i})$ and define $\vv{\rho}_\lambda(r,s)=(\rho_i)_{i\geq 1} \in \Dom(F_\lambda)$ where each $\rho_i \in \Dom(GL_{m_i})$ is set to be $\rho_i=0$ if $i\neq r$ and $\rho_r = (s,0,\ldots, 0)$. Set
$$Y_\lambda\colonequals \{0\}\sqcup \left\{ \vv{\rho}_\lambda(r,s) \in  \Dom(F_\lambda)  \mid r \in \lambda, s\geq 1\right\}.$$
If $\Theta_1^{-1}$ and $\Theta_2^{-1}$ coincide on $\bigsqcup_{\lambda \vdash n} Y_\lambda \subset \mathbf{O}$, then $\Theta_1=\Theta_2$.
\end{lem}
\begin{proof} Let $\Lambda_m$ to be the ring of symmetric Laurent polynomials of $m$ variables. Then $\Lambda_m = \bZ[e_1, e_2, \ldots, e_{m-1}, e_m^{\pm 1}]$ where $e_a=\sum_{i_1<i_2<\cdots<i_a}x_{i_1}x_{i_2}\cdots x_{i_a}$ are elementary symmetric functions. Thus an algebra endomorphism of $\Lambda_m$ stabilizing $e_1, e_2, \ldots, e_{m-1}, e_m$ is the identity on $\Lambda_m$. Since $\bZ[e_1, e_2, \ldots, e_{m-1}, e_m] = \bZ[h_1, h_2, \ldots, h_{m-1}, h_m]$ where $h_a=\sum_{i_1\leq i_2\leq \cdots\leq i_a}x_{i_1}x_{i_2}\cdots x_{i_a}$ are complete homogeneous symmetric functions,  an algebra endomorphism of $\Lambda_m$ stabilizing $h_1, h_2, \ldots, h_{m-1}, h_m$ is also the identity on $\Lambda_m$.

Since there exists a ring isomorphism $\Rep(GL_m) \simeq \Lambda_m$ which sends $V(s, 0, 0, \ldots)$ to $h_s$ (see e.g. \cite[Chapter 7: Appendix 2]{sta86}), if we identify $\Rep(F_\lambda)$ with $\prod_{i\geq 1} \Rep(GL_{m_i})\simeq \prod_{i\geq 1} \Lambda_{m_i}$ then an algebra endomorphism of $\Rep(F_\lambda)$ stabilizing each element in $Y_\lambda$ is the identity. Now the lemma follows from the fact that $\Rep(F_\lambda) \rightarrow \Rep(F_\lambda): V(\vv{\rho}) \mapsto V((\Theta_2 \circ \Theta_1^{-1})(\vv{\rho}))$ is an algebra automorphism by the argument in the previous section.
\end{proof}

\begin{rmk} Indeed, the proof is still valid if we replace $Y_\lambda$ by $\left\{ \vv{\rho}_\lambda(r,s) \in  \Dom(F_\lambda)  \mid r \in \lambda, 1\leq s \leq m_r\right\}.$
\end{rmk}

For a partition $\lambda$, we define $\mathfrak{W}_\lambda(0)$ be the Young tableau of shape $\lambda$ whose $i$-th column is filled with $\lambda_i'-1, \lambda_i'-3, \ldots, 3-\lambda_i', 1-\lambda_i'$ from top to bottom. For example, we have
\ytableausetup{nosmalltableaux}
$$\mathfrak{W}_{(3, 3, 2, 2, 1)}(0) = \begin{ytableau}
4 & 3 & 1
\\2 & 1 & -1
\\0 & -1 
\\-2 & -3
\\-4
\end{ytableau}.
$$
Also, when $i \in \lambda$ we define the Young tableau $\mathfrak{W}_\lambda({ r, s})$ of shape $\lambda$ as follows. The entries of $\mathfrak{W}_\lambda({ r, s})$ are the same as those of $\mathfrak{W}_\lambda(0)$ except $(1, i)$-entries for $1\leq i \leq r$. Let us write $a_i, b_i$ to be the $(1, i)$-th entry of $\mathfrak{W}_\lambda(0)$ and $\mathfrak{W}_\lambda({r, s})$, respectively, for $1\leq i \leq r$. Then $b_i$ are uniquely determined by the following conditions.
\begin{enumerate}[label=$\bullet$]
\item $a_i \leq b_i$ for $1\leq i \leq r$.
\item $b_1\geq b_2\geq \cdots \geq b_r$.
\item $\sum_{i=1}^r b_i = s+\sum_{i=1}^r a_i$.
\item $\sum_{i=1}^r b_i^2$ is the minimum among the choice of $b_i$ satisfying the properties above.
\end{enumerate}
For example, we have
\ytableausetup{nosmalltableaux}
$$\mathfrak{W}_{(3, 3, 2, 2, 1)}(2,5) = \begin{ytableau}
6 & 6 & 1
\\2 & 1 & -1
\\0 & -1 
\\-2 & -3
\\-4
\end{ytableau}.
$$
(Note that only (1,1)- and (1,2)-entries are different.)

We define $\vv{\mu}_\lambda(0)$ (resp. $\vv{\mu}_\lambda(r,s)$) to be a dominant weight of $GL_n$ consisting of entries of $\mathfrak{W}_\lambda(0)$ (resp. $\mathfrak{W}_{\lambda}(r,s)$). For example we have $\vv{\mu}_{(3, 3, 2, 2, 1)}(0) = (4,3,2,1,1,0,-1,-1,-2,-3,-4)$ and $\vv{\mu}_{(3, 3, 2, 2, 1)}(2,5)= (6,6,2,1,1,0,-1,-1,-2,-3,-4)$.

\begin{lem} \label{lem:theta2cal} $\Theta_2(\vv{\mu}_\lambda(0)) = 0$ and $\Theta_2(\vv{\mu}_\lambda(r,s))=\vv{\rho}_\lambda(r,s)$.
\end{lem}
\begin{proof} We briefly review some properties of $\Theta_2$ following \cite{rus17}. For $\lambda \vdash n$ and $1\leq j \leq \lambda_1$, it is clear that $\sum_{i\geq j} m_i = \lambda_j'$. Thus we may consider a standard block diagonal embedding $\zeta_j\colon\prod_{i\geq j} GL_{m_i} \rightarrow GL_{\lambda_j'}$. Let $\zeta=\prod_{j} \zeta_j \colonequals \prod_{1\leq j\leq \lambda_1}(\prod_{i\geq j} GL_{m_i} \rightarrow GL_{\lambda_j'})$ be a product of such morphisms. On the other hand, we also define $\xi_i:  GL_{m_i} \rightarrow  (GL_{m_i})^i$ to be the diagonal embedding, and let $\xi=\prod_{i\in \lambda} \xi_i : \prod_{i\in \lambda}(GL_{m_i} \rightarrow  (GL_{m_i})^i)$ be their product. Then since $\prod_{i\in \lambda}(GL_{m_i})^i = \prod_{i\in \lambda}(\prod_{i\geq j}GL_{m_i}) =  \prod_{1\leq j\leq \lambda_1}(\prod_{i\geq j} GL_{m_i})$, the composition $\zeta \circ \xi : \prod_{i\in \lambda}GL_{m_i} \rightarrow \prod_{1\leq j\leq \lambda_1}GL_{\lambda_j'}$ is well-defined. If we let $L_\lambda$ to be $\prod_{1\leq j\leq \lambda_1}GL_{\lambda_j'}$, then $\zeta \circ \xi$ defines a morphism $ F_\lambda\simeq  \prod_{i\in \lambda}GL_{m_i}  \rightarrow L_\lambda$. For example, if $\lambda=(4,4,2,1,1)$ then $\zeta \circ \xi: GL_{2}\times GL_1 \times GL_2 \rightarrow GL_5 \times GL_3 \times GL_2\times GL_2$ is defined by
$$(A, B, C) \mapsto \left( \begin{pmatrix}A & 0 &0\\ 0&B& 0\\0&0&C\end{pmatrix}, \begin{pmatrix}A & 0\\ 0&B\end{pmatrix} , A, A \right).$$


Let us investigate $(\zeta\circ\xi)^*: \Rep(L_\lambda) \rightarrow \Rep(F_\lambda)$ in terms of Laurent symmetric functions. Recall that $\Rep(GL_m) \simeq \Lambda_m$ where $\Lambda_m$ is the ring of symmetric Laurent polynomials of $m$ variables. Therefore we may identify $\Rep(F_\lambda)$ with $\bigotimes_i \Lambda(\vv{x}_i)$ where $\Lambda(\vv{x}_i)$ is the ring of Laurent symmetric functions with variables $\vv{x}_i = (x_{i1}, x_{i2}, \ldots, x_{im_i})$. Similarly, we identify $\Rep(L_\lambda)$ with $\bigotimes_j \Lambda(\vv{y}_j)$ where $\Lambda(\vv{y}_j)$ is the ring of Laurent symmetric functions with variables $\vv{y}_j=(y_{j1}, y_{j2}, \ldots, y_{j\lambda_j'})$. Then direct calculation shows that $(\zeta\circ\xi)^*$ maps $f(\vv{y}_j) \in \bigotimes_j \Lambda(\vv{y}_j)$ to $f(\vv{x}_{\geq j})$ where $\vv{x}_{\geq j}$ is the union of $\vv{x}_i$ for $i\geq j$. Furthermore, if $P\in  \bigotimes_j \Lambda(\vv{y}_j)$ is homogeneous in $\cup_j\vv{y}_j$ then the image $(\zeta\circ\xi)^*(P)$ is also homogeneous in $\cup_i \vv{x}_i$ and $(\zeta\circ\xi)^*$ is degree-preserving. For example, if $\lambda=(4,4,2,1,1)$, then
$$(\zeta\circ\xi)^*( f(y_{21},y_{22},y_{23})) = f(x_{11},x_{12},x_{21}).$$

Now suppose that we are given $\vv{\rho} \in \Dom(F_\lambda)$. For $\vv{\nu} \in \Dom(GL_n)$, defined $\vv{\nu}' \in \Dom(L_\lambda)$ which satisfies the following conditions.
\begin{enumerate}[label=(\alph*)]
\item As multisets, $\vv{\nu}$ is the same as the union of parts in $\vv{\nu}'$.
\item Let $\Delta_\lambda \in \Dom(L_\lambda)$ be the half sum of positive roots of $L_\lambda$. Then $\vv{\nu}' - 2\Delta_\lambda$ is dominant.
\item Let $V(\vv{\nu}' - 2\Delta_\lambda)$ be the irreducible representation of $L_\lambda$ of highest weight $\vv{\nu}' - 2\Delta_\lambda$. Then $V(\vv{\rho})$ appears as an irreducible constituent of $(\zeta\circ\xi)^*(V(\vv{\nu}' - 2\Delta_\lambda))$, i.e. the restriction of $V(\vv{\nu}' - 2\Delta_\lambda)$ to $F_\lambda$ under $\zeta\circ\xi$.
\end{enumerate}
Let $X_{\vv{\rho}}= \{ \vv{\nu} \in \Dom(GL_n) \mid \textup{such }\vv{\nu}' \textup{ exists}\}$.
Then according to \cite{ach04} and \cite{rus17}, we have that $\Theta_2(\vv{\mu})=\vv{\rho}$ if and only if $\vv{\mu} \in X_{\vv{\rho}}$ and $||\vv{\mu}|| = \min \{||\vv{\nu}|| \mid {\vv{\nu} \in X_{\vv{\rho}}}\}$. Also in this case $\vv{\mu}'$ is uniquely determined.

We set $\vv{\mu}'=\vv{\alpha}+2\Delta_\lambda$ where $\vv{\alpha} =(\alpha_j)_{1\leq j \leq \lambda_1}$ is such that $\alpha_j = (\alpha_{j1}, \alpha_{j2}, \ldots, \alpha_{j\lambda_j'}) \in \Dom(GL_{\lambda_j'}).$ Thus we have $\alpha_{j1}\geq \alpha_{j2}\geq \cdots\geq \alpha_{j\lambda_j'}$ by (b). Under the identification $\Rep(F_\lambda) \simeq \bigotimes_i \Lambda(\vv{x}_i)$ and $\Rep(L_\lambda) \simeq \bigotimes_j \Lambda(\vv{y}_j)$, the irreducible representation $V(\vv{\alpha}) \in \Rep(L_\lambda)$ corresponds to $\prod_js_{\alpha_j}(\vv{y}_j) \in \bigotimes_j \Lambda(\vv{y}_j)$ where $s_{\alpha_j}$ is the Schur function corresponding to $\alpha_j$. Therefore, $(\zeta\circ\xi)^* V(\vv{\alpha})\in \Rep(F_\lambda)$ is identified with $\prod_js_{\alpha_j}(\vv{x}_{\geq j}) \in \bigotimes_i \Lambda(\vv{x}_i)$.

We first suppose that $\Theta_2(\vv{\mu}) = 0$. Then (c) implies that $\prod_j s_{\alpha_j}(\vv{x}_{\geq j}) \in \bigotimes_i \Lambda(\vv{x}_i)$ has a nonzero constant term. Let $b$ be the smallest part in $\lambda$, then we have a factorization $\prod_j s_{\alpha_j}(\vv{x}_{\geq j}) = \prod_{j\leq b} s_{\alpha_{j}}(\vv{x}_{\geq j})  \prod_{j> b} s_{\alpha_j}(\vv{x}_{\geq j})$. Note that the variables $\vv{x}_b$ does not appear  in $\prod_{j> b} s_{\alpha_j}(\vv{x}_{\geq j})$. Thus for $\prod_j s_{\alpha_j}(\vv{x}_{\geq j})$ to have a constant term, we also have that $\vv{x}_b$ does not appear in $\prod_{j\leq b} s_{\alpha_{j}}(\vv{x}_{\geq j})$. Now direct calculation shows that it is only possible when there exists $\beta_j \in \bZ$ for $1\leq j \leq b$ such that $\alpha_j = (\beta_j, \beta_j, \ldots, \beta_j)$ and $\sum_{j=1}^b \beta_j=0$. Then $ \prod_{j\leq b} s_{\alpha_{j}}(\vv{x}_{\geq j})=1$ and $\prod_j s_{\alpha_j}(\vv{x}_{\geq j}) =   \prod_{j> b} s_{\alpha_j}(\vv{x}_{\geq j})$.

We iterate the above argument by enlarging $b$ to cover all the parts in $\lambda$ and conclude that there exists $\beta_j\in \bZ$ for $1\leq j \leq \lambda_1$ such that $\alpha_j=(\beta_j, \beta_j, \ldots, \beta_j)$ for all $j$ and $\sum_{j=1}^k \beta_j = 0$ whenever $k \in \lambda$. In this case we have $\prod_j s_{\alpha_j}(\vv{x}_{\geq j})=1$, thus the assumption is satisfied.
On the other hand, if we let $(\Delta_\lambda)_j$ be the component of $\Delta_\lambda$ corresponding to $GL_{\lambda_j'}$, then from the definition we have
$$2(\Delta_\lambda)_j=(\lambda_j'-1, \lambda_j'-3, \ldots, 3-\lambda_j', 1-\lambda_j').$$
Therefore in order to find $\vv{\mu}$, we need to minimize
$$\sum_{j} ((\beta_{j}+\lambda_j'-1)^2+(\beta_{j}+\lambda_j'-3)^2+\cdots +(\beta_{j}+1-\lambda_j')^2)$$
subject to the condition $\sum_{j=1}^k \beta_j = 0$ whenever $k \in \lambda$.
Now it is easy to show that the minimum is achieved when $\vv{\alpha}=0$, i.e. $\vv{\mu}=2\Delta_\lambda$. 
Thus it follows that $\vv{\mu}=\vv{\mu}_\lambda(0)$ by the definition of $\mathfrak{W}_\lambda(0)$.

We proceed to the case when $\Theta_2(\vv{\mu}) = \vv{\rho}_\lambda(r,s)$ and argue similarly to above. The condition (c) implies that $\br{\prod_js_{\alpha_j}(\vv{x}_{\geq j}), s_s(\vv{x}_r)} \neq 0$ where $\br{\ ,\ }$ is the pairing on $ \bigotimes_k \Lambda(\vv{x}_k)$ induced from the standard bilinear form on the ring of (Laurent) symmetric functions. In this case it is not hard to show that there exists $\beta_j\in \bZ$ for $1\leq j \leq \lambda_1$ and $\gamma_j\in \bN$ for $1\leq j \leq r$ such that
$$\alpha_j = \left\{\begin{aligned} &(\beta_j+\gamma_j, \beta_j, \ldots, \beta_j) & \textup{ if } 1\leq j \leq r,
\\&(\beta_j, \beta_j, \ldots, \beta_j) & \textup{ if } r<j\leq \lambda_1,
\end{aligned}\right.
$$
$\sum_{j=1}^k \beta_j = 0$ whenever $k\in \lambda$, and $\sum_{j=1}^r \gamma_j = s$. Now in order to find $\vv{\mu}$, we need to minimize
\begin{gather*}
\sum_{1\leq j\leq r} ((\beta_{j}+\gamma_j+\lambda_j'-1)^2+(\beta_{j}+\lambda_j'-3)^2+\cdots +(\beta_{j}+1-\lambda_j')^2)
\\+\sum_{r<j\leq \lambda_1} ((\beta_{j}+\lambda_j'-1)^2+(\beta_{j}+\lambda_j'-3)^2+\cdots +(\beta_{j}+1-\lambda_j')^2)
\end{gather*}
subject to the condition that $\sum_{j=1}^k \beta_j = 0$ whenever $k\in \lambda$ and $\sum_{j=1}^r \gamma_j = s$. 

It is again easy to show that $\beta_j=0$ for each $j$, thus it suffices to minimize $\sum_{1\leq j \leq r} (\gamma_j+\lambda_j'-1)^2$. Also it follows that $\vv{\alpha}+2\Delta_\lambda$ is obtained from $\vv{\mu}_\lambda(0)$ by replacing $\lambda'_1-1, \lambda'_2-1, \ldots, \lambda'_r-1$ with $\gamma_1+\lambda'_1-1, \gamma_2+\lambda'_2-1, \ldots, \gamma_r+\lambda'_r-1$. Now it is clear that $\{b_1, b_2, \ldots, b_r\}= \{\gamma_1+\lambda'_1-1, \gamma_2+\lambda'_2-1, \ldots, \gamma_r+\lambda'_r-1\}$ by considering the defining conditions of $b_i$. (Recall that $b_i$ is the $(1,i)$-entry of $\mathfrak{W}_\lambda(r,s)$.) Thus it follows that $\vv{\mu}=\vv{\mu}_\lambda(r,s)$ as desired.
\end{proof}

By Lemma \ref{lem:Jringgen} and \ref{lem:theta2cal}, in order to prove Theorem \ref{thm:lvbij} it suffices to show that $\Theta_1(\vv{\mu}_\lambda(0)) = 0$ and $\Theta_1(\vv{\mu}_\lambda(r,s))=\vv{\rho}_\lambda(r,s)$. We analyze this condition more directly in terms of AMBC.
\begin{lem} \label{lem:blockdiag} Suppose that we are given $\lambda\vdash n$, $\vv{\mu}=(\mu_1, \mu_2, \ldots, \mu_n) \in \Dom(GL_n)$, and $\vv{\rho} \in \Dom(F_\lambda)$. Then the following are equivalent:
\begin{enumerate}[label=\textup{(\arabic*)}]
\item $\Theta_1(\vv{\mu}) = \vv{\rho}$.
\item If $w=\Psi(\Tc_\lambda, \Tc_\lambda, \rev_\lambda(\vv{\rho}))$ then $w$ and $w_{\vv{\mu}}$ are in the same double $\Sym_n$-coset.
\item If $w=\Psi(\Tc_\lambda, \Tc_\lambda, \rev_\lambda(\vv{\rho}))$ then $\{w(1), w(2), \ldots, w(n)\}=\{n\mu_1+1, n\mu_2+2, \ldots, n\mu_n+n\}.$
\item Let $w=\Psi(\Tc_\lambda, \Tc_\lambda, \rev_\lambda(\vv{\rho}))$ and $R \subset \cB_w$ be a set of representatives in each translation class by $(n,n)$. (For example we may take $R=\{(x, w(x))\mid x \in [1,n]\}$.) Then $\{D(x) \mid x \in R\}=\vv{\mu}$ as multisets, where $D(x)$ is the block diagonal defined in \ref{sec:phipsi}.
\end{enumerate}
\end{lem}
\begin{proof} (1) and (2) are equivalent by the definition of $\Theta_1$. Now if (3) holds, then since $w_{\vv{\mu}} = [n\mu_1+1, n\mu_2+2, \ldots, n\mu_n+n]$, there exists $\sigma \in \Sym_n$ such that $w_{\vv{\mu}}\sigma =w$. Thus (2) holds. Now let us assume (2). Then $w\in (\lcc_\lambda)^{-1}\cap \lcc_\lambda$, thus it is of the minimal length in its double $\Sym_n$-coset. On the other hand, it is easy to show that $w_{\vv{\mu}}$ is of the minimal length in its right $\Sym_n$-coset. Thus there exists $\sigma \in \Sym_n$ such that $w_{\vv{\mu}}\sigma =w$ which implies that $\{w(1), w(2), \ldots, w(n)\} = \{w_{\vv{\mu}}(1), w_{\vv{\mu}}(2), \ldots, w_{\vv{\mu}}(n)\}$. As $w_{\vv{\mu}}= [n\mu_1+1, n\mu_2+2, \ldots, n\mu_n+n]$, (3) holds.

If we assume (3), then (4) clearly holds. Now we assume (4).  Let $\vv{\mu}' =(\mu'_1, \mu'_2, \ldots, \mu'_n) \in \Dom(GL_n)$ be such that $w$ and $w_{\vv{\mu}'}$ are in the same double $\Sym_n$-coset, which always exists by the bijection $\Dom(GL_n) \simeq \Sym_n \backslash \extSn / \Sym_n \simeq  \bigsqcup_{\lambda \vdash n} (\lcc_\lambda)^{-1}\cap \lcc_\lambda$. As (2) implies (3), we have $\{w(1), w(2), \ldots, w(n)\} = \{n\mu'_1+1, n\mu'_2+2,\ldots, n\mu'_n+n\}$.  But then $\{D(i, w(i)) \mid i \in [1,n]\}=\vv{\mu}'$ as multisets, thus $\vv{\mu}=\vv{\mu}'$ by assumption and thus (3) holds.
\end{proof}

Therefore, in order to prove Theorem \ref{thm:lvbij}, it suffices to find a set of representative $R\subset \cB_w$ in each translation class by $(n,n)$ such that 
$$\{D(x) \mid x\in R\} =\left\{\begin{aligned} 
&\vv{\mu}_\lambda(r,s) &&\textup{ when }w=\Psi(\Tc_\lambda, \Tc_\lambda, \rev_\lambda(\vv{\rho}_\lambda(r,s))),
\\&\vv{\mu}_\lambda(0) &&\textup{ when }w=\Psi(\Tc_\lambda, \Tc_\lambda, 0).
\end{aligned}\right.$$

From now on we proceed by induction on the number of rows in $\lambda$. We first consider one-row case, i.e. $\lambda=(n)$. For $s \geq 0$, there exist $a, b\in \bN$ such that $0\leq b <n$ and $s=an+b$. Then $\vv{\mu}_{\lambda}(n,s)$ (when $s>0$) consists of $b$ numbers of $(a+1)$'s and $(n-b)$ numbers of $a$'s. Thus by Lemma \ref{lem:blockdiag} it suffices to show that
\begin{align*}
\{w(1), w(2), \ldots, w(n)\} &= \{(a+1)n+i \mid 1\leq i \leq b\} \sqcup \{an+i \mid b<i\leq n\}
\\&= \{s+1, s+2, \ldots, s+n\}
\end{align*}
where $w=\Psi(\Tc_\lambda, \Tc_\lambda, (s))$. But this is obvious because $w=\shift^s$ in this case.

Now suppose that we are given $\lambda \vdash n$ such that $l(\lambda)\geq 2$ and assume that Theorem \ref{thm:lvbij} is true for partitions of $<l(\lambda)$ parts. We set $w=\Psi(\Tc_\lambda, \Tc_\lambda, \rev_\lambda(\vv{\rho}))$ for $\vv{\rho}=\vv{\rho}_\lambda(r,s)$ or $\vv{\rho}=0$ and let $\vv{\mu} \in \Dom(GL_n)$ such that $\Theta_1(\vv{\mu})=\vv{\rho}$. Also we let $n' \colonequals n-\lambda_1$, $\tilde{\lambda} \colonequals (\lambda_2, \lambda_3, \ldots) \vdash n'$ and $w' \colonequals \fw(w)$. Then there exists $\vv{\rho}' \in \Dom(F_{\tilde{\lambda}})$ such that $w' = \Psi(\Tc_{\tilde{\lambda}}, \Tc_{\tilde{\lambda}}, \vv{\rho}')$ as a partial permutation. Indeed, $\vv{\rho}'$ is defined such that $\rev_{\tilde{\lambda}}(\vv{\rho}')$ is obtained by removing the first entry from $\rev_{\lambda}(\vv{\rho})$. Thus in particular we have
$$\vv{\rho}' = \left\{
\begin{aligned}
&\vv{\rho}_{\tilde{\lambda}}(r,s) &&\textup{ if } \vv{\rho}=\vv{\rho}_\lambda(r,s) \textup{ where } r \neq \lambda_1 \textup{ or } r=\lambda_1=\lambda_2,
\\&0 &&\textup{ otherwise}.
\end{aligned}
\right.$$
Let $\vv{\mu}' \in \Dom(GL_{n'})$ such that $\Theta_1(\vv{\mu}') = \vv{\rho}'$. Then by induction assumption, $\vv{\mu}'=\vv{\mu}_{\tilde{\lambda}}(0)$ if $\vv{\rho}'=0$ and $\vv{\mu}'=\vv{\mu}_{\tilde{\lambda}}(r,s)$ if $\vv{\rho}'=\vv{\rho}_{\tilde{\lambda}}(r,s)$.
Also, for some $r>0$ we have
$$\st(w) = \left\{
\begin{aligned}
&\st_s(\{n'+1,n'+2, \ldots, n\}, \{n'+1,n'+2, \ldots, n\}) &&\textup{ if } \vv{\rho}=\vv{\rho}_\lambda(\lambda_1,s) \textup{ and } \lambda_1>\lambda_2,
\\&\st_0(\{n'+1,n'+2, \ldots, n\},\{n'+1,n'+2, \ldots, n\}) &&\textup{ otherwise}.
\end{aligned}
\right.$$
(See \ref{sec:phipsi} for the definition of $\st(w)$. Also, $\st_a(A, B)$ is a stream of altitude $a$ yielding a bijection from $A+n\bZ$ to $B+n\bZ$. See \cite[Section 3.4]{cpy18} for more details.)

\noindent{\bfseries{Case 1.}}
Let us first consider the case when $r=\lambda_1>\lambda_2$, $\st(w) = \st_s(\{n'+1,n'+2, \ldots, n\},\{n'+1,n'+2, \ldots, n\})$ for some $s\geq 0$, and $\vv{\rho}'=0$.
We know that Theorem \ref{thm:lvbij} holds when $\vv{\rho}=0$ since $\Theta_1^{-1}(0) = \Theta_2^{-1}(0)$; see Lemma \ref{lem:Jringgen} and the following remark. We prove the statement by induction on $s$, increasing it by $1$. 

Consider the tableau $\mathfrak{W}_\lambda(\lambda_1,s)$ and define its {\it {$\mathfrak{W}$-channel}} as follows. It is a subset of boxes such that (1) its intersection with each column has exactly one box; (2) there exists $k\in \bZ$ such that the content of each box in it is either $k$ or $k+1$; (3) the boxes in the right columns are not lower than the boxes in the left columns.

\begin{example}
Two $\mathfrak{W}$-channels of the tableau $\mathfrak{W}_{(4,3,2,2,1)}(0)$ are shown below:
\[
\ytableaushort{
43{*(white!80!black)1}{*(white!80!black)0},2{*(white!80!black)1}{-1},{*(white!80!black)0}{-1},{-2}{-3},{-4}
}\quad \text{   and   }\quad 
\ytableaushort{
431{*(white!80!black)0},21{*(white!80!black)-1},{*(white!80!black)0}{*(white!80!black)-1},{-2}{-3},{-4}
}
\]
\end{example}


\begin{lem}
Each $\mathfrak{W}_\lambda(\lambda_1,s)$ has either one or two $\mathfrak{W}$-channels.
\end{lem}

\begin{proof}
Let $k$ be the filling of the single box in the leftmost column among those of size $1$ (which is possible as we assume that $\lambda_1>\lambda_2$). By definition of $\mathfrak{W}_\lambda(\lambda_1,s)$ we know $k$ is non-negative, since $b_i \geq a_i =0$ if $i$ is the index of the column of size 1. It is clear that any $\mathfrak{W}$-channel must contain this $k$, and therefore it can only be filled with either $k$-s and $k+1$-s, or with $k$ and $k-1$. Furthermore, for each of those two choices there is at most one $\mathfrak{W}$-channel filled with them, this is because no two entries in the same column of $\mathfrak{W}_\lambda(\lambda_1,s)$ are consecutive integers. What remains to argue is that at least one of those two $\mathfrak{W}$-channels does exist. 

Assume that $i$ is the smallest column index such that $b_i=k$. First we suppose that we can find one of the entries $k-1, k, k+1$ in the column $i-1$ below the first row. If it is $k-1$, then we can find either $k$ or $k-1$ in each column to the left of it. If it is $k$ or $k+1$, we can find either $k$ or $k+1$ in each column to the left of it. Assume now that all entries below the first row in column $i-1$ are at most $k-2$. This means that $a_{i-1} \leq k$. If we had $b_{i-1} \geq k+2$, we could reduce $b_{i-1}$ by $1$ and increase $b_i$ by $1$ while decreasing the sum of squares of $b$-s. Thus $b_{i-1}=k+1$. Let $j$ be the smallest column index such that $b_j=k+1$. Then if we can find an entry equal to $k$ or $k+1$ below the first row in column $j-1$, we can proceed to find such entry in each next column to the left. If not, we know that $a_{j-1} \leq k+1$. On the other hand, by assumption $b_{j-1} \geq k+2$. Reducing $b_{j-1}$ by $1$ and increasing $b_i$ by $1$ decreases the sum of squares of $b$'s, which is a contradiction. This completes the proof. 
\end{proof}


Let $w = \Psi(\Tc_\lambda, \Tc_\lambda, \rev_\lambda(\vv{\rho}_\lambda(\lambda_1,s)))$
and assume that the entries of $\mathfrak{W}_\lambda(\lambda_1,s)$ coincide with the multiset $\{D(x,w(x)) \mid x \in [1,n]\}$. (This holds for $\Psi(\Tc_\lambda, \Tc_\lambda, 0)$ as observed in Lemma \ref{lem:blockdiag}, which is the base case for our induction step.)

\begin{lem} Suppose the situation above. Then there is an one-to-one correspondence between channels of $w$ and $\mathfrak{W}$-channels of $\mathfrak{W}_\lambda(\lambda_1,s)$ so that the multiset $\{D(b) \mid b\in C\cap [1,n]\times \bZ\}$ is equal to the entries of the $\mathfrak{W}$-channel corresponding to the channel $C$.
\end{lem}

\begin{proof}
Throughout the proof we are using the fact that $w$ lies in both left and right canonical cells, i.e. $w(1)<w(2)<\cdots<w(n)$ and $w^{-1}(1)<w^{-1}(2)<\cdots <w^{-1}(n)$.
Then from the condition it easily follows that if $C$ is a stream of $w$ then there exists $k \in \bZ$ such that $D(b) \in \{k, k+1\}$ for any $b \in C$. Moreover, for any $k \in \bZ$ the set $\{(x, w(x)) \mid D(x, w(x)) \in \{k, k+1\}\}$ is always a stream of $w$ by similar reason. Thus it is indeed a channel of $w$ if and only if $k$ and $k+1$ together appear in $\mathfrak{W}_\lambda(\lambda_1,s)$ exactly $\lambda_1$ times. This is equivalent to the existence of a $\mathfrak{W}$-channel of entries in $\{k, k+1\}$. Thus the result follows.
\end{proof}
We proceed with the step of induction. Let us increase $s$ by $1$. By Theorem \cite[Theorem 16.9]{cpy18} this results in a shift of size $1$ of the most northeast channel of the indexing river. (See \cite[Definition 13.9]{cpy18} for the definition of the shift of a stream, \cite[Definition 3.18]{cpy18} for a river, and \cite[Definition 16.6]{cpy18} for an indexing river.) In our situation, since $\lambda_1>\lambda_2$ we only have one river of $w$ and thus it is equal to the set of all channels of $w$. When there is only one channel, then it will automatically the most northeast channel and we set $k\in \bZ$ so that the set of entries of the corresponding $\mathfrak{W}$-channel are either $\{k\}$ or $\{k, k+1\}$. If there are two of them, the most northeast one corresponds to the $\mathfrak{W}$-channel of entries in $\{k,k+1\}$ if the other one corresponds to that of entries $\{k,k-1\}$ for some $k \in \bZ$.
\begin{rmk}
 Theorem \cite[Theorem 16.9]{cpy18} is stated in the generality when the streams have the same flow as the width of the Shi poset. This is equivalent to $\lambda_1 = \lambda_2$. We are interested in the case when $\lambda_1 > \lambda_2$. The proof of all statement in \cite[Section 16.1]{cpy18} however extends verbatim to this generality as well. 
\end{rmk}
It is easy to see that the multiset of block diagonals $D(x)$ as we vary $x$ over equivalence classes in the most northeast channel changes so that we get one less of $D(x)=k$ and one more of $D(x)=k+1$. Note that by construction the smallest entry in the first row belongs to all (i.e. one or two) $\mathfrak{W}$-channels. Thus, it follows that $k$ is indeed the smallest entry in the first row.

We now argue that the same change happens to the total content of $\mathfrak{W}_\lambda(\lambda_1,s)$ as we increase $s$ by $1$, i.e.  $\mathfrak{W}_\lambda(\lambda_1,s+1)$ is obtained from $\mathfrak{W}_\lambda(\lambda_1,s)$ by increasing the smallest entry of the first row by $1$. Indeed, keeping in mind that we are trying to minimize $\sum b_i^2$ conditioned on knowing $\sum b_i$ and the fact that $b_i \geq a_i \geq 0$, it is not hard to check that first, one is indeed related to the other by a single increase of an entry by $1$, and second, it has to be the smallest entry. This completes the step of induction.

\begin{example}
 The following illustrates how the $\mathfrak{W}$-channels change as we start with $s=0$ and increase it till $s = 5$ for $\mathfrak{W}_{(4,3,3,3,1)}(4,s)$. In each step, the union of all boxes belonging to some $\mathfrak{W}$-channel is described as follows.
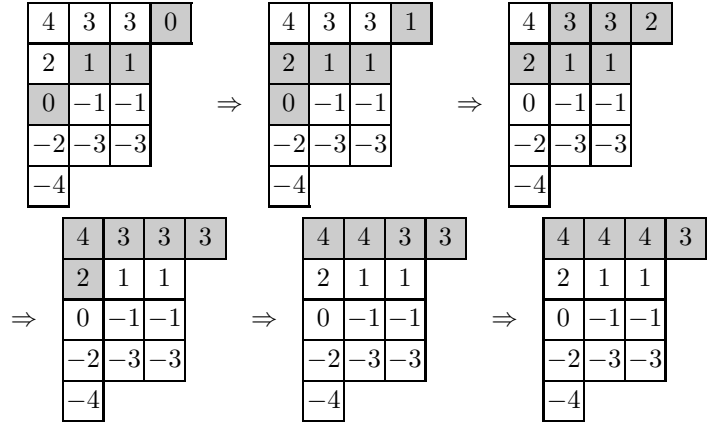
\begin{gather*}
\begin{ytableau}
4&3&3&*(white!80!black)0\\
2&*(white!80!black)1&*(white!80!black)1\\
*(white!80!black)0&-1&-1\\
-2&-3&-3\\
-4
\end{ytableau}\ \ \Rightarrow \ \ 
\begin{ytableau}
4&3&3&*(white!80!black)1\\
*(white!80!black)2&*(white!80!black)1&*(white!80!black)1\\
*(white!80!black)0&-1&-1\\
-2&-3&-3\\
-4
\end{ytableau}\ \ \Rightarrow \ \ 
\begin{ytableau}
4&*(white!80!black)3&*(white!80!black)3&*(white!80!black)2\\
*(white!80!black)2&*(white!80!black)1&*(white!80!black)1\\
0&-1&-1\\
-2&-3&-3\\
-4
\end{ytableau}
\\\ \ \Rightarrow \ \ 
\begin{ytableau}
*(white!80!black)4&*(white!80!black)3&*(white!80!black)3&*(white!80!black)3\\
*(white!80!black)2&1&1\\
0&-1&-1\\
-2&-3&-3\\
-4
\end{ytableau}\ \ \Rightarrow \ \ 
\begin{ytableau}
*(white!80!black)4&*(white!80!black)4&*(white!80!black)3&*(white!80!black)3\\
2&1&1\\
0&-1&-1\\
-2&-3&-3\\
-4
\end{ytableau}\ \ \Rightarrow \ \ 
\begin{ytableau}
*(white!80!black)4&*(white!80!black)4&*(white!80!black)4&*(white!80!black)3\\
2&1&1\\
0&-1&-1\\
-2&-3&-3\\
-4
\end{ytableau}
\end{gather*}
\end{example}

\noindent{\bfseries{Case 2.}}
It remains to consider the case when $\st(w) = \st_0(\{n'+1,n'+2, \ldots, n\},\{n'+1,n'+2, \ldots, n\})$ and $\vv{\rho}'=\vv{\rho}_{\tilde{\lambda}}(r,s) $ for some $r$ and $s$.
Here we define some ad-hoc notations; for an integer vector $\vv{\nu}$ and $i \in \bZ$, we define $\vv{\nu}^{(i)}$ to be the number of $i$ in $\vv{\nu}$. Also let $\vv{\nu}^{(\geq i)} = \sum_{j\geq i} \vv{\nu}^{(j)}$ and define $\vv{\nu}^{(>i)}, \vv{\nu}^{(\leq i)}$, and $\vv{\nu}^{(<i)}$ similarly. We start with the following lemma.

\begin{lem} \label{lem:mu} Let $\vv{\mu}_{\tilde{\lambda}}(0), \vv{\mu}_{\tilde{\lambda}}(r,s)$ be as before. Here $\vv{\mu}'$ can be any of $\vv{\mu}_{\tilde{\lambda}}(0)$ or $\vv{\mu}_{\tilde{\lambda}}(r,s)$.
\begin{enumerate}[label=\textup{(\arabic*)}]
\item For $k \in \bZ$, we have $\vv{\mu}'^{(k)}+\vv{\mu}'^{(k+1)} \leq \lambda_2$.
\item For $k \in \bZ$, we have $\vv{\mu}_{\tilde{\lambda}}(0)^{(k)}=\vv{\mu}_{\tilde{\lambda}}(0)^{(-k)}$.
\item For $k <0$, we have $\vv{\mu}_{\tilde{\lambda}}(r,s)^{(k)}=\vv{\mu}_{\tilde{\lambda}}(0)^{(k)}$.
\item $\vv{\mu}_{\tilde{\lambda}}(r,s)^{(\geq 0)} = \vv{\mu}_{\tilde{\lambda}}(0)^{(\geq 0)}$.
\item $\vv{\mu}_{\tilde{\lambda}}(r,s)^{(0)}\leq \vv{\mu}_{\tilde{\lambda}}(0)^{(0)}$.
\item $\vv{\mu}'^{(>0)}\geq \vv{\mu}'^{(<0)}$ or equivalently $\vv{\mu}'^{(\geq0)}\geq \vv{\mu}'^{(\leq0)}$.
\item For $k>0$, we have $\vv{\mu}'^{(\geq 0)}-\vv{\mu}'^{(<-k)}\leq k\lambda_2$.
\end{enumerate}
\end{lem} 
\begin{proof} (1) holds since each column of $\mathfrak{W}_{\tilde{\lambda}}(0)$ and $\mathfrak{W}_{\tilde{\lambda}}(r,s)$ does not contain any two consecutive integers. (2) is obvious. (3) is also straightforward since the negative entries of $\mathfrak{W}_{\tilde{\lambda}}(0)$ and $\mathfrak{W}_{\tilde{\lambda}}(r, s)$ are the same. (4) follows immediately from (3). (5) follows from the description of $\mathfrak{W}_{\tilde{\lambda}}(r,s)$. (6) holds because of (2), (4), and (5). For (7), by (3) and (4) it suffices to prove the inequality when $\vv{\mu}'=\vv{\mu}_{\tilde{\lambda}}(0)$. But $\vv{\mu}_{\tilde{\lambda}}(0)^{(\geq 0)}=\vv{\mu}_{\tilde{\lambda}}(0)^{(\leq 0)}$ by (2), thus the result follows from (1).
\end{proof}

Suppose that the window notation of $w'$ is given by
\begin{gather*}
[c_1, c_2, \ldots, c_{\vv{\mu}'^{(<0)}},b_{\vv{\mu}'^{(0)}}, \ldots,b_{2},b_{1}, a_{\vv{\mu}'^{(>0)}},  \ldots, a_2, a_1,\emptyset, \emptyset, \ldots, \emptyset].
\end{gather*}
In particular, we have
$$c_1<c_2<\cdots<c_{\vv{\mu}'^{(<0)}}<b_{\vv{\mu}'^{(0)}}<\cdots <b_{2}<b_{1}<a_{\vv{\mu}'^{(>0)}}<\cdots <a_{2}<a_{1}.$$
Furthermore, it is clear that
\begin{align*}
\vv{\mu}'=&\ \left(\ceil{\frac{a_1}{n}}-1, \ceil{\frac{a_2}{n}}-1, \ldots, \ceil{\frac{a_{\vv{\mu}'^{(>0)}}}{n}}-1, \ceil{\frac{b_1}{n}}-1, \ceil{\frac{b_2}{n}}-1, \ldots, \ceil{\frac{b_{\vv{\mu}'^{(0)}}}{n}}-1,\right.
\\& \left. \quad \ceil{\frac{c_{\vv{\mu}'^{(<0)}}}{n}}-1,\ldots,  \ceil{\frac{c_2}{n}}-1, \ceil{\frac{c_1}{n}}-1\right).
\end{align*}
It follows that $\ceil{\frac{a_i}{n}}\geq2,\ceil{\frac{b_i}{n}}=1$, and $\ceil{\frac{c_i}{n}}\leq0$ for any $a_i, b_i, c_i$. 
\begin{example} Let $n=11$, $\lambda=(4, 3, 3, 1)$, and $\vv{\rho}'=(0,2,0) = \vv{\rho}_{(3,3,1)}(3,2)$. Then
$$w'=[ -15, -6, -5, 4, 23,24, 25, \emptyset,\emptyset,\emptyset,\emptyset].$$
Thus we have
\begin{gather*}
a_1=25,\  a_2=24,\  a_3=23,\  b_1=4,\  c_3=-5,\  c_2=-6,\  c_1=-15,
\\ \vv{\mu}'=(2, 2, 2, 0, -1, -1, -2) = \vv{\mu}_{(3,3,1)}(3,2).
\end{gather*}
Indeed, $\mathfrak{W}_{(3,3,1)}(3,2)$ is given by
$$\tableau[sY]{2, 2, 2 \\ 0, -1, -1 \\ -2}.$$
\end{example}

The following lemma will turn out to be useful later on.

%
%

\begin{lem} \label{lem:ineq} Let $k\in \bZ_{>0}$.
\begin{enumerate}[label=\textup{(\arabic*)}]
\item For $1\leq i \leq \vv{\mu}'^{(>0)}$ and $k\in \bZ_{>0}$, $\ceil{\frac{a_i}{n}}-1 \leq k$ if and only if $i > \vv{\mu}'^{(>k)}$.
\item For $1\leq i \leq \vv{\mu}'^{(<0)}$ and $k\in \bZ_{>0}$, $1-\ceil{\frac{c_i}{n}} \leq k$ if and only if $i > \vv{\mu}'^{(< -k)}$.
\item For any  $1\leq i \leq j \leq \vv{\mu}'^{(<0)}$, we have $2\floor{\frac{j-i}{\lambda_2}} \leq \ceil{\frac{c_j}{n}}- \ceil{\frac{c_i}{n}}$.
\end{enumerate}
\end{lem}
\begin{proof} (1) and (2) follow from the description of $\vv{\mu}'$ above. (3) is equivalent to that if $j-i\geq k\lambda_2$ for some $k \in \bN$ then $\ceil{\frac{c_j}{n}}- \ceil{\frac{c_i}{n}} \geq 2k$. This is easily obtained from Lemma \ref{lem:mu}(1).
\end{proof}


%
%

%

We define (recall that $n'=n-\lambda_1$)
\begin{align*}
&A_i = (n'+1-i,a_i) &&\textup{ for } 1\leq i \leq \vv{\mu}'^{(>0)},
\\&B_i = (\vv{\mu}'^{(\leq 0)}+1-i,b_i)  &&\textup{ for } 1\leq i \leq \vv{\mu}'^{(0)},
\\&C_i = \left((1-\ceil{\frac{c_i}{n}})n+i,(1-\ceil{\frac{c_i}{n}})n+c_i\right) &&\textup{ for } 1\leq i \leq \vv{\mu}'^{(<0)}.
\end{align*}
Then it is immediate that
$$A_i, B_i, C_i \in \cB_{w'} \cap \left(\{(a,b) \in \bZ^2 \mid a\geq 1, b\geq 1\} - \{(a,b) \in \bZ^2 \mid a\geq n'+1, b\geq n'+1\}\right)$$
and that $ \{A_i \mid 1\leq i \leq \vv{\mu}'^{(>0)}\} \sqcup  \{B_i \mid 1\leq i \leq \vv{\mu}'^{(0)}\} \sqcup \{C_i \mid 1\leq i \leq \vv{\mu}'^{(<0)}\}$ contains exactly one representative in each translation class by $(n,n)$ in $\cB_{w'}$.
Also, we have
\begin{gather*}
A_1 <_{NW} A_2 <_{NW}  \cdots <_{NW}  A_{\vv{\mu}'(>0)} <_{NW}  B_1,
\\C_1 <_{NW}  C_2 <_{NW}  \cdots <_{NW}  C_{\vv{\mu}'(<0)} <_{NW}  B_1,
\\B_1 <_{NW}  B_2 <_{NW}  \cdots <_{NW}  B_{\vv{\mu}'(0)}.
\end{gather*}
However, $A_i$ and $C_j$ are not comparable with respect to northwest ordering.

\begin{example} When $w'=[ -15, -6, -5, 4, 23,24, 25, \emptyset,\emptyset,\emptyset,\emptyset]$ as above, we have
$$A_1 = (7,25), A_2 = (6,24), A_3=(5,23), B_1 = (4,4), C_3=(13,5), C_2=(14,6), C_1=(23,7).$$
Thus $A_1<_{NW} A_2<_{NW} A_3<_{NW} B_1$ and $C_1<_{NW} C_2<_{NW} C_3<_{NW} B_1$. However, any of $A_i$ and $C_j$ are not comparable.
\end{example}

In order to proceed the backward AMBC, it is necessary to calculate the backward numbering $d^{\bk, \st(w)}_{w'}$ as in \cite[Section 4.2]{cpy18}. For this, first we fix a proper numbering of $\st(w)$ such that $\st(w)^{(-1)} = (0, 0)$, i.e. $(0,0)$ is labeled -1. Also we define the numbering $d$ on $A_i, B_i,$ and $C_i$ by
$$d(A_i) = -i, \quad d(B_i) = -(\vv{\mu}'^{(>0)} + i), \quad d(C_i) = -i.$$
and extend to the whole of $\cB_{w'}$ by periodicity, i.e. if $b \in \cB_{w'}$ then $d(b+k(n,n))=d(b)+k\lambda_1$.
Note that this is the largest monotone numbering on $\{A_i\}_i \cup \{B_i\}_i \cup \{C_i\}_i$ which respects $\leq_{NW}$ subject to the condition that $d(A_1)=d(C_1) = -1$. (Recall that $\vv{\mu}'^{(>0)}\geq\vv{\mu}'^{(<0)}$ by Lemma \ref{lem:mu}.)

\begin{example} When $w'=[ -15, -6, -5, 4, 23,24, 25, \emptyset,\emptyset,\emptyset,\emptyset]$ as above, $d$ is defined by
$$d(A_1) = d(C_1) =-1, \quad d(A_2)=d(C_2)=-2,\quad  d(A_3)=d(C_3) = -3, \quad d(B_1) = -4.$$
\end{example}

\begin{lem}Let $d$ be the numbering of $\cB_{w'}$ defined above. Then $d=d^{\bk, \st(w)}_{w'}$.
\end{lem}
\begin{proof} Suppose otherwise. Then from the construction of backward numbering, there exist $b, b' \in \cB_{w'}$ such that $b<_{NW} b'$ and $d(b) \geq d(b')$. Without loss of generality we may assume that $b$ is one of $A_i$, $B_i$, or $C_i$. We consider each case in the following.
\begin{enumerate}[label=(\alph*), leftmargin=*]
\item Suppose that $b=A_i$ for some $1\leq i \leq \vv{\mu}'^{(>0)}$. In particular, $d(b) = -i$.
\begin{enumerate}[label=$\bullet$, leftmargin=*]
\item If $b'=A_j+k(n,n)$ for some $1\leq j \leq \vv{\mu}'^{(>0)}$ and $k \in \bZ$, then we have $d(b') = -j+k\lambda_1\leq -i$, i.e. $k\leq \frac{j-i}{\lambda_1}$. Comparing the $y$-coordinates, we also have $a_j+kn>a_i$, i.e. $k >\frac{a_i-a_j}{n}$. It follows that $\frac{a_i-a_j}{j-i}< \frac{n}{\lambda_1} <1$. But this is impossible since $(a_1, a_2, \ldots)$ is a strictly decreasing sequence.
%
\item $b'=B_j+k(n,n)$ for some $1\leq j \leq \vv{\mu}'^{(0)}$ and $k \in \bZ$. This case we have $d(b') = -(\vv{\mu}'^{(>0)}+j)+k\lambda_1\leq -i$, i.e. $k\leq\frac{\vv{\mu}'^{(>0)}+j-i}{\lambda_1}$. Comparing the $y$-coordinates, we also have $b_j+kn>a_i$, i.e. $k>\frac{a_i-b_j}{n}$. It follows that $\frac{a_i-b_j}{\vv{\mu}'^{(>0)}+j-i}< \frac{n}{\lambda_1}<1$. This is again impossible since $(a_1, a_2, \ldots, a_{\vv{\mu}'^{(>0)}}, b_1, b_2, \ldots)$ is strictly decreasing.
%
\item $b'=C_j+k(n,n)$ for some $1\leq j \leq \vv{\mu}'^{(<0)}$ and $k \in \bZ$. This case we have $d(b') = -j+k\lambda_1\leq -i$. Comparing the $y$-coordinates, we also have $(1-\ceil{\frac{c_j}{n}})n+c_j+kn>a_i$, i.e. $k>\frac{a_i-c_j}{n}-1+\ceil{\frac{c_j}{n}}\geq \frac{a_i}{n}-1$. Thus $\ceil{\frac{a_i}{n}}-1\leq k$, which implies $i > \vv{\mu}'^{(>k)}$ by Lemma \ref{lem:ineq}. Thus we have
$\vv{\mu}'^{(>0)}\geq j \geq k\lambda_1 + i  > k\lambda_1 + \vv{\mu}'^{(>k)}\geq \vv{\mu}'^{(\geq 0)}$
by Lemma \ref{lem:mu} (as $\lambda_1\geq \lambda_2$), which is not possible.
\end{enumerate}
\item Suppose that $b=B_i$ for some $1\leq i \leq \vv{\mu}'^{(0)}$. In particular, $d(b) = -(\vv{\mu}'^{(>0)} + i)$.
\begin{enumerate}[label=$\bullet$, leftmargin=*]
\item If $b'=A_j+k(n,n)$ for some $1\leq j \leq \vv{\mu}'^{(>0)}$ and $k \in \bZ$, then we have $d(b') = -j+k\lambda_1\leq -(\vv{\mu}'^{(>0)} + i)$, i.e. $k\lambda_1\leq j-i-\vv{\mu}'^{(>0)} \leq -i<0$. However, comparing the $x$-coordinates we have $\vv{\mu}'^{(\leq 0)}+1-j<kn+(n'+1-i)$, i.e. $kn>\vv{\mu}'^{(\leq 0)}-j+i-n'=\vv{\mu}'^{(> 0)}-j+i\geq i>0$. This is impossible.
\item $b'=B_j+k(n,n)$ for some $1\leq j \leq \vv{\mu}'^{(0)}$ and $k \in \bZ$. This case we have $d(b') = -(\vv{\mu}'^{(>0)}+j)+k\lambda_1\leq -(\vv{\mu}'^{(>0)}+j)$, i.e. $k\lambda_1 \leq j-i$. Since $|j-i| < \vv{\mu}'^{(0)}\leq \lambda_2\leq \lambda_1$ by Lemma \ref{lem:mu}, it follows that $k=0$. This is impossible since it implies $b'=B_j$.
\item $b'=C_j+k(n,n)$ for some $1\leq j \leq \vv{\mu}'^{(<0)}$ and $k \in \bZ$. This case we have $d(b') = -j+k\lambda_1\leq-(\vv{\mu}'^{(>0)} + i)$, i.e. $k\lambda_1 \leq j-i-\vv{\mu}'^{(>0)} \leq -i<0$ since $\vv{\mu}'^{(<0)} \leq \vv{\mu}'^{(>0)}$ by Lemma \ref{lem:mu}. By looking at the $y$-coordinates, we also have $b_i<(1-\ceil{\frac{c_j}{n}})n+c_j+kn$, i.e. $kn>b_i-c_j-(1-\ceil{\frac{c_j}{n}})n\geq b_i-n>-n$, thus $k>-1$. This is again impossible.
\end{enumerate}

\item Suppose that $b=C_i$ for some $1\leq i \leq \vv{\mu}'^{(<0)}$. In particular, $d(b) = -i$.
\begin{enumerate}[label=$\bullet$, leftmargin=*]
\item If $b'=A_j+k(n,n)$ for some $1\leq j \leq \vv{\mu}'^{(>0)}$ and $k \in \bZ$, then we have $d(b') = -j+k\lambda_1\leq -i$. By looking at the $x$-coordinates, we also have $n'+1-j+kn>(1-\ceil{\frac{c_i}{n}})n+i$, i.e. $k>\frac{i+j-n'-1}{n}+1-\ceil{\frac{c_i}{n}}>1+\ceil{\frac{c_i}{n}}$. Thus by Lemma \ref{lem:ineq}  it follows that $i> \vv{\mu}'^{(<-k)}$. But then we have
$\vv{\mu}'^{(\geq0)}\geq \vv{\mu}'^{(>0)} \geq j \geq k\lambda_1 + i > k\lambda_1 + \vv{\mu}'^{(<-k)}$
but this is not possible by Lemma \ref{lem:mu}.
\item $b'=B_j+k(n,n)$ for some $1\leq j \leq \vv{\mu}'^{(0)}$ and $k \in \bZ$. This case we have $d(b') = -(\vv{\mu}'^{(>0)}+j)+k\lambda_1\leq -i$. By looking at the $x$-coordinates, we also have $(1-\ceil{\frac{c_i}{n}})n+i < \vv{\mu}'^{(\leq 0)}+1-j+kn$, i.e. $k>\frac{i+j-\vv{\mu}'^{(\leq 0)}-1}{n}+1-\ceil{\frac{c_i}{n}}>1-\ceil{\frac{c_i}{n}}$. Thus $i>\vv{\mu}'^{(<-k)}$ by Lemma \ref{lem:ineq}. Now we have
$\vv{\mu}'^{(\geq 0)}\geq j+ \vv{\mu}'^{(>0)}\geq k\lambda_1 + i > k\lambda_1 + \vv{\mu}'^{(<-k)}$
but this is not possible by Lemma \ref{lem:mu}.
\item $b'=C_j+k(n,n)$ for some $1\leq j \leq \vv{\mu}'^{(<0)}$ and $k \in \bZ$. This case we have $d(b') = -j+k\lambda_1\leq -i$, i.e. $k\leq \frac{j-i}{\lambda_1}$. By looking at the $y$-coordinates, we also have $(1-\ceil{\frac{c_j}{n}})n+c_j+kn>(1-\ceil{\frac{c_i}{n}})n+c_i$, i.e. $k>\frac{c_i-c_j}{n}+\ceil{\frac{c_j}{n}}-\ceil{\frac{c_i}{n}}>-1.$ Thus $k\geq 0$, and since $k=0$ case is not possible we have $k\geq 1$. Also $j\geq i$. This time we compare the $x$-coordinates and obtain $(1-\ceil{\frac{c_j}{n}})n+j+kn>(1-\ceil{\frac{c_i}{n}})n+i$, i.e. 
$k>\frac{i-j}{n}+\ceil{\frac{c_j}{n}}-\ceil{\frac{c_i}{n}}.$ Thus $k \geq \ceil{\frac{c_j}{n}}-\ceil{\frac{c_i}{n}}$ as $|i-j|<n$. By Lemma \ref{lem:ineq}, it implies that $2\floor{\frac{j-i}{\lambda_1}} \leq 2\floor{\frac{j-i}{\lambda_2}}\leq k\leq \frac{j-i}{\lambda_1},$ which is true only when $\frac{j-i}{\lambda_1}<1$. This contradicts that $1\leq k \leq \frac{j-i}{\lambda_1}$.
\end{enumerate}
\end{enumerate}

We considered all the possible cases. The lemma is proved.
\end{proof}

We set $R' \colonequals \bigsqcup_{i=1}^{\lambda_1}R_i'$ where $R'_i = \{b \in \cB_{w'} \mid d(b) = -i\}.
$ 
Then it is clear that $R'$ is the set of representatives in each translation class by $(n,n)$ in $\cB_{w'}$. Also for $1\leq i \leq \lambda_1$ let $Z_i$ be the zig-zag with the back corner-post at $\st(w)^{(-i)} = (1-i, 1-i)$ and outer corner-posts at the balls in $R_i'$. We set $R\colonequals \bigsqcup_{i=1}^{\lambda_i}R_i$ where $R_i =\{\textup{inner corner-posts of } Z_i\}$. Then it is also clear that $R$ is the set of representatives in each translation class by $(n,n)$ in $\cB_{w}$. Our goal is to show that $\{D(b) \mid b \in R\}$ is equal to $\vv{\mu}_{\lambda}(0)$ or $\vv{\mu}_{\lambda}(r,s)$, and by Lemma \ref{lem:blockdiag} it implies Theorem \ref{thm:lvbij}.

For $k \in \bZ$ and $1\leq i \leq \lambda_1$, let us set
\begin{align*}
&\tilde{A}_{k,i} \colonequals A_{k\lambda_1+i}+k(n,n)&& \textup{ if } 1 \leq k\lambda_1+i \leq \vv{\mu}'^{(>0)},
\\&\tilde{B}_{k,i} \colonequals  B_{k\lambda_1+i-\vv{\mu}'^{(>0)}}+k(n,n)&& \textup{ if } 1\leq k\lambda_1+i -\vv{\mu}'^{(>0)}\leq \vv{\mu}'^{(0)},
\\&\tilde{C}_{k,i} \colonequals C_{k\lambda_1+i}+k(n,n) &&  \textup{ if } 1\leq k\lambda_1+i \leq \vv{\mu}'^{(<0)}.
\end{align*}
Note that $d(\tilde{A}_{k,i})=d(\tilde{B}_{k,i})=d(\tilde{C}_{k,i})=-i$.
Also let $N_i=\floor{\frac{\vv{\mu}'^{(<0)}-i}{\lambda_1}}$ so that $N_i\lambda_1+i \leq \vv{\mu}'^{(<0)} <(N_i+1)\lambda_1 +i$. Since $\vv{\mu}'^{(\geq 0)}-\vv{\mu}'^{(<0)}<\lambda_1$ by Lemma \ref{lem:mu}, $(N_i+2)\lambda_1+i> \vv{\mu}'^{(\geq 0)}$. Now we consider the three possible cases below.
\begin{enumerate}[label=$(\alph*)$, leftmargin=*]
\item Suppose that $(N_i+1)\lambda_1+i\leq \vv{\mu}'^{(>0)}$. Then
$$R_i'=\{\tilde{A}_{k,i} \mid 1\leq k\leq N_i+1\} \sqcup\{\tilde{C}_{k,i} \mid 1\leq k\leq N_i\}.$$
\item Suppose that  $ \vv{\mu}'^{(>0)} < (N_i+1)\lambda_1+i \leq \vv{\mu}'^{(\geq0)}$. Then
$$R_i'=\{\tilde{A}_{k,i} \mid 1\leq k\leq N_i\}\sqcup \{\tilde{B}_{N_i+1,i}\} \sqcup\{\tilde{C}_{k,i} \mid 1\leq k\leq N_i\}.$$
\item Suppose that  $(N_i+1)\lambda_1+i > \vv{\mu}'^{(\geq0)}$. Then
$$R_i'=\{\tilde{A}_{k,i} \mid 1\leq k\leq N_i\} \sqcup\{\tilde{C}_{k,i} \mid 1\leq k\leq N_i\}. $$
\end{enumerate}
\begin{example}
Suppose that $n=15$, $\lambda=(5,4,4,2)$, and $\vv{\rho}' = (0,4,0)$. Thus 
$$w' = \Psi(\Tc_{(4,4,2)},\Tc_{(4,4,2)},(0,3,0)) = [-21, -20, -8, -7, 5, 6, 32, 33, 34, 46,\emptyset,\emptyset,\emptyset,\emptyset,\emptyset].$$
Then
\begin{gather*}
A_1=(10,46), \quad A_2=(9,34), \quad A_3=(8,33), \quad A_4=(7,32), \quad B_1 = (6,6), \quad B_2=(5,5),
\\C_4 = (18,7), \quad C_3=(19,8), \quad C_2=(31,9), \quad C_1=(32,10)
\end{gather*}
and $d(A_i)=d(C_i) = -i, d(B_i) = -5-i$. Also, 
\begin{gather*}
\tilde{A}_{0,i}=A_i \textup{ for } 1\leq i \leq 4, \quad \tilde{C}_{0,i}=C_i \textup{ for } 1\leq i \leq 4, 
\\\tilde{B}_{1,1}=B_{2}+(15,15) = (20,20), \quad \tilde{B}_{0,5}=B_{1} = (6,6).
\end{gather*}
Thus in particular $d(\tilde{A}_{0,i})=d(\tilde{C}_{0,i})=-i$ for $1\leq i \leq 4$ and $d(\tilde{B}_{1,1}) = -1, d(\tilde{B}_{0,5})= -5$.
\end{example}

Recall the process to calculate inner corner-posts of a zigzag from the back and outer corner-posts. If a zigzag $Z$ consists of the back corner post $(a_0, b_0)$ and outer corner-posts $(a_1, b_1), (a_2,b_2),$ $\ldots,(a_r,b_r)$ such that $a_1<a_2<\cdots < a_r$, then the outer corner-posts of $Z$ is given by
$$\{(a_0, b_1), (a_1, b_2), \ldots, (a_{r-1}, b_r), (a_r, b_0)\}.$$

We wish to apply this rule to each $Z_i$ for $1\leq i\leq \lambda_1$. However, in fact it is sufficient only to calculate the block diagonals $\{D(x) \mid x \in Z_i\}$ for each $Z_i$ for our purpose. To this end, we simplify the argument by exploiting the notion of block diagonals. 
\begin{defn}
For $(x,y) \in \bZ^2$, we define the $n$-block coordinate (or block coordinate if there is no confusion) of $(x,y)$ to be $\nbr{\ceil{\frac{x}{n}}-1,\ceil{\frac{y}{n}}-1}$ and write $\nco{(x,y)} = \nbr{\ceil{\frac{x}{n}}-1,\ceil{\frac{y}{n}}-1}$. For a subset $A \subset \bZ^2$, we define $\nco{A} \colonequals \{\nco{b} \mid b\in A\}$. Note that if $\nco{b} = \nbr{x,y}$ then $D(b) = y-x$.
\end{defn}
For $1\leq i \leq \lambda_1$, direct calculation shows that $\nco{\st(w)^{(-i)}} = \nbr{-1,-1}$ and also
\begin{align*}
\nco{\tilde{A}_{k,i}} = \nbr{k, \ceil{\frac{a_{k\lambda_1+i}}{n}}+k-1}, \qquad \nco{\tilde{B}_{k,i}} = \nbr{k,k}, \qquad \nco{\tilde{C}_{k,i}} =\nbr{k+1-\ceil{\frac{c_{k\lambda_1+i}}{n}}, k}.
\end{align*}
Therefore, $\nco{R_i'}$ and $\{D(x) \mid x\in R_i'\}$ for $1\leq i \leq \lambda_1$ is given as follows.

\begin{enumerate}[label=$(\alph*)$, leftmargin=*]
\item Suppose that $(N_i+1)\lambda_1+i\leq \vv{\mu}'^{(>0)}$. Then
\begin{gather*}
\nco{R_i'}=\left\{ \nbr{k, \ceil{\frac{a_{k\lambda_1+i}}{n}}+k-1} \mid 1\leq k\leq N_i+1\right\} \sqcup\left\{\nbr{k+1-\ceil{\frac{c_{k\lambda_1+i}}{n}}, k} \mid 1\leq k\leq N_i\right\},
\\\{D(x) \mid x\in R_i'\} = \left\{ \ceil{\frac{a_{k\lambda_1+i}}{n}}-1\mid 1\leq k\leq N_i+1\right\} \sqcup\left\{\ceil{\frac{c_{k\lambda_1+i}}{n}}-1\mid 1\leq k\leq N_i\right\}.
\end{gather*}
\item Suppose that $ \vv{\mu}'^{(>0)} < (N_i+1)\lambda_1+i \leq \vv{\mu}'^{(\geq0)}$. Then
\begin{align*}
\nco{R_i'}=& \ \left\{ \nbr{k, \ceil{\frac{a_{k\lambda_1+i}}{n}}+k-1} \mid 1\leq k\leq N_i\right\}\sqcup \{\nbr{N_i+1,N_i+1}\} ,=
\\&\quad \sqcup\left\{\nbr{k+1-\ceil{\frac{c_{k\lambda_1+i}}{n}}, k}\mid 1\leq k\leq N_i\right\},
\\\{D(x) \mid x\in R_i'\} =&\ \left\{ \ceil{\frac{a_{k\lambda_1+i}}{n}}-1\mid 1\leq k\leq N_i\right\} \sqcup\{0\}\sqcup\left\{\ceil{\frac{c_{k\lambda_1+i}}{n}}-1\mid 1\leq k\leq N_i\right\}.
\end{align*}
\item Suppose that $(N_i+1)\lambda_1+i > \vv{\mu}'^{(\geq0)}$. Then
\begin{gather*}
\nco{R_i'}=\left\{ \nbr{k, \ceil{\frac{a_{k\lambda_1+i}}{n}}+k-1} \mid 1\leq k\leq N_i\right\} \sqcup\left\{\nbr{k+1-\ceil{\frac{c_{k\lambda_1+i}}{n}}, k} \mid 1\leq k\leq N_i\right\},
\\\{D(x) \mid x\in R_i'\} = \left\{ \ceil{\frac{a_{k\lambda_1+i}}{n}}-1\mid 1\leq k\leq N_i\right\} \sqcup\left\{\ceil{\frac{c_{k\lambda_1+i}}{n}}-1\mid 1\leq k\leq N_i\right\}.
\end{gather*}
\end{enumerate}

Now we calculate $\nco{R_i}$ for each $1\leq i \leq \lambda_1$. Note that there are no two balls $b, b' \in R_i'$ such that $\nco{b}$ and $\nco{b'}$ have the same $x$ or $y$ coordinate. Thus, to this end we may apply the usual backward AMBC process on $\{\nco{(1-i,1-i)}\} \cup \nco{R_i}$ directly; it will give the same result as applying the backward AMBC process on $\{(1-i,1-i)\} \cup R_i$ and calculating the block coordinates of the result. Now direct calculation shows the following.

\begin{enumerate}[label=$(\alph*)$, leftmargin=*]
\item Suppose that $(N_i+1)\lambda_1+i\leq \vv{\mu}'^{(>0)}$. Then
\begin{gather*}
\begin{aligned}
\nco{R_i}=& \ \left\{ \nbr{k-1, \ceil{\frac{a_{k\lambda_1+i}}{n}}+k-1} \mid 1\leq k\leq N_i+1\right\}\sqcup\{\nbr{N_i+1,N_i}\}
\\&\quad \sqcup\left\{\nbr{k+1-\ceil{\frac{c_{k\lambda_1+i}}{n}}, k-1} \mid 1\leq k\leq N_i\right\},
\end{aligned}
\\\{D(x) \mid x\in R_i\} = \left\{ \ceil{\frac{a_{k\lambda_1+i}}{n}}\mid 1\leq k\leq N_i+1\right\} \sqcup\{-1\}\sqcup\left\{\ceil{\frac{c_{k\lambda_1+i}}{n}}-2\mid 1\leq k\leq N_i\right\}.
\end{gather*}
\item Suppose that $ \vv{\mu}'^{(>0)} < (N_i+1)\lambda_1+i \leq \vv{\mu}'^{(\geq0)}$. Then
\begin{align*}
\nco{R_i}=& \ \left\{ \nbr{k-1, \ceil{\frac{a_{k\lambda_1+i}}{n}}+k-1} \mid 1\leq k\leq N_i\right\}\sqcup \{\nbr{N_i,N_i+1},\nbr{N_i+1,N_i}\} 
\\&\quad \sqcup\left\{\nbr{k+1-\ceil{\frac{c_{k\lambda_1+i}}{n}}, k-1}\mid 1\leq k\leq N_i\right\},
\\\{D(x) \mid x\in R_i\} =&\ \left\{ \ceil{\frac{a_{k\lambda_1+i}}{n}}\mid 1\leq k\leq N_i\right\} \sqcup\{1,-1\}\sqcup\left\{\ceil{\frac{c_{k\lambda_1+i}}{n}}-2\mid 1\leq k\leq N_i\right\}.
\end{align*}
\item Suppose that $(N_i+1)\lambda_1+i > \vv{\mu}'^{(\geq0)}$. Then
\begin{gather*}
\begin{aligned}
\nco{R_i}=& \left\{ \nbr{k, \ceil{\frac{a_{k\lambda_1+i}}{n}}+k-1} \mid 1\leq k\leq N_i\right\} \sqcup\{\nbr{N_i,N_i}\}
\\&\quad \sqcup\left\{\nbr{k+1-\ceil{\frac{c_{k\lambda_1+i}}{n}}, k} \mid 1\leq k\leq N_i\right\},
\end{aligned}
\\\{D(x) \mid x\in R_i\} = \left\{ \ceil{\frac{a_{k\lambda_1+i}}{n}}\mid 1\leq k\leq N_i\right\} \sqcup\{0\}\sqcup\left\{\ceil{\frac{c_{k\lambda_1+i}}{n}}-2\mid 1\leq k\leq N_i\right\}.
\end{gather*}
\end{enumerate}

\begin{example}
Recall the example above when $n=15$, $\lambda=(5,4,4,2)$, $\vv{\rho}' = (0,4,0)$, and 
$$w' = \Psi(\Tc_{(4,4,2)},\Tc_{(4,4,2)},(0,3,0)) = [-21, -20, -8, -7, 5, 6, 32, 33, 34, 46,\emptyset,\emptyset,\emptyset,\emptyset,\emptyset].$$
We have
\begin{alignat*}{3}
\{\st(w)^{(-1)}\}\cup R_1' &= \{(0,0), \tilde{A}_{0,1}, \tilde{B}_{1,1}, \tilde{C}_{0,1}\} &&= \{(0,0),(10,46), (20,20), (32,10)\},
\\\{\st(w)^{(-2)}\}\cup R_2' &=\{(-1,-1), \tilde{A}_{0,2},  \tilde{C}_{0,2}\} &&= \{(-1,-1), (9,34), (31,9)\},
\\\{\st(w)^{(-3)}\}\cup R_3' &=\{(-1,-1), \tilde{A}_{0,3},  \tilde{C}_{0,3}\} &&= \{(-2,-2), (8,33), (19,8)\},
\\\{\st(w)^{(-4)}\}\cup R_4' &=\{(-1,-1), \tilde{A}_{0,4},  \tilde{C}_{0,4}\} &&= \{(-3,-3), (7,32), (18,7)\},
\\\{\st(w)^{(-5)}\}\cup R_5' &=\{(-1,-1), \tilde{B}_{0,5}\} &&= \{(-4,-4),(6,6)\}.
\end{alignat*}
Therefore
\begin{alignat*}{3}
\nco{\{\st(w)^{(-1)}\}\cup R_1'} &= \{(-1,-1),(0,3), (1,1), (2,0)\},
\\\nco{\{\st(w)^{(-2)}\}\cup R_2'} &= \{(-1,-1), (0,2), (2,0)\},
\\\nco{\{\st(w)^{(-3)}\}\cup R_3'} &= \{(-1,-1), (0,2), (1,0)\},
\\\nco{\{\st(w)^{(-4)}\}\cup R_4'} &= \{(-1,-1), (0,2), (1,0)\},
\\\nco{\{\st(w)^{(-5)}\}\cup R_5'} &= \{(-1,-1),(0,0)\}.
\end{alignat*}
By applying the backward AMBC algorithm to each $\nco{R_i'}$, we get
\begin{alignat*}{3}
\nco{R_1} &= \{(-1,3),(0,1), (1,0), (2,-1)\},
\\\nco{R_2} &= \{(-1,2), (0,0), (2,-1)\},
\\\nco{R_3} &= \{(-1,2), (0,0), (1,-1)\},
\\\nco{R_4} &= \{(-1,2), (0,0), (1,-1)\},
\\\nco{R_5} &= \{(-1,0),(0,-1)\}.
\end{alignat*}
Note that this is consistent with
\begin{alignat*}{3}
R_1 &= \{(0,46),(10,20), (20,10), (32,0)\},
\\R_2 &= \{(-1,34), (9,9), (31,-1)\},
\\R_3 &= \{(-2,33), (8,8), (19,-2)\},
\\R_4 &= \{(-3,32), (7,7), (18,-3)\},
\\R_5 &= \{(-4,6),(6,-4)\}.
\end{alignat*}
Furthermore, we have
\begin{alignat*}{3}
\{D(x)\mid x\in R_1'\} &= \{3, 0, -2\}, \qquad&\{D(x)\mid x\in R_2'\} &= \{2,-2\},
\\\{D(x)\mid x\in R_3'\}&=\{D(x)\mid x\in R_4'\} = \{2, -1\}, \qquad &\{D(x)\mid x\in R_5'\}&= \{0\}.
\end{alignat*}
and 
\begin{alignat*}{3}
\{D(x)\mid x\in R_1'\} &= \{4, 1, -1,-3\}, \qquad&\{D(x)\mid x\in R_2'\} &= \{3,0,-3\},
\\\{D(x)\mid x\in R_3'\}&=\{D(x)\mid x\in R_4'\} = \{3, -0,-2\}, \qquad &\{D(x)\mid x\in R_5'\}&= \{1,-1\}.
\end{alignat*}
\end{example}

From this calculation, we derive the relations between $\vv{\mu}$ and $\vv{\mu}'$. (Recall that $N_i=\floor{\frac{\vv{\mu}'^{(<0)}-i}{\lambda_1}}$.)
\begin{enumerate}[label=$\bullet$]
\item For $j>0$, $ \vv{\mu}^{(j+1)}=\vv{\mu}'^{(j)}$ and $\vv{\mu}^{(-j-1)}=\vv{\mu}'^{(-j)}$.
\item $\vv{\mu}^{(0)}= \#\{i \mid 1\leq i \leq \lambda_1, (N_i+1)\lambda_1+i > \vv{\mu}'^{(\geq0)}\}$
\item $\vv{\mu}^{(1)}= \#\{i \mid 1\leq i \leq \lambda_1, \vv{\mu}'^{(>0)} < (N_i+1)\lambda_1+i \leq \vv{\mu}'^{(\geq0)}\}$
\item $\vv{\mu}^{(-1)}= \#\{i \mid 1\leq i \leq \lambda_1, (N_i+1)\lambda_1+i \leq \vv{\mu}'^{(\geq0)}\}$
\end{enumerate}
Now it is clear from the description that $\vv{\mu}^{(1)} = \vv{\mu}'^{(\geq0)}-\vv{\mu}'^{(>0)} =\vv{\mu}'^{(0)}$. Also by Lemma \ref{lem:mu},
\begin{align*}
\vv{\mu}^{(-1)}&=\vv{\mu}'^{(\geq0)}-\vv{\mu}'^{(<0)}=\vv{\mu}_{\tilde{\lambda}}(0)^{(\geq0)}-\vv{\mu}_{\tilde{\lambda}}(0)^{(<0)} = \vv{\mu}_{\tilde{\lambda}}(0)^{(0)}=\sum_{k\geq 1} (-1)^{k-1} \lambda_k.
\end{align*}
Since $\sum_{j \in \bZ}\vv{\mu}^{(j)} = \lambda_1+\sum_{j \in \bZ}\vv{\mu}'^{(j)}$, it follows that $\vv{\mu}^{(0)}=\sum_{k\geq 2} (-1)^k \lambda_k$.

Now suppose that $\vv{\mu}'=\vv{\mu}_{\tilde{\lambda}}(0)$. Then it is easy to see that $\vv{\mu}_{\lambda}(0)^{(j+1)}=\vv{\mu}_{\tilde{\lambda}}(0)^{(j)}, \vv{\mu}_{\lambda}(0)^{(-j-1)}=\vv{\mu}_{\tilde{\lambda}}(0)^{(-j)}$ for $j\geq 0$ and $\vv{\mu}_{\lambda}(0)^{(0)}=\#\{ i \mid 1\leq i \leq \lambda_1, \lambda_i'\in 2\bZ\}=\sum_{k\geq 2} (-1)^k \lambda_k$. By comparing this with the relations between $\vv{\mu}$ and $\vv{\mu}'$, we conclude that $\vv{\mu}=\vv{\mu}_{\lambda}(0)$.

This time suppose that  $\vv{\mu}'=\vv{\mu}_{\tilde{\lambda}}(r,s)$ where $r \in \tilde{\lambda}$. Then by Lemma \ref{lem:mu} and the argument above it is still true that $\vv{\mu}_{\lambda}(s,r)^{(-j-1)}=\vv{\mu}_{\tilde{\lambda}}(s,r)^{(-j)}$ for $j>0$ and also 
\begin{gather*}
\vv{\mu}_{\lambda}(s,r)^{(-1)}=\vv{\mu}_{\lambda}(0)^{(-1)}=\#\{ i \mid 1\leq i \leq \lambda_1, \lambda_i'-1\in 2\bZ\}=\sum_{k\geq 1} (-1)^{k-1} \lambda_k
\\\vv{\mu}_{\lambda}(s,r)^{(0)}=\vv{\mu}_{\lambda}(0)^{(0)}=\#\{ i \mid 1\leq i \leq \lambda_1, \lambda_i'\in 2\bZ\}=\sum_{k\geq 2} (-1)^{k} \lambda_k.
\end{gather*}
(The second equation holds since $r\leq \lambda_2$.) Also from the construction of $\mathfrak{W}_{\tilde{\lambda}}(r,s)$ and $\mathfrak{W}_{\lambda}(r,s)$, it is also easy to show that  $\vv{\mu}_{\lambda}(s,r)^{(j+1)}=\vv{\mu}_{\tilde{\lambda}}(s,r)^{(j)}$ for $j>0$. Now it follows from $\sum_{j \in \bZ}\vv{\mu}_\lambda(r,s)^{(j)} = \lambda_1+\sum_{j \in \bZ}\vv{\mu}_{\tilde{\lambda}}^{(j)}$ that $\vv{\mu}_{\lambda}(s,r)^{(1)}=\vv{\mu}_{\tilde{\lambda}}(s,r)^{(0)}$. But this also shows that $\vv{\mu}=\vv{\mu}_\lambda(s,r)$ by considering the relations between $\vv{\mu}$ and $\vv{\mu}'$.

This completes the proof of Theorem \ref{thm:lvbij}.


\appendix
\section{Asymptotic Hecke algebras for $SL_n$ and $PGL_n$}
So far, we discussed the extended affine symmetric group and its asymptotic Hecke algebra. In terms of representation theory, it means that we only considered the affine Weyl group of $GL_n$. In this section, we observe how our result can be extended to simple groups of type $A$, especially $SL_n$ and $PGL_n$. For simplicity, we assume that $GL_n$, $SL_n$, and $PGL_n$ are all defined over $\bC$.

\subsection{$G=SL_n$ case} Note that $\shift^n \in \extSn$ is in the center of $\extSn$. The affine Weyl group of $SL_n$ can be identified with $\extSn/\br{\shift^n}$, and its two-sided cell (resp. left cell, resp. right cell) is given by the image of such a cell of $\extSn$ under the quotient map $\pi: \extSn \twoheadrightarrow \extSn/\br{\shift^n}$. In particular, its two-sided cells are also parametrized by the partitions of $n$.

Likewise, for a two-sided cell $\tsc$ of $\extSn$ we define
$\cJ_\tsc^{SL} \colonequals \cJ_{\tsc}/(\shift^n-1)$. Then it is clear that this is an asymptotic Hecke algebra corresponding to $SL_n$ attached to the two-sided cell $\pi(\tsc)$. Now the following theorem gives a structural description of $\cJ_\tsc^{SL}$ in terms of Theorem \ref{thm:matisom}.
\begin{thm} For $\tsc=\tsc_\lambda$, the isomorphism in Theorem \ref{thm:matisom} factors through $$\cJ_\tsc^{SL}  \simeq \Mat_{\chi\times\chi}(\Rep(F_\lambda))/(\shift^n-1) \simeq \Mat_{\chi\times\chi}(\Rep(F_\lambda)/\br{V(\lambda)})$$
where both $\cJ_\tsc$ and $\Mat_{\chi\times\chi}(\Rep(F_\lambda))$ are regarded as $\bZ[\shift^n]$-algebras and $V(\lambda)$ is the irreducible representation of $F_\lambda$ of highest weight $\lambda=(\lambda_1, \lambda_2, \ldots) \in \Dom(F_\lambda)$. In particular, $\cJ_\tsc^{SL}$ is a matrix algebra.
\end{thm}
\begin{proof}
The first isomorphism is obvious. The second isomorphism follows from the fact that if $\Phi(w) = (P, Q, \vv{\rho})$ then $\Phi(\shift^n) =(P, Q, \vv{\rho}+\lambda)$ which follows from Proposition \ref{prop:shift}.
\end{proof}
\begin{rmk}
A similar statement can also be found in \cite[Section 8.4]{xi02}.
\end{rmk}

\subsection{$G=PGL_n$ case} \label{sec:pgln} Consider the case when $G=PGL_n$. Then its affine Weyl group is
$$\affSn \colonequals \{ w \in \extSn \mid \sum_{i=1}^n w(i) = n(n+1)/2\},$$
i.e. the (non-extended) affine symmetric group. For an integer sequence $(a_1, a_2, \ldots)$, let us define $|(a_1, a_2, \ldots)|\colonequals \sum_i a_i$. Then by \cite[Theorem 10.3]{cpy18}, we have
$$\Phi(\affSn)=\{(P, Q, \vv{\rho}) \in \Odom \mid |\vv{\rho}|=0\}$$
Also its two-sided cell (resp. left cell, resp. right cell) is obtained from the intersection of such a cell in $\extSn$ with $\affSn$. In particular, the two-sided cells of $\affSn$ are also parametrized by the partitions of $n$.

Let $\cJ_\tsc^{PGL} \subset \cJ_{\tsc}$ be the asymptotic Hecke algebra corresponding to $PGL_n$ attached to $\tsc \cap \affSn$. In general, $\cJ_\tsc^{PGL}$ is no longer a matrix algebra; a counterexample is given in \cite[Section 8.3]{xi02}. Here, we discuss a sufficient condition when $\cJ_\tsc^{PGL}$ is indeed isomorphic to a matrix algebra.
\begin{thm}  Suppose that $\tsc=\tsc_\lambda$. If $\gcd(\lambda_1', \lambda_2', \ldots)=1$, then there exists a ring isomorphism $\cJ_\tsc^{PGL} \simeq \Mat_{\chi\times \chi}(\Rep(\tilde{F}_\lambda))$ where $\tilde{F}_\lambda = F_\lambda/\{cI\mid c\in \bC\}$.
\end{thm}
\begin{proof} As $\gcd(\lambda_1', \lambda_2', \ldots)=1$, by \cite[Theorem 8.6]{clp17} any tabloid $T$ of shape $\lambda$ can be obtained from $\Ta_\lambda$ by successive star operations (without applying $\shift$). Now for each $T \in \RSYT(\lambda)$, we fix a series of such star operations. Then it induces a bijection $\phi_T: \lca_\lambda \rightarrow \lc_T$ defined by the composition of right star operations. Here, $\lc_T$ is the left cell parametrized by $T$, i.e. $w\in \lc_T$ if and only if $Q(w) = T$. Also note that star operations stabilizes $\affSn \subset \extSn$, which implies that it restricts to $\phi_T : \lca_\lambda \cap \affSn \rightarrow \lc_T \cap \affSn$.
	
Now for any tabloid $T$ of shape $\lambda$, we define $w_T\colonequals \phi_T(w_0^\lambda)$ where $w_0^\lambda$ is as in Section \ref{sec:cancells}. Then it follows that $\Phi(w_T)=(\Ta_\lambda, T, \vv{\rho}_T+\offset_{\Ta_\lambda, T})$ for some $\vv{\rho}_T = (\rho_1, \rho_2, \ldots)$ such that $|\vv{\rho}_T|+|\offset_{\Ta_\lambda, T}|=0$. By similar argument to Lemma \ref{lem:determ}, we see that $\vv{\rho}_T$ is determinantal and $\ft_{w_T}=\ft(\Ta, \Ta, \vv{\rho}_T)\ft_{\Ta,T}$. Note that $\ft(\Ta, \Ta, \vv{\rho}_T)$ and $\ft_{\Ta,T}$ are not in general contained in $\cJ_\tsc^{PGL}$, but $\ft_{w_T}$ is an element of $\cJ_\tsc^{PGL}$.

For $P, Q \in \RSYT(\lambda)$, let us define $\ft'_{P, Q}\colonequals (\ft_{w_P})^{-1}\ft_{w_Q} \in \cJ_\tsc^{PGL}$. Then it is clear that $\ft_{P,Q}' = \ft(P, P, \vv{\rho}_Q-\vv{\rho}_P)\ft_{P,Q}$. As $\vv{\rho}_Q-\vv{\rho}_P$ is determinantal, $\{\ft'_{P,Q} \mid P, Q \in \RSYT(\lambda)\}$ also gives a matrix basis of $\cJ_\tsc$, i.e. we have an algebra isomorphism $\Upsilon': \cJ_\tsc \rightarrow \Mat_{\chi\times\chi}(\Rep(F_\lambda))$ where $\Upsilon'(\ft'_{P,Q})$ is an elementary matrix for any $P, Q \in \RSYT(\lambda)$. (In general $\Upsilon'$ is different from $\Upsilon$ defined in Theorem \ref{thm:matisom} since $\vv{\rho}_Q-\vv{\rho}_P$ needs not be zero.)

It remains to show that $\Upsilon'(\cJ_\tsc^{PGL}) =  \Mat_{\chi\times \chi}(\Rep(\tilde{F}_\lambda))$; here we identify $\Rep(\tilde{F}_\lambda)$ with the subring of $\Rep(F_\lambda)$ generated by $V(\vv{\mu})$ such that $|\vv{\mu}|=0$. Recall that for any $w \in \affSn$ we have $\Phi(w) = (P, Q, \vv{\rho})$ such that $|\vv{\rho}| = 0$. In particular, 
\begin{align*}
\ft_w &= \ft(P, Q, \vv{\rho}) = \ft(P, P, \vv{\rho}-\offset_{P, Q})\ft_{P,Q}
\\&= \ft(P, P, \vv{\rho}-\offset_{P, Q}+\vv{\rho}_P-\vv{\rho}_Q) \ft'_{P,Q}
\end{align*}
where
$$|\vv{\rho}-\offset_{P, Q}+\vv{\rho}_P-\vv{\rho}_Q| = |\vv{\rho}_P|-|\vv{\rho}_Q|-|\offset_{P, Q}|=-|\offset_{\Ta_\lambda, P}|+|\offset_{\Ta_\lambda, Q}|-|\offset_{P, Q}|=0$$
since $\offset_{\Ta_\lambda, Q}-\offset_{\Ta_\lambda, P}=\offset_{P, Q}$. As $\Upsilon'(\ft_w)$ is a matrix whose only nonzero entry corresponds to $V(\vv{\rho}-\offset_{P, Q}+\vv{\rho}_P-\vv{\rho}_Q) \in \Rep(\tilde{F}_\lambda)$, we conclude that $\Upsilon'(\cJ_\tsc^{PGL}) \subset \Mat_{\chi\times \chi}(\Rep(\tilde{F}_\lambda))$. The other inclusion also follows in almost the same manner.
\end{proof}

%

%
%
%
%


\bibliographystyle{amsalpha}
\bibliography{ambc}

\providecommand{\bysame}{\leavevmode\hbox to3em{\hrulefill}\thinspace}
\providecommand{\MR}{\relax\ifhmode\unskip\space\fi MR }
\providecommand{\MRhref}[2]{%
  \href{http://www.ams.org/mathscinet-getitem?mr=#1}{#2}
}
\providecommand{\href}[2]{#2}
\begin{thebibliography}{Rus17b}

\bibitem[Ach01]{ach01}
Pramod~N. Achar, \emph{Equivariant coherent sheaves on the nilpotent cone for
  complex reductive {L}ie groups}, Ph.D. thesis, Massachusetts Institute of
  Technology, 2001.

\bibitem[Ach04]{ach04}
\bysame, \emph{On the equivariant ${K}$-theory of the nilpotent cone in the
  general linear group}, Represent. Theory \textbf{8} (2004), 180--211.

\bibitem[Bez03]{bez03}
Roman Bezrukavnikov, \emph{Quasi-exceptional sets and equivariant coherent
  sheaves on the nilpotent cone}, Represent. Theory \textbf{7} (2003), 1--18.

\bibitem[Bez04]{bez04}
\bysame, \emph{On tensor categories attached to cells in affine {W}eyl groups},
  Adv. Stud. Pure Math. \textbf{40} (2004), 69--90.

\bibitem[Bez09]{bez09}
\bysame, \emph{Perverse sheaves on affine flags and nilpotent cone of the
  {L}anglands dual group}, Israel J. Math. \textbf{170} (2009), 185--206.

\bibitem[CG97]{cg97}
Neil Chriss and Victor Ginzburg, \emph{Representation {T}heory and {C}omplex
  {G}eometry}, Modern Birkh\"auser Classics, vol.~2, Birkh\"auser Basel, 1997.

\bibitem[CLP17]{clp17}
Michael Chmutov, Joel~Brewster Lewis, and Pavlo Pylyavskyy, \emph{Monodromy in
  {K}azhdan-{L}usztig cells in affine type ${A}$}, Available at
  \url{https://arxiv.org/abs/1706.00471} (2017).

\bibitem[CM93]{col93nilp}
David~H. Collingwood and William~M. McGovern, \emph{Nilpotent {O}rbits in
  {S}emisimple {L}ie algebra: an {I}ntroduction}, CRC Press, 1993.

\bibitem[CPY18]{cpy18}
Michael Chmutov, Pavlo Pylyavskyy, and Elena Yudovina, \emph{Matrix-ball
  construction of affine {R}obinson-{S}chensted correspondence}, Selecta Math.
  (N. S.) \textbf{24} (2018), no.~2, 667--750.

\bibitem[Gar90]{gar90}
Devra Garfinkle, \emph{On the classification of primitive ideals for complex
  classical {L}ie algebras, {I}}, Compos. Math. \textbf{75} (1990), 135--169.

\bibitem[Gar92]{gar92}
\bysame, \emph{On the classification of primitive ideals for complex classical
  {L}ie algebras, {I}{I}}, Compos. Math. \textbf{81} (1992), 307--336.

\bibitem[Gar93]{gar93}
\bysame, \emph{On the classification of primitive ideals for complex classical
  {L}ie algebras, {I}{I}{I}}, Compos. Math. \textbf{88} (1993), 187--234.

\bibitem[HT12]{ht12}
Anthony Henderson and Peter~E. Trapa, \emph{The exotic {R}obinson-{S}chensted
  correspondence}, J. Algebra \textbf{370} (2012), 32--45.

\bibitem[Hum91]{hum91}
James~E. Humphreys, \emph{Representations of {S}emisimple {L}ie {A}lgebras in
  the {B}{G}{G} {C}ategory $\mathcal{O}$}, Graduate Studies in Mathematics,
  vol.~94, American {M}athematical {S}ociety, 1991.

\bibitem[Knu70]{knu70}
Donald~E. Knuth, \emph{Permutations, matrices, and generalized {Y}oung
  tableaux}, Pacific J. Math. \textbf{34} (1970), no.~3, 709--727.

\bibitem[LP07]{lp07}
Cristian Lenart and Alexander Postnikov, \emph{Affine {W}eyl groups in
  ${K}$-theory and representation theory}, Int. Math. Res. Not. IMRN (2007).

\bibitem[Lus87a]{lus87:cell}
George Lusztig, \emph{Cells in affine {W}eyl groups, {I}{I}}, J. Algebra
  \textbf{109} (1987), 536--548.

\bibitem[Lus87b]{lus87:leading}
\bysame, \emph{Leading coefficients of character valeus of {H}ecke algebras},
  Proc. Sympos. Pure Math. \textbf{47} (1987), no.~235--262.

\bibitem[Lus89]{lus89:cell}
\bysame, \emph{Cells in affine {W}eyl groups, {I}{V}}, J. Fac. Sci. Univ. Tokyo
  \textbf{36} (1989), 297--328.

\bibitem[Lus14]{lus14:hecke}
\bysame, \emph{Hecke {A}lgebras with {U}nequal {P}arameters}, revised version
  ed., Available at \url{https://arxiv.org/abs/math/0208154}, 2014.

\bibitem[LX88]{lx88}
George Lusztig and Nanhua Xi, \emph{Canonical left cells in affine {W}eyl
  groups}, Adv. Math. \textbf{72} (1988), 284--288.

\bibitem[Ost00]{ost00}
Viktor Ostrik, \emph{On the equivariant ${K}$-theory of the nipotent cone},
  Represent. Theory \textbf{4} (2000), 296--305.

\bibitem[Rus17a]{rus17:thesis}
David Rush, \emph{Computing the {L}usztig-{V}ogan bijection}, Ph.D. thesis,
  Massachusetts Institute of Technology, 2017.

\bibitem[Rus17b]{rus17}
\bysame, \emph{Computing the {L}usztig-{V}ogan bijection}, Available at
  \url{https://arxiv.org/abs/1711.00148}.

\bibitem[Shi86]{shi86}
Jian-Yi Shi, \emph{The {K}azhdan-{L}usztig {C}ells in {C}ertain {A}ffine {W}eyl
  {G}roups}, Lecture {N}otes in {M}ath., vol. 1179, Springer-{V}erlag, 1986.

\bibitem[Shi91]{shi91}
\bysame, \emph{The generalized {R}obinson-{S}chensted algorithm on the affine
  {W}eyl group of type $\tilde{{A}}_{n-1}$}, J. Algebra \textbf{139} (1991),
  364--394.

\bibitem[Sta86]{sta86}
Richard~P. Stanley, \emph{Enumerative {C}ombinatorics}, vol.~2, Cambridge
  {S}tudies in {A}dvanced {M}athematics, no.~62, Cambridge {U}niversity
  {P}ress, 1986.

\bibitem[Vog91]{vog91}
David~A. Vogan, \emph{Associated varieties and unipotent representations},
  Harmonic {A}nalysis on {R}eductive {G}roups, Progress in Mathematics, vol.
  101, Birkh{\"a}user {B}oston, 1991, pp.~315--388.

\bibitem[Wei13]{wei13}
Charles~A. Weibel, \emph{The ${K}$-book: {A}n {I}ntroduction to {A}lgebraic
  ${K}$-theory}, Graduate Studies in Mathematics, vol. 145, American
  {M}athematical {S}ociety, Providence, Rhode Island, 2013.

\bibitem[Xi02]{xi02}
Nanhua Xi, \emph{The based ring of two-sided cells of affine {W}eyl groups of
  type $\tilde{{A}}_{n-1}$}, Mem. Amer. Math. Soc. \textbf{157} (2002),
  no.~749.

\bibitem[Yun16]{yun16}
Zhiwei Yun, \emph{Lectures on {S}pringer theories and orbital integrals},
  Available at \url{https://arxiv.org/abs/1602.01451} (2016).

\end{thebibliography}

\end{document}